\documentclass[12pt]{article}

\usepackage[margin=2cm]{geometry}
\usepackage{amscd}
\usepackage{amsmath}
\usepackage{amssymb}
\usepackage{amsthm}
\usepackage{mathrsfs}
\usepackage{lineno}
\usepackage{fancyhdr}
\usepackage{verbatim}
\usepackage{tikz}
\usetikzlibrary{snakes}
\usepackage[all]{xy}
\usepackage{color}
\usepackage[colorlinks=true, linkcolor=blue, citecolor=blue, pagebackref=true]{hyperref}
\usepackage[pagebackref=true]{hyperref}

\newtheorem{theorem}{Theorem}[section]
\newtheorem{lemma}[theorem]{Lemma}
\newtheorem{proposition}[theorem]{Proposition}

\theoremstyle{plain}

\theoremstyle{definition} 
\newtheorem{definition}[theorem]{Definition}

\theoremstyle{definition} 
\newtheorem*{definition*}{Definition}

\theoremstyle{example} 
\newtheorem*{example*}{Example}

\theoremstyle{theorem}

\theoremstyle{theorem} 
\newtheorem*{conjecture*}{Conjecture}

\theoremstyle{theorem}

\theoremstyle{definition}

\theoremstyle{definition} 
\newtheorem*{assumption*}{Assumption}

\theoremstyle{remark} 
\newtheorem{remark}[theorem]{Remark}

\theoremstyle{remark} 
\newtheorem*{remark*}{Remark}

\newcommand{\bone}{\mathbf{1}}

\newcommand{\Hyp}{\mathbb{H}}
\newcommand{\NN}{\mathbb{N}}
\newcommand{\en}{N}
\newcommand{\RR}{\mathbb{R}}

\newcommand{\ZZ}{\mathbb{Z}}
\newcommand{\LLL}{\mathcal{L}}
\newcommand{\calC}{\mathcal{C}}
\newcommand{\D}{\mathcal{D}}
\newcommand{\E}{\mathcal{E}}
\newcommand{\calP}{\mathcal{P}}

\newcommand{\U}{\mathcal{U}}

\newcommand{\Z}{\mathcal{Z}}

\newcommand{\di}{\mathrm{d}}

\newcommand{\bk}{\boldsymbol{k}}

\newcommand{\bm}{\boldsymbol{m}}
\newcommand{\bn}{\boldsymbol{n}}
\newcommand{\bw}{\boldsymbol{w}}

\newcommand{\bz}{\boldsymbol{z}}

\newcommand{\brho}{\boldsymbol{\rho}}

\newcommand{\bZ}{\boldsymbol{\mathcal{Z}}}
\newcommand{\g}{\mathfrak{g}}

\newcommand{\Sl}{\mathfrak{sl}}
\newcommand{\SL}{\mathrm{SL}}
\newcommand{\SO}{\mathrm{SO}}
\newcommand{\so}{\mathfrak{so}}
\newcommand{\PSL}{\mathrm{PSL}}
\newcommand{\Isom}{\mathrm{Isom}}
\newcommand{\Cinf}{C^{\infty}}
\newcommand{\Lii}{L^{2}}

\newcommand{\HH}{\mathcal{H}}
\newcommand{\II}{\mathcal{I}}

\providecommand{\floor}[1]{\lfloor#1\rfloor}
\providecommand{\ceil}[1]{\lceil#1\rceil}
\providecommand{\abs}[1]{\lvert#1\rvert}
\providecommand{\Abs}[1]{\left|#1\right|}
\providecommand{\norm}[1]{\lVert#1\rVert}
\providecommand{\Norm}[1]{\left\|#1\right\|}
\providecommand{\inner}[2]{\langle#1,#2\rangle}

\newcommand{\aside}[1]{\vspace{5 mm}\par \noindent
\marginpar{\textsc{aside}}
\framebox{\begin{minipage}[c]{0.95 \textwidth}
\tt #1 \end{minipage}}\vspace{5 mm}\par}
\def\a{{\alpha}}
\def\b{{\beta}}

\def\s{{\sigma}}
\def\vars{{\varsigma}}

\def\t{{\tau}}

\def\k{{\kappa}}
\def\l{{\lambda}}
\def\L{{\Lambda}}
\def\o{{\omega}}
\def\O{{\Omega}}
\def\G{{\Gamma}}

\title{\textbf{Invariant distributions and cohomology for geodesic flows and higher cohomology of higher-rank Anosov actions}}
\author{Felipe A.~Ram{\'i}rez\footnote{email: \texttt{f.a.ramirez@bristol.ac.uk}} \\ University of Bristol}
\date{}

\begin{document}

\maketitle
\thispagestyle{empty}

%===============================================

%\begin{abstract}
%We are motivated by a conjecture of A.~and S.~Katok to study the higher-degree smooth cohomologies of a family of Weyl chamber flows. The conjecture in question is a natural generalization of the Liv{\v{s}}ic Theorem to Anosov actions by higher-rank abelian groups; it involves a description of top-degree cohomology and a vanishing statement for lower-degree cohomologies.  Our main result, proved in Part~\ref{partii} of this article, verifies the conjecture in lower degrees for the systems we consider, and steps in the ``correct'' direction in the top degree. In Part~\ref{parti} we study our ``base case'': geodesic flows of finite-volume hyperbolic manifolds. We describe the obstructions (invariant distributions) to solving the coboundary equation in irreducible unitary representations of $\SO^\circ (\en,1)$, the group of orientation-preserving isometries of hyperbolic $\en$-space, and we study the Sobolev regularity of solutions. (As a byproduct, we find a smooth version of the Liv{\v{s}}ic Theorem for geodesic flows of hyperbolic manifolds with cusps.) The results in Part~\ref{parti} provide the tools needed in Part~\ref{partii} for the main theorem.
%\end{abstract}

\begin{abstract}
We are motivated by a conjecture of A.~and S.~Katok to study the smooth cohomologies of a family of Weyl chamber flows. The conjecture is a natural generalization of the Liv{\v{s}}ic Theorem to Anosov actions by higher-rank abelian groups; it involves a description of top-degree cohomology and a vanishing statement for lower degrees. Our main result, proved in Part~\ref{partii}, verifies the conjecture in lower degrees for our systems, and steps in the ``correct'' direction in top degree. In Part~\ref{parti} we study our ``base case'': geodesic flows of finite-volume hyperbolic manifolds. We describe obstructions (invariant distributions) to solving the coboundary equation in unitary representations of the group of orientation-preserving isometries of hyperbolic $N$-space, and we study Sobolev regularity of solutions. (One byproduct is a smooth Liv{\v{s}}ic Theorem for geodesic flows of hyperbolic manifolds with cusps.) Part~\ref{parti} provides the tools needed in Part~\ref{partii} for the main theorem.
\end{abstract}
%===============================================

{\footnotesize\tableofcontents}

\newpage

\subsection{Introduction}

This article has two parts.  In the first, we study the coboundary equation~\eqref{cobeq} for geodesic flows of finite-volume hyperbolic manifolds.  In the second, we study higher-degree cohomology of Anosov $\RR^d$-actions on the unit tangent bundles of irreducible finite-volume quotients of $\Hyp^{\en_1}\times\dots\times\Hyp^{\en_d}$, where $\Hyp^{\en}$ denotes hyperbolic $\en$-space.

\paragraph{On the Liv\v{s}ic Theorem \emph{vs.}~higher-rank phenomena.}

The cohomology of group actions has played an essential role in dynamics since the 1970s.  One of the most important results is A.~N.~Liv\v{s}ic's theorem for Anosov flows and diffeomorphisms~\cite{Livsic}.  For an Anosov flow $\phi^t$ on a compact manifold $\mathcal{M}$, the Liv\v{s}ic Theorem states that given a smooth function $f$ on $\mathcal{M}$, the coboundary equation 
\begin{equation}\label{cobeq}
	\frac{d}{dt}\left(g\circ\phi^t\right)\Bigr\vert_{t=0} = f
\end{equation}
can be solved with a continuous function $g$ if (and only if) the integral of $f$ around every periodic $\phi^t$-orbit is zero.  This was later strengthened to give regular (smooth or analytic) solutions when the given function is regular~\cite{GK80, dlLMM}.  The Liv\v{s}ic Theorem has been generalized in several directions, for example, to cocycles taking values in matrix groups~\cite{Kal11}, and to the setting of partially hyperbolic diffeomorphisms~\cite{Wil08}.  

In contrast to the Liv{\v{s}}ic Theorem, which names the obstructions to solving the (degree-one) coboundary equation for Anosov flows, one finds that for \emph{higher-rank} actions there are \emph{no} obstructions:~A.~Katok and R.~Spatzier showed that the degree-one cohomology of Anosov $\RR^d$-actions \emph{trivializes} when $d \geq 2$, \emph{i.e.} every smooth $1$-cocycle is smoothly cohomologous to one defined by constant functions~\cite{KS94first}. This dichotomy is a recurring theme in differentiable dynamics. We commonly find ``rigidity'' for higher-rank actions where we would not have found it in rank one. For example, trivialization of first cohomology seems to be widespread among higher-rank actions, not just Anosov ones~\cite{M2,M1,Ram09}, but is apparently all but non-existent for rank-one actions. (Conjecturally, the only rank-one ``cohomology-free'' systems are Diophantine toral rotations~\cite{KR}.)

\paragraph{On higher-degree cohomology and the related Katok--Katok Conjecture.}
The apparent dichotomy between rank-one and higher-rank cohomological phenomena can be reconciled by introducing the notion of higher-degree cohomology over a higher-rank action, and noting that for flows the first cohomology is the top-degree cohomology, whereas for higher-rank actions it is one of the lower degrees. With this in mind, it is natural to expect that for an Anosov $\RR^d$-action there are obstructions to solving the coboundary equation in degree $d$, but not in degrees $1$ through $d-1$. In fact, the ``right'' generalization of the Liv\v{s}ic Theorem to higher-rank actions should state that the \emph{only} obstructions in the top degree are those coming from integration over closed orbits.  This is essentially the content of the following conjecture for standard partially hyperbolic actions, due to A.~and S.~Katok.

\begin{conjecture*}[Katok--Katok,~\cite{KK95}]
Let $\a$ be a standard partially hyperbolic action of $\ZZ_+^d$, $\ZZ^d$, or $\RR^d$, $d\geq2$.  Then for $1\leq n\leq d-1$ the smooth $n$-cohomology of $\a$ trivializes, and the only obstructions to solving the degree-$d$ coboundary equation come from integration (or sums) over closed orbits.
\end{conjecture*}

Much less is known for higher-degree cohomology than for degree-one cohomology and, indeed, the above conjecture has only been investigated for very few cases. Katok and Katok proved it for $\ZZ^d$-actions on tori~\cite{KK95, KK05}, and in~\cite{Ramhc} we treated the case of Anosov $\RR^d$-actions on quotients of $\SL(2,\RR)\times\dots\times\SL(2,\RR)$.  There, it is shown that the lower-degree cohomologies trivialize, as conjectured, and that the obstructions in the top degree come from action-invariant distributions. 

\paragraph{On this article's contribution.}
The main result of the present article is Theorem~\ref{one}. It extends the results in~\cite{Ramhc} to Weyl chamber flows associated to $\SO^\circ (\en_1,1)\times\dots\times\SO^\circ (\en_d,1)$ with $\en_i \geq 2$.

Weyl chamber flows, described in Section~\ref{resultssec}, are an important class of Anosov actions that are obtained through a construction involving semisimple Lie groups. They are one of the three types of actions constituting the so-called \emph{standard Anosov actions}, the other two being Anosov actions on tori and nilmanifolds, and ``twisted'' actions (details are found in~\cite{KS94first}). Challenges arise when passing from actions on tori to Weyl chamber flows, stemming from the fact that the Fourier-analytic techniques that are available in the former setting are in the latter replaced by representation-theoretic techniques, which tend to be more complicated, especially as the group's unitary dual increases in complexity. In this sense, the group $\SL(2,\RR)$ is the ``easiest'' because its unitary dual is completely understood, and numerous wieldy descriptions of it are found in the literature.  Therefore, the results in~\cite{Ramhc} constitute the most natural first examples of Weyl chamber flows where the Katok--Katok Conjecture should be considered.

In passing from the representation theory of $\SL(2,\RR)$ to that of $\SO^\circ (\en,1)$, we encounter new and significant technical issues. Among them, the fact that a basis in each irreducible unitary representation is no longer parametrized by $\ZZ$ (points on a line), but by a collection of ``Gelfand--Cejtlin arrays" satisfying certain relations among one another (see Figure~\ref{figarray}). So for example, in the base case $d=1$, what was once a one-parameter difference equation corresponding to the $1$-coboundary equation in an irreducible unitary representation of $\SL(2,\RR)$ is now a \emph{partial} difference equation in \emph{several} parameters corresponding to the same coboundary equation in an irreducible unitary representation of $\SO^\circ(\en,1)$. Also, what before was a two-dimensional space of obstructions to the coboundary equation in the base case is now a space of obstructions that has \emph{infinite} countable dimension (see Theorem~\ref{invdistthm}).

Therefore, much of the effort in this article is in adapting the methods from~\cite{Ramhc} to this new scenario. In Part~\ref{parti} we establish a base case of our main result by articulating and solving the above mentioned partial difference equation; we also find a good description of the obstructions in the form of Theorem~\ref{invdistthm}. In Part~\ref{partii} we address the Katok--Katok Conjecture for our actions by adapting the arguments from~\cite{Ramhc} to accommodate this infinite-dimensional space of obstructions.

\paragraph{On invariant distributions, the obstructions.}

It is easy to see the ``only if'' part of the Liv\v{s}ic Theorem.  The fundamental theorem of calculus tells us that it is necessary for $f$ to have zero integral around every periodic orbit in order that $f$ be the derivative of some other function in the orbit direction (in which case $f$ is called a \emph{coboundary}).  The fact that for Anosov flows this condition is also sufficient is what makes the Liv\v{s}ic Theorem striking.  For a non-Anosov example where the same is not true, one need only look at irrational Liouvillean toral rotations (the definition of which we omit for the sake of brevity, but refer the reader to~\cite{KR}).  There, the necessary condition from the Liv\v{s}ic Theorem holds vacuously for any function, as there are no periodic orbits, yet not every smooth function is a coboundary.  

In general, one should look at the space of \emph{distributions} (elements of the dual to the space of smooth functions) that are invariant under the action.  Invariant distributions are always obstructions to solving the top-degree coboundary equation, and in the case of Anosov flows on compact manifolds they happen to be approximable (in the weak topology) by linear combinations of closed orbit measures.  In fact, if one can show that the same is true for Anosov $\RR^d$-actions, then this, together with Theorem~\ref{one}, would constitute a full verification of Katok and Katok's conjecture for the actions we consider.  However, the above example of Liouvillean toral rotations shows that this is not always a given.  

\paragraph{On representation theory.}

Another prominent example where there are invariant distributions not approximated by closed orbit measures is the horocycle flow of a hyperbolic surface.  These were studied by L.~Flaminio and G.~Forni in~\cite{FF}.  They show that the invariant distributions form a complete set of obstructions to solving the coboundary equation, and that for elements of Sobolev spaces, solutions come with a fixed loss of Sobolev regularity and with tame bounds on their Sobolev norms.  The main tool for this result is a detailed description of the unitary dual of $\PSL(2,\RR)$, the group of orientation-preserving isometries of the hyperbolic plane.  The result has a statement in terms of unitary representations with spectral gap which, when applied to the regular representation on $\Lii(\PSL(2,\RR)/\G)$, yields the familiar statement about functions on the unit tangent bundle of the surface.

D.~Mieczkowski~\cite{M2} used a similar strategy to show that the same holds for geodesic flows on quotients of $\PSL(2,\RR)$.  In the case of compact quotients, this (together with the knowledge that closed orbit measures approximate invariant distributions) recovers the Liv\v{s}ic Theorem for geodesic flows of compact hyperbolic surfaces.  On the other hand, Mieczkowski's result also holds for non-compact quotients of $\PSL(2,\RR)$, in which case it offers a strategy for proving a version of the Liv\v{s}ic Theorem for geodesic flows of finite-volume hyperbolic surfaces with cusps.

In Part~\ref{parti} of this paper we obtain a similar result to Flaminio--Forni and Mieczkowski's for geodesic flows of finite-volume hyperbolic manifolds of any dimension.  Theorem~\ref{invdistthm} is a complete list of the invariant distributions in any irreducible unitary representation of $\SO^\circ (\en,1)$, and Theorem~\ref{cobgeodflow} shows that the invariant distributions form a complete set of obstructions to the coboundary equation in any irreducible unitary representation.  We also give tame estimates on the Sobolev norms of the solutions.  Our main tool is Hirai's description of the irreducible unitary representations of $\SO^\circ (\en,1)$ in the papers \cite{Hir1, Hir2, Hir3, Hir4}.

A version of Theorem~\ref{cobgeodflow} for ``$M$-invariant'' elements of irreducible unitary representations is then implemented in Part~\ref{partii} as the base case of an inductive argument to prove Theorems~\ref{generaltopdegree} and~\ref{generallowerdegree}, which are ``Sobolev spaces''-versions of our main results for higher cohomologies of $\RR^d$-actions.

%\paragraph{On applications}

%The top degree part of Katok and Katok's conjecture has an interesting application to finding spanning sets for automorphic cusp forms.  This is through a program developed by T. Foth and S. Katok in~\cite{FK01}, and introduced by S. Katok in~\cite{Kat85}. 

%============================

\subsection{Results}\label{resultssec}
All of the $\RR^d$-actions in this work are obtained in the following way.  Let $G:=G_1 \times\dots\times G_d$, with $G_i = \SO^\circ (\en_i,1)$, and $\G \subset G$ an irreducible lattice.  Let $X_i \in \g_i$ be a semisimple element of the Lie algebra $\g_i$ of $G_i$.  We consider the group $A \cong \RR^d$ determined by flowing along $X_1,\dots,X_d$ on $M\backslash G/\G$, where $M=M_1\times\dots\times M_d$ and $M_i = Z_{K_i} (X_i)$ is the centralizer of $X_i$ in the maximal compact subgroup $K_i:=\SO(\en_i) \subset G_i$. This $\RR^d$-action is Anosov, and is in fact a member of an important class of Anosov actions called \emph{Weyl chamber flows}, all of which are obtained through a construction similar to the one we have just presented.

It is also worthwhile to understand our Weyl chamber flows from a more geometric point of view. It is well-known that $\SO^\circ(\en,1)$ is isomorphic to the group of orientation-preserving isometries of the hyperbolic space $\Hyp^{\en}$.  The maximal compact subgroup $K \subset \SO^\circ(\en,1)$ stabilizes a point, so that the identification $\Hyp^{\en} \cong K\backslash\SO^\circ(\en,1)$ can be made and, after choosing a semisimple $X\in\so(\en,1)$ and setting $M = Z_K (X)$, the quotient $M\backslash \SO^\circ(\en,1)$ is identified with the unit tangent bundle $S\Hyp^{\en}$ and the vector field $X$ generates its geodesic flow.  Therefore, the $A\cong\RR^d$-actions defined above are actions on unit tangent bundles of finite-volume irreducible quotients of $\Hyp^{\en_1}\times\dots\times\Hyp^{\en_d}$, defined by flowing along the elements of $\so(\en_i,1)$ defining the geodesic flows of each factor.

The main result of this article is the following theorem on the cohomology of the $\RR^d$-actions defined above.  It is proved in Part~\ref{partii}.

\begin{theorem}\label{one}
Let $\mathcal{M}$ be a finite-volume irreducible quotient of the product $\Hyp^{\en_1}\times\dots\times\Hyp^{\en_d}$ by a discrete group of isometries, and let $S\mathcal{M}$ be its unit tangent bundle.  
\begin{itemize}
\item \textbf{Top-degree cohomology:}  If $f \in \Cinf(\Lii(S\mathcal{M}))$ is in the kernel of every $A$-invariant distribution, then there exist smooth functions $g_1,\dots,g_d \in \Cinf(\Lii(S\mathcal{M}))$ satisfying the degree-$d$ coboundary equation
	\[
		X_1\,g_1 + \dots + X_d\,g_d = f.
	\]
\item \textbf{Lower-degree cohomology:} The smooth $n$-cohomology of the $A$-action on $S\mathcal{M}$ trivializes for $1 \leq n \leq d-1$.  
\end{itemize}
\end{theorem}

This extends the results of~\cite{Ramhc}, where the same is proved for Anosov $\RR^d$-actions on quotients of $\SL(2,\RR)\times\dots\times\SL(2,\RR)$.  There, one of the main ingredients is a theorem of Mieczkowski \cite{M2} showing that the obstructions to solving the coboundary equation $Xg=f$ for the geodesic flow on a quotient of $\PSL(2,\RR)$ are exactly the flow-invariant distributions.  A version of the same statement for $\SL(2,\RR)$ serves as the base case of an inductive argument, and is implemented as a ``black box.''  Therefore, in order to use similar methods to obtain Theorem~\ref{one}, we have to prove the corresponding theorem for geodesic flows of finite-volume hyperbolic manifolds of higher dimension (Theorem~\ref{two}).

The following theorem is the main aim of Part~\ref{parti}.  It is a ``Sobolev spaces''-version of Theorem~\ref{one} for $d=1$, and its proof is essentially the base case for the inductive argument used in Part~\ref{partii} to prove Theorem~\ref{one}.

\begin{theorem}\label{two}
Let $\mathcal{M}$ be a finite-volume hyperbolic manifold of dimension at least two, and let $\phi^t$ be the geodesic flow on its unit tangent bundle $S\mathcal{M}$.  Given any $s >1$ and $t \leq s - 1$, there exists a constant $C_{s,t}>0$ such that for any element $f$ of the mutual kernel of all $\phi^t$-invariant distributions of Sobolev order $s$, there is a solution to the coboundary equation 
\[
	\frac{d}{dt}(g\circ \phi^t)\Bigr\vert_{t=0}=f
\]
that satisfies the estimate $\norm{g}_t \leq C_{s,t}\,\norm{f}_s$ on Sobolev norms.
\end{theorem}

\begin{remark*}
The constant $C_{s,t}$ may also depend on things that are ``fixed'' in the theorem statement, such as the dimension of the manifold $\mathcal{M}$ and its homotopy type. Actually, we will find that it depends on a ``spectral gap parameter,'' to be described in Section~\ref{spectralgapsection}.
\end{remark*}

Notice that if $\mathcal{M}$ is compact, then Theorem~\ref{two} is decidedly weaker than the Liv\v{s}ic Theorem, because closed orbit measures form a subset of invariant distributions.  If, on the other hand, $\mathcal{M}$ has cusps, then Theorem~\ref{two} can be combined with results of T.~Foth and S.~Katok to get a smooth version of the Liv\v{s}ic Theorem for geodesic flows of hyperbolic manifolds with cusps (Theorem~\ref{livsicthm}).

We also reiterate that it is not so much Theorem~\ref{two}, but its proof, which will be needed in Part~\ref{partii}.  Namely, we use a representation-theoretic version (Theorem~\ref{cobgeodflow}) of Theorem~\ref{two}, and we use a result (Theorem~\ref{invdistthm}) listing the $X$-invariant distributions in any irreducible unitary representation of $\SO^\circ(\en,1)$, which is interesting in itself for its contrast with pre-existing results of Flaminio and Forni~\cite[Theorem~$3.2$]{FF} and Mieczkowski~\cite[Theorem~$1.1$]{M2} for horocycle and geodesic flows of hyperbolic surfaces.  (In fact, both Theorem~\ref{two} and the corresponding theorem in~\cite{M2} are inspired by the work of Flaminio and Forni for horocycle flows.)

Finally, we point out that, like Theorem~\ref{two}, Theorem~\ref{one} is proved in representation-theoretic pieces and has ``Sobolev spaces''-versions which are listed here as Theorems~\ref{generaltopdegree} and~\ref{generallowerdegree} in Section~\ref{higherdegreeresults}.

%============================

\subsection{Preliminaries}\label{preliminaries}

We now set some definitions and state some facts that are relevant to both parts of the article.

\subsubsection{Cohomology in dynamics}

Our definitions for cohomology over an $\RR^d$-action on a manifold $\mathcal{M}$ are essentially the definitions for the tangential De Rham cohomology of $\mathcal M$ with respect to its $\RR^d$-orbit foliation.  More precisely, for $1\leq n\leq d$, an \emph{$n$-form} $\o$ is a smooth assignment\footnote{By this we mean that for a fixed $n$-tuple $(V_1,\dots,V_n)\in(\RR^d)^n$, the evaluation $\o_x (V_1,\dots,V_n)$ is a smooth function of $x\in \mathcal M$.}~of a multi-linear and skew-symmetric map
\[
	\o_x: (T_x (\RR^d\cdot x))^n\cong(\RR^d)^n\to\RR
\]
to each $x\in \mathcal M$, where $\RR^d\cdot x$ is the orbit of $x$ and $T_x(\RR^d\cdot x)\subset T_x \mathcal M$ is its tangent space at $x$, which is naturally identified with the acting group $\RR^d$.  The set of $n$-forms is denoted $\O_{\RR^d}^{n}(\mathcal M)$.  The \emph{exterior derivative}, defined by
\[
	(\di\o) (V_1,\dots,V_{n+1}) := \sum_{j=1}^{n}(-1)^{j+1}\, V_j \,\o (V_1,\dots,\widehat{V_j},\dots,V_{n+1}), 
\]
takes $n$-forms to $(n+1)$-forms, where ``$\quad\widehat{}\quad$'' denotes omission.  An $n$-form $\o$ is said to be \emph{closed}, and is called a \emph{cocycle}, if $\di\o=0$.  It is said to be \emph{exact}, and is called a \emph{coboundary}, if there is an $(n-1)$-form $\eta$ satisfying the \emph{coboundary equation} $\di\eta=\o$.  Two cocycles are \emph{cohomologous} if their difference is a coboundary.

For the $\RR^d$-actions described in Section~\ref{resultssec}, any cocycle is determined by its values on the vectors $X_1,\dots,X_d$.  Therefore, an $n$-cocycle $\o$ is determined by $\binom{d}{n}$ smooth functions $\o(X_{i_1},\dots,X_{i_n})$, with $1\leq i_1<\dots<i_n\leq d$, and these smooth functions satisfy relations among one another given by the requirement that $\di\o=0$.  Similarly, solving the coboundary equation $\di\eta=\o$ is a matter of finding $\binom{d}{n-1}$ smooth functions satisfying a system of partial differential equations.  

All of this is easiest to write in the top and bottom degrees.  For top degree, a $d$-cocycle $\o$ is given by a single smooth function $f=\o(X_1,\dots,X_d)$, and solving $\di\eta=\o$ is equivalent to finding smooth functions $g_1,\dots,g_d$ satisfying
\[
	X_1\,g_1+\dots+X_d\,g_d = f. 
\]
In the case of a flow along $X$ on a quotient of $\SO^\circ(\en,1)$, this is just $Xg=f$, which is exactly the coboundary equation~\eqref{cobeq} from the Introduction.  This is the focus of Part~\ref{parti}.

As for bottom degree, a $1$-cocycle $\o$ is determined by the smooth functions $f_i=\o(X_i)$, for $i=1,\dots,d$, which satisfy certain relations among one another, in accordance with the closedness requirement $\di\o=0$.  Solving the coboundary equation is then a matter of finding a single smooth function---a ``$0$-form,'' $g$---satisfying $X_i\,g = f_i$ for all $i=1,\dots,d$.  One can compare this with the usual definitions, where an $\RR$-valued $1$-cocycle over a group action $\RR^d\curvearrowright \mathcal M$ is a map $\a:\RR^d\times \mathcal M\to\RR$ satisfying the cocycle identity $\a(r_1 + r_2, x) = \a(r_1, r_2\,x) + \a(r_2,x)$, and it is cohomologous to a second $1$-cocycle $\b$ if there is a transfer function $g:\mathcal M\to\RR$ satisfying $\a(r,x) = -g(rx)+\b(r,x)+g(x)$.  When $\a, \b, g$ are smooth, one can differentiate these expressions in the $\RR^d$-directions to recover the definitions we have presented here for degree-$1$ cohomology.

\subsubsection{Direct decompositions of unitary representations}

The main goal of this article is to solve coboundary equations, say $Xg=f$ for a given function $f$ which is smooth or a member of some Sobolev space.  Our general strategy is to see this as a problem about representations of $\SO^\circ(\en,1)$, where $\HH = \Lii(\SO^\circ(\en,1)/\G)$ is the Hilbert space of the left-regular unitary representation, and $f \in \Cinf(\Lii(\SO^\circ(\en,1)/\G))$ is a smooth vector.  Since the unitary dual of $\SO^\circ(\en,1)$ is well-understood, we would like to use this to find a solution to the same coboundary equation in any \emph{irreducible} unitary representation, and somehow patch the resulting solutions together in a way that is meaningful for the representation on $\Lii(\SO^\circ(\en,1)/\G)$.  For this, we use what is called the \emph{direct integral decomposition}.  The following facts about direct integral decompositions are standard, and can be found in~\cite{Mau50a, Mau50b}.

Any unitary representation $\pi$ of a locally compact second countable group $G$ on a separable Hilbert space $\HH$ has a direct integral decomposition over $\RR$.  That is, there is some positive Stieltjes measure $ds$ on $\RR$ such that the Hilbert space $\HH$ can be written as
\[
	\HH = \int_\RR^\oplus \HH_\l\,ds(\l)
\]
where the $\HH_\l$ are Hilbert spaces with unitary representations $\pi_\l$ of $G$, such that for every $f\in\HH$ and $g\in G$, 
\[
	\pi(g)f=\int_\RR^\oplus \pi_\l (g) f_\l \,ds(\l).
\]
That is, the operators $\pi(g)$ decompose accordingly.  Furthermore, $ds$-almost every $\pi_\l$ is an irreducible unitary representation of $G$.

If $G$ is a Lie group (as is the case throughout this article) then the operators coming from the universal enveloping algebra of $\g$ also decompose.  This means that we have the same decomposition for Sobolev spaces $W^s(\HH)$ and for spaces of invariant distributions $\II^{s}(\HH)$.  

These facts allow us to restrict most of our attention to irreducible unitary representations.  All of our theorems have ``irreducible'' versions: statements for irreducible unitary representations which imply the main statements after observing the above facts about direct integral decompositions and invoking a ``spectral gap'' fact (to be discussed in Section~\ref{spectralgapsection}).

\subsubsection{Asymptotic notations}

We use the so-called Vinogradov asymptotic notation throughout.  Let $\psi_1, \psi_2$ be positive-valued functions.  As is standard, $\psi_1\ll\psi_2$ means that there is some constant $C>0$ such that $\psi_1\leq C\cdot\psi_2$ throughout the domain of $\psi_1$ and $\psi_2$.  The notation $\psi_1\asymp\psi_2$ means ``$\psi_1\ll\psi_2$ \emph{and} $\psi_2\ll\psi_1$.''  Any subscripts next to the symbols $\ll$ and $\asymp$ are parameters on which the constant $C$ may depend.  So, for example, the bound on Sobolev norms in Theorem~\ref{two} can be written $\norm{g}_t \ll_{s,t} \norm{f}_s$.

We warn that many of our asymptotes may depend on things that are ``fixed" in the theorem statements, and so we will not denote them. The most consistent example of this omission is that most of our asymptotes \emph{should} have a subscript of `` $k$ '' alongside the `` $\nu_0,s,t$ '' where $k = \floor{\en/2}$ for $\SO^\circ(\en,1)$. But since we fix the group in the theorem statements, we do not gain anything by acknowledging this dependence on $k$ in our calculations.

%=============================================

\part{Invariant distributions and cohomology for geodesic flows}\label{parti}

\section{Results for geodesic flows}

We aim to solve the coboundary equation $Xg=f$ where $f \in \Cinf(\HH)$ is a smooth vector of a unitary representation of $\SO^\circ(\en,1)$, or an element of some Sobolev space $W^s (\HH)$, and $X \in \so(\en,1)$ is the semisimple element
\[
	X = \begin{pmatrix} 0 & \cdots & 0 &0\\ 
						\vdots & \ddots & \vdots &\vdots \\
						0 & \cdots & 0 & 1 \\
						0 & \cdots & 1 & 0 \end{pmatrix}
	=\begin{pmatrix}\mathbf{0}_n & e_\en \\ e_\en^t & 0 \end{pmatrix}.
\]
This is the element defining the geodesic flow through the standard identification of $\SO^\circ(\en,1)$ with the orientation-preserving isometries of hyperbolic $\en$-space, $\Isom(\Hyp^{\en})$.

There are obstructions to solving the coboundary equation coming from flow-invariant distributions
\[
	\II_X(\HH) := \left\{ \D\in\E^{\prime}(\HH)\mid\LLL_X \D=0\right\},
\]
where $\E^{\prime}(\HH)$ denotes the dual to the space of smooth vectors $\Cinf(\HH)$ for the representation on $\HH$.  It is necessary that $f\in\Cinf(\HH)$ be in the kernel $\ker\II_X(\HH)$ of all of these.  Similarly, in order to solve the coboundary equation for an element of a Sobolev space $f\in W^s(\HH)$, it is necessary that $f$ belong to the kernel of all $X$-invariant distributions 
\[
	\II_X^s(\HH) := \left\{ \D\in W^{-s}(\HH)\mid\LLL_X \D=0\right\},
\]
of Sobolev order $s$.  The main result of Part~\ref{parti} is that this is also a sufficient condition, and that solutions come with tame bounds on their Sobolev norms.

\begin{theorem}[Coboundary equation in irreducible unitary representations]\label{cobgeodflow}
Consider a non-trivial irreducible unitary representation $\pi:\SO^\circ(\en,1)\to\U(\HH)$ where $\en\geq 3$.  There is a number $s_0 >0$ such that for any $s>s_0$ and $t\leq s - s_0$, there is a constant $C_{\pi,s,t}>0$ such that for any $f \in \ker\II_{X}^{s}(\HH)$ there exists $g \in W^{t}(\HH)$ satisfying $X g = f$ and $\norm{g}_t \leq C_{\pi,s,t}\,\norm{f}_s$.
\end{theorem}

\begin{remark*}
We prove the theorem with $s_0 = \floor{\en/2}-1/2$.
\end{remark*}

The proof relies on specific knowledge of the unitary dual of $\SO^\circ (\en,1)$.  Along the way, we characterize the spaces of invariant distributions for irreducible representations.  They are infinite-dimensional, unlike for the geodesic and horocycle flows on $\SO^\circ (2,1)\cong\PSL(2,\RR)$, where in any irreducible unitary representation the space of invariant distributions is at most $2$-dimensional.  (See~\cite[Theorem~$1.1$]{FF} and~\cite[Theorem~$1.1$]{M2}.)

\begin{theorem}[Invariant distributions in irreducible unitary representations]\label{invdistthm}
Let $\SO^\circ (\en,1)\to\U(\HH)$ be an irreducible unitary representation, $\en\geq 3$.  For any $s > 1/2$, the space $\II_X^s(\HH)$ of $X$-invariant distributions of Sobolev order $s$ has infinite countable dimension and is spanned by the set $\left\{\D^{\bm,\l}\mid \bm \in\floor{M_\l}\right\}$.
\end{theorem}

\begin{remark*}
The distributions $\D^{\bm,\l}$ are defined in Section~\ref{formalsection}, Definitions~\ref{defeven} and~\ref{defodd}.
\end{remark*}

Theorem~\ref{two} does not exactly follow from Theorem~\ref{cobgeodflow}.  Rather, their proofs run along parallel representation-theoretic tracks: Theorem~\ref{cobgeodflow} holds for \emph{any} non-trivial irreducible unitary representation of $\SO^\circ (\en,1)$, while Theorem~\ref{two} is proved with only representations that admit $\mathbb{M}(\en)$-invariant elements in mind.  The result is that Theorem~\ref{cobgeodflow}, while holding for all irreducible representations, suffers from a stronger dependence (in the form of the positive constant $C_{\pi,s,t}$) on the parameters that define the representation, and so cannot be easily used to make statements about general (reducible) unitary representations of $\SO^\circ (\en,1)$,\footnote{In particular, Theorem~\ref{cobgeodflow} cannot be directly applied to the problem of solving the coboundary equation for the flow of $X$ on $\SO^\circ(\en,1)/\G$.  More on this in Section~\ref{dependencesection}.} whereas the irreducible representations that we need to consider for proving Theorem~\ref{two} are easier to handle, and yield results with weaker dependence on the representation that allow us to make a more general statement about the $\mathbb{M}(\en)$-invariant elements of the (reducible) left-regular unitary representation of $\SO^\circ(\en,1)$ on $\Lii(\SO^\circ(\en,1)/\G)$.  The reader will therefore notice that many of the intermediate lemmas have a general version, used in the proof of Theorem~\ref{cobgeodflow}, and a stronger version that only holds for the $\mathbb{M}(\en)$-invariant case and is used for Theorem~\ref{two}.

We end this section with a brief discussion of the relationship between Theorems~\ref{two} and~\ref{cobgeodflow} and the Liv\v{s}ic Theorem.  Since integration over a periodic orbit is itself an invariant distribution, some part of Theorem~\ref{two} is implied by the Liv\v{s}ic Theorem.  Namely, if the manifold $\mathcal{M}$ is compact, then the Liv\v{s}ic Theorem implies Theorem~\ref{two}, \emph{without} the estimates on Sobolev norms.  However, since Theorem~\ref{cobgeodflow} takes place in abstract representations, it cannot be directly deduced from the Liv\v{s}ic Theorem.  It is Theorem~\ref{cobgeodflow}, or rather, an $\mathbb{M}(\en)$-invariant version of it and the ingredients that go into its proof, that we rely on in Part~\ref{partii}.

In the other direction, one may ask what contribution, if any, Theorem~\ref{two} makes to the Liv\v{s}ic Theorem. After combining with a result of T.~Foth and S.~Katok~\cite{FK01}, where a bounded Lipschitz version of the Liv\v{s}ic Theorem for manifolds with cusps is proved, we obtain the following statement of the regularity of solutions, which, to our knowledge, is not currently presented anywhere in the literature.

\begin{theorem}[Smooth version of the Liv\v{s}ic Theorem for bounded functions on finite-volume hyperbolic manifolds]\label{livsicthm}
Let $\mathcal{M}$ be a finite-volume (not necessarily compact) hyperbolic manifold, with geodesic flow $\phi^t$ on its unit tangent bundle $S\mathcal{M}$.  Suppose $f$ is a smooth and bounded function on $S\mathcal{M}$ with the property that its integrals over all periodic orbits of $\phi^t$ are $0$.  Then there is a smooth and bounded solution $g \in \Cinf(S\mathcal{M})$ to the coboundary equation \eqref{cobeq}.
\end{theorem}

The Foth--Katok result in question is~\cite[Section 3.1]{FK01} and, in fact, they also prove a ``special'' Liv\v{s}ic Theorem that holds for some non-Anosov homogeneous flows.  In the Anosov case, our Theorem~\ref{two} simply implies that the solutions in~\cite{FK01} are smooth if the initial data are.  

Let us emphasize that there may be other methods for proving regularity of solutions, particularly ways that make explicit use of the hyperbolicity of the geodesic flow and are more accessible and direct from a dynamical point of view than our representation-theoretic approach.  However, it is also worth pointing out that those methods are no longer available for higher-degree cohomology over higher-rank actions, at least not without drastic re-interpretation.  Even the Anosov Closing Lemma (an essential ingredient for the Liv\v{s}ic Theorem \emph{without} regularity) is absent for higher-rank actions.  It is unclear how one should modify these classical ideas to make them work toward proving Katok and Katok's conjecture.  

On the other hand, the representation-theoretic methods we adopt seem very promising in this regard, at least for the types of actions we consider here.  For one, they allow us to consider functions that are not necessarily bounded.  Notice that Theorem~\ref{two} only calls for smooth $\Lii$-functions, which may conceivably diverge at cusps.  Another advantage of taking a representation-theoretic point of view is that we gain a precise description of the invariant distributions, in the form of Theorem~\ref{invdistthm}.  This is an indispensable part of our strategy in Part~\ref{partii} for the proof of Theorem~\ref{one}.

\section{Notation for Part~\ref{parti}}

We put $\mathbb{G}(\en):=\SO^\circ(\en,1)$; the maximal compact subgroup of $\mathbb{G}(N)$ is $\mathbb{K}(\en):=\SO(\en)$, and $\mathbb{M}(\en):=\SO(\en-1)$ is the centralizer of $X$ in $\mathbb{K}(\en)$.  So $\mathbb{K}(\en)\backslash \mathbb{G}(\en)/\G$ is a finite-volume hyperbolic manifold, and $\mathbb{M}(\en)\backslash \mathbb{G}(\en)/\G$ is its unit tangent bundle; $\mathbb{K}(\en)\backslash \mathbb{G}(\en) \cong\Hyp^{\en}$ is its universal cover. When there is no risk of confusion we use the labels $G:=\mathbb{G}(N)$, $K:=\mathbb{K}(N)$, and $M:=\mathbb{M}(N)$ to make the notation more concise.

\subsection{Basis for the Lie algebra $\so(\en,1)$}

Let us fix some elements of $\so(\en,1)$ once and for all.  Let $E_{i,j}$ be the $(\en\times \en)$-matrix with $1$ in the $i^{\textrm{th}}$ row and $j^{\textrm{th}}$ column and $0$'s in every other entry.  Let $e_i$ be the $i^\textrm{th}$ standard basis vector for $\RR^\en$.  Put
\begin{align*}
Y_i &=\begin{pmatrix}\mathbf{0}_\en & e_i \\ e_i^t & 0 \end{pmatrix} &\textrm{for } i=1,\dots,\en
\intertext{and}
\Theta_{i,j} &=\begin{pmatrix}E_{j,i} - E_{i,j} & 0 \\ 0^t & 0 \end{pmatrix} &\textrm{for } 1\leq i < j \leq \en.
\end{align*}
The elements $Y_i$ are semisimple, and we are most interested in $Y_\en :=X$, our generator for the geodesic flow.  The elements $\Theta_{i,j}$ form an orthonormal basis for the subalgebra $\so(\en)$ with respect to the inner product defined by minus half the Killing form: $\langle X,Y \rangle = -\frac{1}{2}\mathrm{tr}(XY)$.

%===================================================

\section{Unitary representations of $\SO^\circ(\en,1)$}

Our sources for the unitary dual of $\SO^\circ(\en,1)$ are papers of Hirai \cite{Hir1,Hir2,Hir3,Hir4} and Thieleker \cite{Thi73, Thi74}.

%==========================================================

\subsection{Principal, complementary, and discrete series representations}\label{pced}

Let $\SO^\circ(\en,1)\to\U(\HH)$ be an irreducible unitary representation.  It is determined by a $\floor{\frac{\en-1}{2}}$-tuple $\bn=(n_i)$ of integers satisfying
\begin{align*}
&0 \leq n_1 \leq\dots\leq n_{k-1} ,&\textrm{ if } &\en=2k,\textrm{ and} \\
&\abs{n_1}\leq n_2 \leq \dots \leq n_k,&\textrm{ if } &\en=2k+1,
\end{align*}
and a complex number 
\[
	\nu \in \begin{cases}
			 \displaystyle{i\RR_{\geq 0}} &\bullet\textrm{ for the \emph{principal series} representations}\\ \\
			 \displaystyle{\left\{\left(0,\frac{\en-1}{2}-j\right)\right\}_{j=0,\dots,k-1}} &\bullet\textrm{ for \emph{complementary series} representations}\\ \\
			 \displaystyle{\left\{\frac{\en-1}{2}-j\right\}}_{j=0,\dots,k-1} &\bullet\textrm{ for \emph{end-point} representations}\\ \\
			 \displaystyle{\ZZ_+ - \frac{1}{2}} &\bullet\textrm{ for \emph{discrete series} representations (for $\en$ even)}
			 \end{cases} 
\]
and hence will be denoted $\HH_{{\bn},\nu}$.  We use the notation $\ceil{\bn} := \norm{\bn}_\infty$.

The parameters $\bn, \nu$ determine whether the representation on $\HH_{\bn,\nu}$ is from the so-called \emph{principal series}, \emph{complementary series}, their \emph{end-point representations}, or \emph{discrete series}.

\paragraph{When $\en=2k$:}

\begin{itemize}
\item \textbf{Principal series, $\calP_{\bn,\nu}$}\\ $\nu \in i\RR_{\geq 0}$
\item \textbf{Complementary series, $\calC_{\bn,\nu}^j$}\\ $n_{j-1}=0< n_j$ for some $j = 1,\dots, k$, and $0<\nu<j-1/2$.  (For convenience, we put $n_0 = 0$ and $n_k = \infty$.)
\item \textbf{End-points, $\E_{\bn,\nu}^j$}\\ $n_{j-1}=0< n_{j}$ for some $j = 1,\dots, k$, and $\nu=j-1/2$.  The representation contains $\bm$ for which $m_j = 0$.  (See below.)
\item \textbf{Discrete series, $\D_{\bn,\nu}^{+}$ and $\D_{\bn,\nu}^{-}$}\\ $n_1> 0$ and $p \in \ZZ$ with $0<p\leq n_1$.  Then $\nu=p-1/2$, and the representation contains $\bm$ for which $m_1\geq p$ and, respectively, $-m_1\geq p$.  (See below.)
\end{itemize}

\paragraph{When $\en=2k+1$:}
\begin{itemize}
\item \textbf{Principal series, $\calP_{\bn,\nu}$}\\ $\nu \in i\RR_{\geq 0}$
\item \textbf{Complementary series, $\calC_{\bn,\nu}^j$}\\ $n_j = 0 < n_{j+1}$ for some $j=1,\dots, k$, and $0<\nu<j$.
\item \textbf{End-points, $\E_{\bn,\nu}^j$}\\ $n_j=0 < n_{n+1}$ and $\nu=j$.  The representation contains $\bm$ for which $m_j = 0$.  (See below.)
\end{itemize}

The restricted representation of the maximal compact subgroup $\SO(\en)\subset\SO^\circ(\en,1)$ is a direct sum of $\SO(\en)$-irreducible sub-representations, and they are parametrized by $k$-tuples $\bm = (m_1, \dots, m_k)$ satisfying the inequalities
\begin{align}\label{mcondition}
\begin{cases}
\abs{m_1} \leq n_1 \leq m_2 \leq n_2 \leq\dots\leq m_{k-1}\leq n_{k-1} \leq m_k < \infty,&\textrm{ if }\en=2k,\text{ and} \\
\abs{n_1} \leq m_1 \leq n_2 \leq m_2 \leq\dots\leq m_{k-1}\leq n_{k} \leq m_k < \infty,&\textrm{ if }\en=2k+1.
\end{cases}
\end{align}
We write $\HH_{\bm}$ for the corresponding Hilbert spaces.  Each such representation appears with multiplicity $1$ in the decomposition of $\HH_{{\bn},\nu}$ with respect to the maximal compact subgroup $\SO(\en)$.  We write $\HH_{{\bn},\nu} = \bigoplus \HH_{\bm}$, where the sum is taken over the set $M_{\bn}$ of all $\bm$ that satisfy~\eqref{mcondition}.  Again, $\ceil{\bm}:=\norm{\bm}_\infty = m_k$.

%===================================================

%==========================================================

\begin{figure}[tb]
\begin{tikzpicture}[>=stealth, fill=gray!20, xscale=0.8, yscale=1.2, font=\footnotesize]
%	\draw[help lines, densely dotted] (-4,-5) grid (4,4);
%	\draw (-6,4) node(n0) [draw, ultra thick, fill=gray!80] {$0=n_0$};
%	\draw (-4.8,4) node(leq0) {$\leq$} (-3,4) node(leq1) {$\leq$} (-1,4) node(leq2) {$\leq$};
	\draw (-4,4) node(n1) [draw, ultra thick, fill=gray!80]{$n_1$}  (-2,4) node(n2) [draw, ultra thick, fill=gray!80] {$n_2$}  (0,4) node(nk1) [draw, ultra thick, fill=gray!80] {$n_{k-1}$} 
  		(-4,3) node(m1) [draw, thick, rounded corners, fill=gray!50] {$m_1$}  (-2,3) node(m2) [draw, thick, rounded corners, fill=gray!50] {$m_2$}  (-0,3) node(m3) [draw, thick, rounded corners, fill=gray!50] {$m_3$}  (2,3) node(mk) [draw, thick, rounded corners, fill=gray!50] {$m_k$}
		(-2,2) node(l2k21) [draw, fill, rounded corners] {$\l_1^{(2k-2)}$} (0,2) node(l2k22) [draw, fill, rounded corners] {$\l_2^{(2k-2)}$} (2,2) node(l2k2k1) [draw, fill, rounded corners] {$\l_{k-1}^{(2k-2)}$}
		(-2,1) node(l2k31) [draw, fill, rounded corners] {$\l_1^{(2k-3)}$} (0,1) node(l2k32) [draw, fill, rounded corners] {$\l_2^{(2k-3)}$} (2,1) node(l2k3k1) [draw, fill, rounded corners] {$\l_{k-1}^{(2k-3)}$}
		(0,0) node(l41) [draw, fill, rounded corners] {$\l_1^{(4)}$} (2,0) node(l42) [draw, fill, rounded corners] {$\l_2^{(4)}$}
		(0,-1) node(l31) [draw, fill, rounded corners] {$\l_1^{(3)}$} (2,-1) node(l32) [draw, fill, rounded corners] {$\l_2^{(3)}$}
		(2,-2) node(l21) [draw, fill, rounded corners] {$\l_1^{(2)}$}
		(2,-3) node(l11) [draw, fill, rounded corners] {$\l_1^{(1)}$};
			
	\draw (-5,4) node(n) {$\bn=$};
	\draw (-5,3) node(m) {$\bm=$};
	\draw (2,5) node(ceiln) {$\ceil{\bn}$};
	\draw (4,4) node(ceilm) {$\ceil{\bm}$};
	\draw (4,3) node(ceill) {$\ceil{\l}$};
	\draw[left=5pt] (-4,-0.5) node(l) {$\l=$};
	\draw (-1,-3) node(even) [draw, double, thick] {$\en=2k$};
	
	\draw[snake=brace, thick](-4,-3.5)--(-4,2.4);
	
%	\draw[->, color=black, thick] (n0)--(n1);
	\draw[->, color=black, thick] (n1)--(n2);
	\draw[->, color=black, densely dotted, thick] (n2)--(nk1);
	\draw[-to, color=black, snake=snake, thick] (nk1)--(ceiln);
		
	\draw[->, double, color=black, thick] (m1)--(n1);
	\draw[->, color=black, thick] (n1)--(m2);
	\draw[->, color=black, thick] (m2)--(n2);
	\draw[->, color=black, thick] (n2)--(m3);
	\draw[->, color=black, densely dotted, thick] (m3)--(nk1);
	\draw[->, color=black, thick] (nk1)--(mk);
	\draw[-to, color=black, snake=snake, thick] (mk)--(ceilm);
	
	\draw[->, double, color=black, thick] (m1)--(l2k21);
	\draw[->, color=black, thick] (l2k21)--(m2);
	\draw[->, color=black, thick] (m2)--(l2k22);
	\draw[->, color=black, thick] (l2k22)--(m3);
	\draw[->, color=black, densely dotted, thick] (m3)--(l2k2k1);
	\draw[->, color=black, thick] (l2k2k1)--(mk);
	
	\draw[->, double, color=black, thick] (l2k31)--(l2k21);
	\draw[->, color=black, thick] (l2k21)--(l2k32);
	\draw[->, color=black, thick] (l2k32)--(l2k22);
	\draw[->, color=black, densely dotted, thick] (l2k22)--(l2k3k1);
	\draw[->, color=black, thick] (l2k3k1)--(l2k2k1);
	\draw[-to, color=black, snake=snake, thick] (l2k2k1)--(ceill);
	
	\draw[->, double, color=black, densely dotted, thick] (l2k31)--(l41);
	\draw[->, color=black, densely dotted, thick] (l41)--(l2k32);
	\draw[->, color=black, densely dotted, thick] (l2k32)--(l42);
	\draw[->, color=black, densely dotted, thick] (l42)--(l2k3k1);

	\draw[->, double, color=black, thick] (l31)--(l41);
	\draw[->, color=black, thick] (l41)--(l32);
	\draw[->, color=black, thick] (l32)--(l42);

	\draw[->, double, color=black, thick] (l31)--(l21);
	\draw[->, color=black, thick] (l21)--(l32);

	\draw[->, double, color=black, thick] (l11)--(l21);
\end{tikzpicture}
\quad
\begin{tikzpicture}[>=stealth, fill=gray!20, xscale=0.8, yscale=1.2, font=\footnotesize]
%	\draw[help lines, densely dotted] (-4,-5) grid (4,4);
%	\draw (-4.8,4) node(leq0) {$\leq$} (-3,4) node(leq1) {$\leq$} (-1,4) node(leq2) {$\leq$};
	\draw (-4,4) node(n1) [draw, ultra thick, fill=gray!80]{$n_1$}  (-2,4) node(n2) [draw, ultra thick, fill=gray!80] {$n_2$}  (0,4) node(nk) [draw, ultra thick, fill=gray!80] {$n_k$}
  		(-2,3) node(m1) [draw, thick, rounded corners, fill=gray!50] {$m_1$}  (-0,3) node(m2) [draw, thick, rounded corners, fill=gray!50] {$m_2$}  (2,3) node(mk) [draw, thick, rounded corners, fill=gray!50] {$m_k$}
		(-2,2) node(l2k11) [draw, fill, rounded corners] {$\l_1^{(2k-1)}$} (0,2) node(l2k12) [draw, fill, rounded corners] {$\l_2^{(2k-1)}$} (2,2) node(l2k1k) [draw, fill, rounded corners] {$\l_{k}^{(2k-1)}$}
		(0,1) node(l41) [draw, fill, rounded corners] {$\l_1^{(4)}$} (2,1) node(l42) [draw, fill, rounded corners] {$\l_{2}^{(4)}$}
		(0,0) node(l31) [draw, fill, rounded corners] {$\l_1^{(3)}$} (2,0) node(l32) [draw, fill, rounded corners] {$\l_2^{(3)}$}
		(2,-1) node(l21) [draw, fill, rounded corners] {$\l_1^{(2)}$}
		(2,-2) node(l11) [draw, fill, rounded corners] {$\l_1^{(1)}$};
	
	\draw (-5,4) node(n) {$\bn=$};
	\draw (-4,3) node(m) {$\bm=$};
	\draw (2,5) node(ceiln) {$\ceil{\bn}$};
	\draw (4,4) node(ceilm) {$\ceil{\bm}$};
	\draw (4,3) node(ceill) {$\ceil{\l}$};
	\draw[left=5pt] (-4,0) node(l) {$\l=$};
	\draw (-1,-2) node(phantom) [draw, double, thick] {$\en=2k+1$};
	
	\draw[snake=brace, thick](-4,-2.5)--(-4,2.4);
	
	\draw[->, double, color=black, thick] (n1)--(n2);
	\draw[->, color=black, densely dotted, thick] (n2)--(nk);
	\draw[-to, color=black, snake=snake, thick] (nk)--(ceiln);
			
	\draw[->, double, color=black, thick] (n1)--(m1);
	\draw[->, color=black, thick] (m1)--(n2);
	\draw[->, color=black, thick] (n2)--(m2);
	\draw[->, color=black, densely dotted, thick] (m2)--(nk);
	\draw[->, color=black, thick] (nk)--(mk);
	\draw[-to, color=black, snake=snake, thick] (mk)--(ceilm);

	\draw[->, double, color=black, thick] (l2k11)--(m1);
	\draw[->, color=black, thick] (m1)--(l2k12);
	\draw[->, color=black, thick] (l2k12)--(m2);
	\draw[->, color=black, densely dotted, thick] (m2)--(l2k1k);
	\draw[->, color=black, thick] (l2k1k)--(mk);
	\draw[-to, color=black, snake=snake, thick] (l2k1k)--(ceill);

	\draw[->, double, color=black, densely dotted, thick] (l2k11)--(l41);
	\draw[->, color=black, densely dotted, thick] (l41)--(l2k12);
	\draw[->, color=black, densely dotted, thick] (l2k12)--(l42);
	\draw[->, color=black, densely dotted, thick] (l42)--(l2k1k);

	\draw[->, double, color=black, thick] (l31)--(l41);
	\draw[->, color=black, thick] (l41)--(l32);
	\draw[->, color=black, thick] (l32)--(l42);

	\draw[->, double, color=black, thick] (l31)--(l21);
	\draw[->, color=black, thick] (l21)--(l32);

	\draw[->, double, color=black, thick] (l11)--(l21);
\end{tikzpicture}
\caption{{\footnotesize\textbf{Gelfand--Cejtlin arrays}.  These illustrations show the relationships between the $k$-tuples $\bn$ and $\bm$, and the arrays $\l = \{\l_j^{(i)}\}$ that parametrize the orthonormal basis $\{u(\l)\}$ for the representation $\HH_{\bm}$ of $\SO(\en)$.  On the left we have $\en=2k$ and on the right we have $\en=2k+1$.  The relation $a \leq b$ is denoted with a straight arrow $a \longrightarrow b$, and $\abs{a} \leq b$ is denoted with a double arrow $a \implies b$.  The heights $\ceil{\bn}$, $\ceil{\bm}$, and $\ceil{\l}$ are also indicated.}}
\label{figarray}
\end{figure}

%=================================================

\subsection{Bases for $\HH_{\bm}$ and $\HH_{{\bn},\nu}$ when $\en=2k$}\label{basiseven}
There is an orthonormal basis $\{u(\l)\}$ for $\HH_{\bm}$ parametrized by the set $\L_{\bm}$ of arrays of integers $\l = \{\l_j^{(i)}\mid i = 1,\dots, 2k-2, \textrm{ and } j=1,\dots \ceil{i/2}\}$ satisfying the inequalities indicated in Figure~\ref{figarray}.  These are called Gelfand--Cejtlin arrays; they were introduced in~\cite{GC50}.  We denote $\ceil{\l}:=\l_{k-1}^{(2k-2)}$.  This is just the maximum norm of $\l$ when viewed as an element of some $\ZZ^q$.

One then gets an orthonormal basis $\{u(\bm,\l)\}$ for $\HH_{{\bn},\nu}$ by letting $(\bm,\l)$ range over all the values allowed by~\eqref{mcondition} and the inequalities in Figure~\ref{figarray}.  The action of $X$ on elements of this basis is given by the formula
\begin{align}\label{xactioneven}
	Xu(\bm,\l) = \sum_{j=1}^{k} A_{j}^{-}(\bm,\l)u(\bm-e_j, \l) + \sum_{j=1}^{k} A_{j}^{+}(\bm,\l)u(\bm+e_j, \l),
\end{align}
where $e_j$ is the $j^{\textrm{th}}$ vector in the standard basis for $\ZZ^k$.  The coefficients $A_j^\pm$ satisfy $A_j^- (\bm,\l) = A_j^+ (\bm-e_j, \l)$ and are defined by 
\begin{align}\label{defA}
	A_j^+ (\bm,\l) &= \frac{1}{2}\sqrt{\frac{\prod_{r=1}^{k-1}\left[(x_r - \frac{1}{2})^2-(y_j + \frac{1}{2})^2\right]\left[z_r^2-(y_j + \frac{1}{2})^2\right]\cdot\left[\nu^2-(y_j + \frac{1}{2})^2\right]}{\prod_{r\neq j}(y_r^2 - y_j^2)[y_r^2 - (y_j + 1)^2]}}
\end{align}
where $x_r := \l_r^{(2k-2)} + r$, $y_j := m_j + (j-1)$, and $z_r := n_r + r -\frac{1}{2}$ when $m_j \neq m_{j+1}$, and $A_j^+ (\bm,\l)=0$ otherwise.

\begin{remark*}
We are really only interested in the asymptotic behavior of the coefficients~\eqref{defA}.  For example, $\Abs{A_j^\pm (\bm,\l)}$ decays to $0$ as $\ceil{\bm}\to\infty$, when $j=1,\dots, k-1$.  On the other hand, Lemma~\ref{coeffseven} shows that $\Abs{A_k (\bm,\l)}$ diverges polynomially.  The specific asymptotic properties we will need are proved through computations in Appendices~\ref{Akbound} and~\ref{abound}.
\end{remark*}

\begin{remark*}
In view of Proposition~\ref{minvelements}, many of the calculations involving the bases $\{u(\bm,\l)\}$ will be accompanied by corresponding calculations where $\l\equiv 0$.  These of course turn out to be simpler than their more general counterparts, and yield stronger results for the cases to which they apply, namely, representations that admit $\mathbb{M}(\en)$-invariant elements.
\end{remark*}
%=================================================

%=============================================================

\subsection{Bases for $\HH_{\bm}$ and $\HH_{{\bn},\nu}$ when $\en=2k+1$}\label{basisodd}

The orthonormal basis $\{u(\l)\}$ for $\HH_{\bm}$ is parametrized by the set $\L_{\bm}$ of arrays of integers $\l=\{\l_j^{(i)}\mid i=1,\dots,2k-1,\textrm{ and }j=1,\dots,\ceil{i/2}\}$ of the form illustrated in Figure~\ref{figarray}.  As above, one gets an orthonormal basis $\{u(\bm,\l)\}$ for $\HH_{{\bn},\nu}$ by letting $(\bm,\l)$ range over all the possible values.  The action of $X$ on elements of this basis is given by the formula
\begin{align}\label{xactionodd}
	Xu(\bm,\l) = \sum_{j=1}^{k} B_{j}^{-}(\bm,\l)u(\bm-e_j,\l) + C(\bm,\l)u(\l) + \sum_{j=1}^{k} B_{j}^{+}(\bm,\l)u(\bm+e_j,\l),
\end{align}
where the coefficients $B_j^\pm$ satisfy $B_j^- (\bm, \l) = B_j^+ (\bm-e_j, \l)$ and are defined by 
\begin{align}\label{defB}
	B_j^+ (\bm,\l) &= \sqrt{\frac{\prod_{r=1}^{k}(x_{r}^2 - y_{j}^2)\prod_{r=1}^{k}(z_r^2 - y_{j}^2)\cdot (\nu^2 - y_{j}^2)}{y_{j}^2(4y_{j}^2-1)\prod_{r\neq j}(y_{r}^2-y_{j}^2)[(y_{r}-1)^2-y_{j}^2]}}
\end{align}
and
\begin{align}\label{defC}
	C(\bm,\l) &= \frac{\left(\prod_{r=1}^{k}x_{r}\right)\left(\prod_{r=1}^{k}z_r\right)  \nu}{\prod_{r=1}^{k}y_{r}(y_{r}-1)}
\end{align}
with $x_{r}:= \l_r^{(2k-1)} + (r-1)$, $y_{j}:= m_{j} + j$, and $z_r := n_r + r -1$.

As in the case of even $\en$, we only need the asymptotic behavior of these coefficients.  For this, we will use Lemmas~\ref{coeffsodd} and~\ref{betas}, proved in Appendices~\ref{Bkbound} and~\ref{bbound}.

%==========================================================

%==========================================================

\section{Partial difference equations and formal solutions to co\-boundary equations}

On one hand it is natural to write an element $f \in \HH_{{\bn},\nu}$ as 
\begin{align}\label{summl}
	f &= \sum_{\bm}\sum_{\l \in \L_{\bm}} f(\bm,\l) u(\bm,\l)
\end{align}
where $\L_{\bm}$ denotes the set of diagrams $\l$ that can appear ``under'' $\bm$.  Essentially, the first sum is over the different representations $\HH_{\bm}$ of the maximal compact subgroup that appear as sub-representations of $\HH_{{\bn},\nu}$, and the second sum is over the basis vectors in each $\HH_{\bm}$.  

However, it will be more convenient for us to change the order of summation.  Let $\L = \cup_{\bm}\L_{\bm}$ be the set of all diagrams $\l$ that are possible in the representation $\HH_{{\bn},\nu}$.   For $\l \in \L$, let $M_\l = \{\bm\mid\l\in\L_{\bm} \}$ be the set of $\bm\in M_{\bn}$ that ``allow'' $\l$.  We can then re-write~\eqref{summl} as
\begin{align}\label{sumlm}
	f &= \sum_{\l \in \L} \sum_{\bm\in M_\l} f(\bm,\l) u(\bm,\l).
\end{align}
The advantage of~\eqref{sumlm} is that for a fixed $\l \in \L$, the coboundary equation $Xg=f$ becomes a partial difference equation (PdE) on the one-ended rectangular cylinder $M_\l \subset M_{\bn}\subset \ZZ^k$; the expressions~\eqref{xactioneven} and~\eqref{xactionodd} show that $X$ acts by changing the $k$-tuple $\bm$, while leaving the diagram $\l$ untouched.

%===============================================
\subsection{The PdE}

Let us fix $\l \in \L$ and drop it temporarily from the notation, writing $f(\bm):=f(\bm,\l)$, and so on.  Then the coboundary equation $Xg=f$ becomes
\begin{align*}
	&\sum_{\bm \in M_{\l}} f(\bm)u(\bm) = X\left(\sum_{\bm \in M_{\l}} g(\bm)u(\bm)\right) \\
		&= \begin{cases} \displaystyle{\sum_{\bm \in M_{\l}} g(\bm)\left[\sum_{j=1}^{k} A_{j}^{+}(\bm)u(\bm+e_j) + \sum_{j=1}^{k} A_{j}^{-}(\bm)u(\bm-e_j)\right]}\\
			\displaystyle{\sum_{\bm\in M_\l} g(\bm)\left[\sum_{j=1}^{k}B_j^+ (\bm) u(\bm+e_j) + C(\bm)u(\bm) + \sum_{j=1}^{k}B_j^- (\bm) u(\bm-e_j)\right]}
		\end{cases}
\end{align*}
by~\eqref{xactioneven} and~\eqref{xactionodd}.  (The first line corresponds to the case $\en=2k$, and the second corresponds to $\en=2k+1$.)  Collecting terms, and recalling that $A_j^- (\bm) = A_j^+ (\bm-e_j)$ and $B_j^- (\bm) = B_j^+ (\bm-e_j)$, we have
\begin{align}\label{pde}
\begin{cases}
	\displaystyle{f(\bm) = \sum_{j=1}^{k}\left[A_j^+ (\bm) g(\bm+e_j) + A_j^- (\bm) g(\bm-e_j)\right]},&\textrm{ if }\en=2k,\textrm{ and} \\
	\displaystyle{f(\bm) = \sum_{j=1}^{k}\left[B_j^+ (\bm) g(\bm+e_j) + B_j^- (\bm) g(\bm-e_j)\right] + C(\bm)g(\bm)},&\textrm{ if }\en=2k+1.
\end{cases}
\end{align}
Expression~\eqref{pde} is the partial difference equation corresponding to $Xg=f$ with a fixed $\l$.

%===============================================

%===============================================

\subsection{Formal solutions}\label{formalsection}

Since the PdEs corresponding to the coboundary equation $Xg=f$ in irreducible representations of $\SO^\circ(\en,1)$ are different depending on whether $\en$ is even or odd, we must treat the two cases separately.  However, we have made an effort to choose notation that allows the solutions to be written almost identically in the two situations.
%===============================================

\subsubsection{Formal solution when $\en=2k$}
\begin{definition}\label{defeven}
We set the following definitions: 
\begin{itemize}
\item[(\emph{i})] For $\bm_1,\bm_2 \in M_\l$, we write 
\[
	\bm_1 \leq \bm_2 \iff \abs{m_1^{(2)}-m_1^{(1)}}+\dots+\abs{m_{k-1}^{(2)}-m_{k-1}^{(1)}} \leq m_k^{(2)} - m_k^{(1)}.
\]
More intuitively, one may say that $\bm_1 \leq \bm_2$ means that $\bm_2$ lies in the conical region ``above" $\bm_1$.  We occasionally write $\bm_2 \geq \bm_1$, meaning exactly the same.  We may also write $\bm_1 < \bm_2$, meaning that the above inequality holds strictly.  (See Figure~\ref{fig1}.)
\item[(\emph{ii})] For $\bm,{\bk} \in M_\l$, let $\calP(\bm,{\bk})$ be the set of paths 
\[
	p=\left\{\bm=p(1),p(2),\dots,p(l+1)={\bk} \right\} \subset M_\l,\quad l = \abs{p} = \textrm{length of path}
\]
beginning at $\bm$ and ending at ${\bk}$ such that for every $s=1,\dots,l$, 
\[
	h(s):= p(s+1) - p(s) = \begin{cases} e_k \pm e_j  & \textrm{with }j=1,\dots,k-1, \textrm{ or} \\  2e_k.\end{cases}
\]
We will call these \emph{cocubic paths} because each step is from one vertex of a cocube to another.  We emphasize that the set $\calP(\bm,{\bk})$ may be empty.  (See Figure~\ref{fig1}.)
\item[(\emph{iii})] For $p \in \calP(\bm,{\bk})$ and $s = 1,\dots,\abs{p}$, let
\[
	\a_{h(s)}(p(s)) = \begin{cases} \frac{A_j^{\pm}(p(s))}{A_k^-(p(s+1))} &\textrm{if}\quad h(s)=e_k \pm e_j\\
										\frac{A_k^{+}(p(s))}{A_k^-(p(s+1))} &\textrm{if}\quad h(s) = 2e_k.
						\end{cases}
\]
\item[(\emph{iv})] For $\bm,{\bk} \in M_\l$, let 
\begin{align*}
	\D^{\bm}({\bk}) := \D^{\bm,\l}({\bk},\l) &= \sum_{p \in \calP(\bm,{\bk})} (-1)^{\abs{p}} \prod_{s=1}^{\abs{p}}\a_{h(s)}(p(s)).
\end{align*}
Empty sums are by convention equal to $0$.  Similarly, empty products are $1$.  Therefore, $\D^{\bm}({\bk})=0$ whenever $\calP(\bm,{\bk})$ is the empty set (say, if $\bm\nleq{\bk}$), and $\D^{\bm}(\bm)=1$, where we interpret $\calP(\bm,\bm)$ as the singleton set, consisting of the path $\{\bm=p(1)\}$ of length $0$.
\end{itemize}
\end{definition}

%=========================

\begin{figure}[tb]
\center{\begin{tikzpicture}[>=latex]
%  \draw[line width=0.3cm,color=gray!30,cap=round,join=round] (0,0)--(-2,2)--(-2,7.5);
%  \draw[line width=0.3cm,color=gray!30,cap=round,join=round] (0,0)--(4,4)--(4,7.5);
  \draw[help lines, densely dotted] (-2,-1) grid (4,7.5);
  \draw[->, shorten >=2pt, shorten <=2pt, ultra thick] (0,0)--(-1,1);
  \draw[->, shorten >=2pt, shorten <=2pt, ultra thick] (-1,1)--(0,2);
  \draw[->, shorten >=2pt, shorten <=2pt, ultra thick] (0,2)--(0,4);
  \draw[->, shorten >=2pt, shorten <=2pt, ultra thick] (0,4)--(1,5);
  \draw[->, shorten >=3pt, shorten <=2pt,  ultra thick] (1,5)--(2,6);
  \draw[->, shorten >=2pt, shorten <=2pt, ultra thick, dashed] (0,0)--(1,1);
  \draw[->, shorten >=2pt, shorten <=2pt, ultra thick, dashed] (1,1)--(1,3);
  \draw[->, shorten >=2pt, shorten <=2pt, ultra thick, dashed] (1,3)--(2,4);
  \draw[->, shorten >=2pt, shorten <=2pt, ultra thick, dashed] (2,4)--(3,5);
  \draw[->, shorten >=3pt, shorten <=2pt, ultra thick, dashed] (3,5)--(2,6);
  \draw[line width=5pt, cap=round] (-2,-1)--(4,-1);
  \draw[->, thick] (-2,-1)--(-2,7.5);
  \draw[->, thick] (4,-1)--(4,7.5);
  \draw[below left] (0,0) node(m){$\bm$};
  \draw[above right] (2,6) node(k){$\bk$};
  \fill (0,0) circle (0.1cm) (2,6) circle (0.1cm);
  \fill (-1,1) circle (0.05cm) (1,1) circle (0.05cm)
  		(-2,2) circle (0.05cm) (0,2) circle (0.05cm) (2,2) circle (0.05cm)
		(-1,3) circle (0.05cm) (1,3) circle (0.05cm)
		(3,3) circle (0.05cm) 
		(-2,4) circle (0.05cm) (0,4) circle (0.05cm) (2,4) circle (0.05cm) (4,4) circle (0.05cm)
		(-1,5) circle (0.05cm) (1,5) circle (0.05cm) (3,5) circle (0.05cm)
		(-2,6) circle (0.05cm) (0,6) circle (0.05cm) (2,6) circle (0.05cm) (4,6) circle (0.05cm)
		(-1,7) circle (0.05cm) (1,7) circle (0.05cm) (3,7) circle (0.05cm);
%  \draw[->,rounded corners=0.2cm,shorten >=2pt]
%   (1.5,0.5)-- ++(0,-1)-- ++(1,0)-- ++(0,2)-- ++(-1,0)-- ++(0,2)-- ++(1,0)--
%    ++(0,1)-- ++(-1,0)-- ++(0,-1)-- ++(-2,0)-- ++(0,3)-- ++(2,0)-- ++(0,-1)--
%    ++(1,0)-- ++(0,1)-- ++(1,0)-- ++(0,-1)-- ++(1,0)-- ++(0,-3)-- ++(-2,0)--
%    ++(1,0)-- ++(0,-3)-- ++(1,0)-- ++(0,-1)-- ++(-6,0)-- ++(0,3)-- ++(2,0)--
%    ++(0,-1)-- ++(1,0);
\end{tikzpicture}
\qquad
\begin{tikzpicture}[>=latex]
%  \draw[line width=0.3cm,color=gray!30,cap=round,join=round] (0,0)--(-2,2)--(-2,7.5);
%  \draw[line width=0.3cm,color=gray!30,cap=round,join=round] (0,0)--(4,4)--(4,7.5);
  \draw[help lines, densely dotted] (-2,-1) grid (4,7.5);
  \draw[->, shorten >=2pt, shorten <=2pt, ultra thick] (0,0)--(0,1);
  \draw[->, shorten >=2pt, shorten <=2pt, ultra thick] (0,1)--(-1,2);
  \draw[->, shorten >=2pt, shorten <=2pt, ultra thick] (-1,2)--(0,3);
  \draw[->, shorten >=2pt, shorten <=2pt, ultra thick] (0,3)--(1,4);
  \draw[->, shorten >=2pt, shorten <=2pt, ultra thick] (1,4)--(0,5);
  \draw[->, shorten >=3pt, shorten <=2pt,  ultra thick] (0,5)--(1,6);
  \draw[->, shorten >=2pt, shorten <=2pt, ultra thick, dashed] (0,0)--(1,1);
  \draw[->, shorten >=2pt, shorten <=2pt, ultra thick, dashed] (1,1)--(1,2);
  \draw[->, shorten >=2pt, shorten <=2pt, ultra thick, dashed] (1,2)--(2,3);
  \draw[shorten >=3pt, shorten <=2pt, ultra thick, dashed] (2,3)--(2,4);
  \draw[->, shorten >=2pt, shorten <=2pt, ultra thick, dashed] (2,4)--(2,5);
  \draw[->, shorten >=3pt, shorten <=2pt, ultra thick, dashed] (2,5)--(1,6);
  \draw[line width=5pt, cap=round] (-2,-1)--(4,-1);
  \draw[->, thick] (-2,-1)--(-2,7.5);
  \draw[->, thick] (4,-1)--(4,7.5);
  \draw[below left] (0,0) node(m){$\bm$};
  \draw[above right] (1,6) node(k){$\bk$};
  \fill (0,0) circle (0.1cm) (1,6) circle (0.1cm);
  \fill (-1,1) circle (0.05cm) (0,1) circle (0.05cm) (1,1) circle (0.05cm)
  		(-2,2) circle (0.05cm) (-1,2) circle (0.05cm) (0,2) circle (0.05cm) (1,2) circle (0.05cm) (2,2) circle (0.05cm)
		(-2,3) circle (0.05cm) (-1,3) circle (0.05cm) (0,3) circle (0.05cm) (1,3) circle (0.05cm)
		(2,3) circle (0.05cm) (3,3) circle (0.05cm) 
		(-2,4) circle (0.05cm) (-1,4) circle (0.05cm) (0,4) circle (0.05cm) (1,4) circle (0.05cm) (2,4) circle (0.05cm) (3,4) circle (0.05cm) (4,4) circle (0.05cm)
		(-2,5) circle (0.05cm) (-1,5) circle (0.05cm) (0,5) circle (0.05cm) (1,5) circle (0.05cm) (2,5) circle (0.05cm) (3,5) circle (0.05cm) (4,5) circle (0.05cm)
		(-2,6) circle (0.05cm) (-1,6) circle (0.05cm) (0,6) circle (0.05cm) (1,6) circle (0.05cm) (2,6) circle (0.05cm) (3,6) circle (0.05cm) (4,6) circle (0.05cm)
		(-2,7) circle (0.05cm) (-1,7) circle (0.05cm) (0,7) circle (0.05cm) (1,7) circle (0.05cm) (2,7) circle (0.05cm) (3,7) circle (0.05cm) (4,7) circle (0.05cm);
%  \draw[->,rounded corners=0.2cm,shorten >=2pt]
%   (1.5,0.5)-- ++(0,-1)-- ++(1,0)-- ++(0,2)-- ++(-1,0)-- ++(0,2)-- ++(1,0)--
%    ++(0,1)-- ++(-1,0)-- ++(0,-1)-- ++(-2,0)-- ++(0,3)-- ++(2,0)-- ++(0,-1)--
%    ++(1,0)-- ++(0,1)-- ++(1,0)-- ++(0,-1)-- ++(1,0)-- ++(0,-3)-- ++(-2,0)--
%    ++(1,0)-- ++(0,-3)-- ++(1,0)-- ++(0,-1)-- ++(-6,0)-- ++(0,3)-- ++(2,0)--
%    ++(0,-1)-- ++(1,0);
\end{tikzpicture}}
\caption{{\footnotesize\textbf{Illustrations of $M_\l$}. \emph{On the left}: $\en=4$; two paths in $\calP(\bm,\bk)$ are drawn, both of length $5$; points accessible from $\bm$ through cocubic paths, that is, points with $\calP(\bm,\bullet)\neq\emptyset$, are indicated.  \emph{On the right}: $\en=5$; two paths in $\widetilde\calP(\bm,\bk)$ are drawn, one of length $6$ and the other of length $5$; points with $\widetilde\calP(\bm,\bullet)\neq\emptyset$ are indicated, and in this case, they coincide with points $\{\bullet \geq \bm\}$.  In both pictures, the floor $\floor{M_{\l}}$ is drawn as a thick horizontal line segment.}}
\label{fig1}
\end{figure}

%===========

We now find a formal solution to the partial difference equation~\eqref{pde} by na{\"i}vely solving ``upward" in the $e_k$-direction.  That is, we re-write~\eqref{pde} as
\[
	g(\bm-e_k)= \frac{1}{A_k^- (\bm)}\left[f(\bm) - \sum_{j=1}^{k} A_j^+ (\bm) g(\bm+e_j) + \sum_{j=1}^{k-1}A_j^- (\bm) g(\bm-e_j)\right],
\]
and use this expression successively to write $g(\bm-e_k)$ in terms of all the $f({\bk})$ with $\calP(\bm,{\bk})$ non-empty, arriving at
\begin{align}\label{fseven}
	g(\bm-e_k) &= \sum_{{\bk}\in M_\l}\frac{f({\bk})}{A_k^-(\bm)} \D^{\bm}({\bk}).
\end{align}
\begin{remark*}
Notice that we can take the sum over all ${\bk} \in M_\l$, rather than just ${\bk}$'s connected to $\bm$ by a cocubic path, because $\D^{\bm}({\bk})=0$ for points ${\bk}$ with no such paths.
\end{remark*}

Equation~\eqref{fseven} defines a formal solution to the PdE~\eqref{pde} at all points in $\bm \in M_\l$ with $\ceil{\bm} \geq m_\l +1$, where $m_\l = \min\{\ceil{\bm} \mid \bm \in M_\l\} = \max\{\ceil{\bn},\ceil{\l}\}$ is the height of the ``floor'' of the rectangular cylinder $M_\l$, as determined by $\bn$ and $\l$.  We denote this floor by $\floor{M_\l} = \{\bm\in M_\l \mid \ceil{\bm} = m_\l\}$.  (See Figure~\ref{fig1}.)

\begin{lemma}\label{lemmafixedl}
As defined in~\eqref{fseven}, $g$ constitutes a formal solution to the partial difference equation~\eqref{pde} if and only if 
\begin{align}\label{conditioneven}
	\D^{\bm}(f):= \sum_{{\bk}\in M_\l}f({\bk})\D^{\bm}({\bk}) = 0
\end{align}
for all $\bm\in \floor{M_\l}$.
\end{lemma}

\begin{proof}
	By construction, $g$ solves the PdE~\eqref{pde} for all $f(\bm)$ with $\bm \in M_\l \backslash \floor{M_\l}$.  We need only check that~\eqref{pde} is satisfied for $\bm\in \floor{M_\l}$.  Routine algebraic manipulation of~\eqref{pde} shows that this is exactly equivalent to the condition~\eqref{conditioneven}.
\end{proof}

%===============================================

\subsubsection{Formal solution when $\en=2k+1$}

\begin{definition}\label{defodd}
We set the following definitions:
\begin{itemize}
\item[(\emph{i})] For $\bm,{\bk} \in M_\l$, let $\widetilde\calP(\bm,{\bk})$ be the set of paths 
\[
	p=\left\{\bm=p(1),p(2),\dots,p(l+1)={\bk} \right\}\subset M_\l,\quad l = \abs{p} = \textrm{length of path}
\]
beginning at $\bm$ and ending at ${\bk}$ such that for every $s=1,\dots,l$, 
\[
	h(s):= p(s+1) - p(s) = \begin{cases} e_k \pm e_j  & \textrm{with }j=1,\dots,k-1,\\ e_k,  \textrm{ or} \\  2e_k.\end{cases}
\]
We call these \emph{clipped cocubic paths} because the upward direction may be ``clipped."  Notice that $\widetilde\calP(\bm,{\bk})$ is non-empty if and only if $\bm\leq{\bk}$.  (See Definition~\ref{defeven}(\emph{i}) and Figure~\ref{fig1}).
\item[(\emph{ii})] For $p \in \widetilde\calP(\bm,{\bk})$ and $s = 1,\dots,\abs{p}$, let
\[
	\b_{h(s)}(p(s)) = \begin{cases} \frac{B_j^{\pm}(p(s))}{B_k^-(p(s+1))} &\textrm{if}\quad h(s)=e_k \pm e_j\\
										\frac{C(p(s))}{B_k^-(p(s+1))} &\textrm{if}\quad h(s) = e_k \\
										\frac{B_k^{+}(p(s))}{B_k^-(p(s+1))} &\textrm{if}\quad h(s) = 2e_k.
						\end{cases}
\]
\item[(\emph{iii})] For $\bm,{\bk} \in M_\l$, let 
\begin{align*}
	\D^{\bm}({\bk}) := \D^{\bm,\l}({\bk},\l) &= \sum_{p \in \widetilde\calP(\bm,{\bk})} (-1)^{\abs{p}} \prod_{s=1}^{\abs{p}}\b_{h(s)}(p(s)).
\end{align*}
We trust that no confusion is introduced by using the same notation (namely, $\D^{\bm,\l}$) here as when $\en$ is even.  
\end{itemize}
\end{definition}

We can now solve~\eqref{pde} as we did in the case of even $\en$: by solving ``upward'' in the $e_k$-direction to obtain the formal solution
\begin{align}\label{fsodd}
	g(\bm-e_k) &= \sum_{{\bk}\in M_\l}\frac{f({\bk})}{B_k^-(\bm)} \D^{\bm}({\bk}).
\end{align}

\begin{remark*}
Again, it is worth pointing out that the sum in~\eqref{fsodd} may be taken over all $\bk \in M_\l$ because $\D^{\bm} (\bk) = 0$ whenever there is no clipped cocubic path connecting $\bm$ and $\bk$.
\end{remark*}

We are then led to the following lemma.

\begin{lemma}\label{lemmafixedlodd}
As defined in~\eqref{fsodd}, $g$ constitutes a formal solution to the partial difference equation~\eqref{pde} if and only if 
\begin{align}\label{conditionodd}
	\D^{\bm}(f):= \sum_{{\bk}\in M_\l}f({\bk})\D^{\bm}({\bk}) = 0
\end{align}
for all $\bm\in \floor{M_\l}$.
\end{lemma}

\begin{proof}
The proof is identical to that of Lemma~\ref{lemmafixedl}.
\end{proof}

Finally, we re-introduce $\l$ into the notation.  The following proposition summarizes what we have found.

\begin{proposition}\label{propall}
The sum
\[
	g = \sum_{\l\in\L}\sum_{\bm\in M_\l} g(\bm,\l)u(\bm,\l)
\]
defined by
\begin{align*}
g(\bm-e_k,\l) =
\begin{cases}
	\displaystyle{\sum_{{\bk}\in M_\l}\frac{f({\bk,\l})}{A_k^-(\bm,\l)} \D^{\bm,\l}({\bk},\l)}, &\textrm{if}\quad \en=2k, \textrm{ and} \\
	\displaystyle{\sum_{{\bk}\in M_\l}\frac{f({\bk,\l})}{B_k^-(\bm,\l)} \D^{\bm,\l}({\bk},\l)}, &\textrm{if}\quad \en=2k+1
\end{cases}
\end{align*}
is a formal solution to the coboundary equation $Xg=f$ if and only if 
\begin{align*}
	\D^{\bm,\l}(f):= \sum_{\k\in\L}\sum_{{\bk}\in M_\k}f({\bk},\k)\D^{\bm,\l}({\bk},\k) = 0
\end{align*}
for all $(\bm,\l)$ such that $\bm\in \floor{M_\l} \subset M_\l$.
\end{proposition}

\begin{proof}
The proof is immediate from Lemmas~\ref{lemmafixedl} and~\ref{lemmafixedlodd} and the preceding discussion.
\end{proof}

\begin{remark*}
Let it be clear that $\D^{\bm,\l}({\bk},\k) = 0$ whenever $\k \neq \l$.  It may also be worth foreshadowing that the $\D^{\bm,\l}$'s with $\bm \in \floor{M_\l}$ will turn out to form a spanning set for the $X$-invariant distributions that are called for in Theorem~\ref{cobgeodflow}.  Therefore, Proposition~\ref{propall} serves as a ``formal'' version of Theorem~\ref{cobgeodflow}.
\end{remark*}

The next few sections are devoted to showing that the formal sums for $g$ and $\D$ in Proposition~\ref{propall} actually define \emph{bona fide} elements of Sobolev spaces and their duals.
%================================================

%===============================================

\section{Casimir operators, Laplace operators, spectral gap, and Sobolev norms}\label{sobnorms}

In this section we present some important differential operators coming from the universal enveloping algebra of $\so(\en,1)$, and we describe how they act upon irreducible unitary representations of $\SO^\circ(\en,1)$. We also define the Sobolev norms we will be using.

\subsection{Casimir operators and Laplace operators}\label{specialops}

The Casimir operator for $\SO^\circ(\en,1)$ is
\[
	\Box := \Box_G := -\sum_{i=1}^{\en}X_i^2 +\sum_{1\leq i<j \leq \en}\Theta_{i,j}^2.
\]
It lies in the center of the universal enveloping algebra of $\so(\en,1)$, and therefore acts as some scalar $\mu:=\mu(\bn,\nu)$ in any irreducible unitary representation $\HH_{{\bn},\nu}$.  In fact, it is shown in \cite{Thi73, Thi74} that 
\begin{align}
	\mu(\bn,\nu) &= -\nu^2 + \left(\frac{\en-1}{2}\right)^2 - \langle \bn, \bn + 2\brho_{M}\rangle \nonumber \\
		&=  \begin{cases} \displaystyle{-\nu^2 + \left(\frac{\en-1}{2}\right)^2 - \sum_{i=1}^{k-1}n_i (n_i + 2i-1)} &\textrm{for }\en=2k \\ 
							\displaystyle{-\nu^2 + \left(\frac{\en-1}{2}\right)^2 - \sum_{i=1}^{k}n_i (n_i + 2i-2)} &\textrm{for }\en=2k+1
			\end{cases} \label{muscalar}
\end{align}
where $\brho_M$ is the half-sum of positive roots of $M\cong \SO(\en-1)$.

Similarly, the Casimir operator 
\[
	\Box_K := \sum_{1\leq i<j \leq \en}\Theta_{i,j}^2
\]
acts as a multiplicative scalar in any irreducible unitary representation $\HH_{\bm}$ of $\SO(\en)$.  It is known that this scalar is
\begin{align}
	\langle \bm, \bm + 2\brho \rangle &= -\sum_{i=1}^{k} (m_i^2 + 2m_i \rho_i) := -q(\bm) \nonumber \\
		&= \begin{cases}
				\displaystyle{-\sum_{i=1}^{k}m_i(m_i + 2i-2)} &\textrm{for }\en=2k \\
				\displaystyle{-\sum_{i=1}^{k}m_i (m_i +2i -1)} &\textrm{for }\en=2k+1
			\end{cases} \label{mscalar}
\end{align}
where $\brho=(\rho_1,\dots,\rho_k)$ is the half-sum of the positive roots of $\so(\en)$.  (See, for example, \cite{Rac65}.)

The Laplace operator $\Delta$ is then  $\Delta = \Box - 2\Box_K$.  Since both $\Box$ and $\Box_K$ commute with the subgroups~$\mathbb{K}(\en)$ and $\mathbb{M}(\en)$ of $\SO^\circ(\en,1)$, so does $\Delta$.  This allows us to use the same operator~$\Delta$ to define the Laplace operator and Sobolev norms (see Section~\ref{sec:sobnorms}) on~$M\backslash G/\G$ and~$K\backslash G/\G$, after making a standard identification between the space $\Lii(M\backslash G/\G)$ (respectively $\Lii(K\backslash G/\G)$) and the subspace $\Lii(G/\G)^M$ (respectively $\Lii(G/\G)^K$) of $M$-invariant (respectively $K$-invariant) elements of $\Lii(G/\G)$. 

%======================================================================

\subsection{Spectral gap}\label{spectralgapsection}

\emph{Spectral gap} is an important feature of the results in~\cite{FF} and~\cite{M2}.  There, the spectrum in question is that of the Casimir operator $\Box$ for $\Sl(2,\RR)$.  The unitary dual of $\PSL(2,\RR) \cong \SO^\circ (2,1)$ is parametrized by the scalar, denoted $\mu$, by which the Casimir operator acts in an irreducible unitary representation.  It is related to the parameter $\nu$ by $\nu^2 = 1/4 - \mu$.  Therefore, spectral gap for the Casimir operator in a unitary representation $\pi:\PSL(2,\RR)\to\U(\HH)$ ensures that $\nu$ is bounded away from the end-point of the complementary series representations (which corresponds to the trivial representation) when one considers the direct integral decomposition of $\HH$.  Putting it another way, spectral gap in $\HH$ ensures that $\pi_0:\PSL(2,\RR)\to\U(\HH_0)$ is isolated in the Fell topology from the trivial representation $\mathbf{1}_G$, where $\HH_0$ denotes the orthogonal complement in $\HH$ to the subspace where $\PSL(2,\RR)$ acts trivially.  

We also need that $\nu$ does not come ``too close'' to $\frac{\en-1}{2}$, or, to be more precise, we want to choose some $\nu_0 \in \left[0,\frac{\en-1}{2}\right)$ that is closer to $\frac{\en-1}{2}$ than any of the $\nu$'s appearing in the direct integral decomposition of the unitary representation with which we would like to work.  However, in Part~\ref{parti}, the \emph{only} non-irreducible representation we will work with is the \emph{left-regular} representation $\varrho: \SO^\circ(\en,1)\to\U(\Lii(\SO^\circ(\en,1)/\G))$, and it is already well-known (see for example~\cite[Lemma~$3$]{Bek98}) that the left-regular representation $\varrho_0$ of a semisimple Lie group $G$ on $\Lii_0 (G/\G)$ is isolated from $\mathbf{1}_G$ in the Fell topology, so the desired $\nu_0$ in our work comes, in some sense, for free---we do not actually have to \emph{state} a spectral gap assumption in our main theorems in order to obtain it.

If we \emph{were} to state a spectral gap assumption (say, if we were interested in stating our main theorems in terms of general unitary representations,\footnote{The reason we do not do this is discussed in Section~\ref{dependencesection}.} in the style of~\cite{FF},~\cite{M2}, and~\cite{Ramhc}), it might be more natural to do so in terms of the Laplace operator $\Delta$, which always has positive eigenvalues in irreducible unitary representations of $\SO^\circ(\en,1)$, than for the Casimir operator $\Box$, which for $\en\geq 3$ can act as a negative scalar of arbitrary size.  On the other hand, we should really see the spectral gap assumption as taking place on $K\backslash G/\G$, the locally symmetric space associated to $G/\G$ (where the images of $\Delta$ and $\Box$ coincide); the condition that  $\varrho_0$ be isolated from $\bone_G$ is equivalent to a spectral gap for the Laplacian $\Delta$ on $K\backslash G/\G$. In our case this can be seen from the formulas we have presented in Section~\ref{specialops}:  One can deduce from these that a $K$-invariant eigenvector for $\Delta$ in an irreducible representation $\HH_{\bn,\nu}$ must have eigenvalue $-\nu^2 + \left(\frac{\en-1}{2}\right)^2$.  (Henceforth, we denote this eigenvalue by $\tilde\nu$.)  From here the relationship between spectral gap for $\Delta$ on $K\backslash G/\G$ and the existence of $\nu_0$ is obvious.

\begin{definition}\label{spectralgapparameter}
We set the following terminology:~$\nu_0 \in \left[0,\frac{\en-1}{2}\right)$ is a \emph{spectral gap parameter} if it satisfies the condition that it is closer to $\frac{\en-1}{2}$ than any $\nu$ appearing in the direct decomposition of $\varrho_0 :G\to\U(\Lii_0(G/\G))$. It naturally corresponds through $\tilde\nu = -\nu^2 + \left(\frac{\en-1}{2}\right)^2$ to a spectral gap $\tilde\nu_0$ for the Laplacian $\Delta$ on $K\backslash G/\G$.
\end{definition}

In Part~\ref{partii} we will be dealing with representations of $G_1\times\dots\times G_d = \SO^\circ(\en_1,1)\times\dots\times\SO^\circ(\en_d,1)$. We will need the restriction to each factor of the regular representation on $\Lii((G_1\times\dots\times G_d)/\G)$ to have the same spectral gap property as above, where $\G\subset G_1\times\dots\times G_d$ is an irreducible lattice.  This is known to be generally the case, for example by combining work of Kleinbock and Margulis~\cite[Theorem~$1.12$]{KM} and Clozel~\cite{Clo03} to obtain the following extension of~\cite[Lemma~$3$]{Bek98}.

\begin{theorem}\label{spectralgapthm}
Let $G_1\times\dots\times G_d$ be a product of non-compact simple Lie groups, and $\G \subset G$ an irreducible lattice.  Then the restriction of $\Lii(G/\G)$ to every $G_i$ has a spectral gap.
\end{theorem}

%======================================================================

\subsection{Sobolev norms}\label{sec:sobnorms}
Let $G\to\U(\HH)$ be a unitary representation.  The Sobolev space $W^s (\HH)$, $s>0$, is defined to be the maximal domain of the operator $(\bone+\Delta)^{s/2}$, where $\bone$ is the identity operator and $\Delta$ is the Laplace operator coming from $G$. $W^s(\HH)$ is a Hilbert space with inner product $\inner{f}{g}_s := \inner{(\bone+\Delta)^s f}{g}_{\HH}$, and its dual is denoted $W^{-s}(\HH) \subset \E^{\prime}(\HH)$.  

In an irreducible representation $\SO^\circ(\en,1)\to\U(\HH_{{\bn},\nu})$, the Sobolev norm defined by the inner product $\inner{\cdot}{\cdot}_s$ can be readily computed:~for $f \in W^s(\HH_{\bn,\nu})$,
\begin{align*}
	\norm{f}_s^2 &= \norm{(I+\Delta)^{s/2}f}_{\HH_{{\bn},\nu}}^2 = \inner{(\bone+\Delta)^{s}f}{f}_{\HH_{{\bn},\nu}} \\
		&= \sum_{\l\in\L}\sum_{\bm\in M_\l} \abs{f(\bm,\l)}^2\inner{(\bone+\Delta)^{s}u(\bm,\l)}{u(\bm,\l)} \\
		&= \sum_{\l\in\L}\sum_{\bm\in M_\l} (1+\mu(\bn,\nu)+2q(\bm))^s \abs{f(\bm,\l)}^2\norm{u(\bm,\l)}^2,
\end{align*}
where $(1 + \mu(\bn,\nu) + 2q(\bm))$ is the scalar by which $(\bone+\Delta)$ acts on the basis element $u(\bm,\l)$.  It will be convenient to use the concise notation $(1 + Q_{\bn,\nu}(\bm))$ for this scalar.  We have the following lemma.

\begin{lemma}\label{polynomials}
For any irreducible representation $\HH_{\bn,\nu}$,
\[
1 + Q_{\bn,\nu}(\bm) \asymp_{\bn,\nu} 1 + \ceil{\bm}^2,
\]
whereas if we restrict attention to representations where $\bn=(0,\dots,0,\ceil{\bn})$, then we have
\[
1 + Q_{\bn,\nu}(\bm) \asymp 1 + \tilde\nu + 2\ceil{\bm}^2 - \ceil{\bn}^2
\]
where the asymptote $\asymp$ does not depend on $\bn,\nu$.
\end{lemma}

\begin{proof}
This is checked by combining~\eqref{muscalar} and~\eqref{mscalar} with the inequalities indicated in Figure~\ref{figarray}.
\end{proof}

\begin{remark*}
One can also prove a ``more precise'' statement in the general case, where the asymptote $\asymp_{\bn,\nu}$ does not depend on $\bn,\nu$.  However, we would not gain much from this because the only place where we use the first part of Lemma~\ref{polynomials} is in the proof of Theorem~\ref{cobgeodflow} in Section~\ref{cobgeodflowproof}, where there is already dependence on the parameters $\bn, \nu$, after applying Lemma~\ref{dmlbk}.  The second part of Lemma~\ref{polynomials} is used in Section~\ref{twoproof}, in the proof of Theorem~\ref{two}.
\end{remark*}

%===============================================

\section{Invariant distributions and proof of Theorem~\ref{invdistthm}}
We would like to describe the set
\begin{align*}
	\II_{X}^{s}(\HH_{{\bn},\nu}) &= \{\D\in W^{-s}(\HH_{{\bn},\nu})\mid \D(Xh)=0 \textrm{ for all } h \in \Cinf(\HH_{{\bn},\nu}) \}
\end{align*}
of $X$-invariant distributions of Sobolev order $s>0$.  A distribution $\D \in W^{-s}(\HH_{{\bn},\nu})$ is determined by its values $\D(\bm,\l):=\D(u(\bm,\l))$ on the orthonormal basis $\{u(\bm,\l)\}$, and it is $X$-invariant if and only if these values satisfy the partial difference equation
\begin{align}\label{distde}
\begin{cases}
\displaystyle{0 = \sum_{j=1}^{k}\left[A_j^+ (\bm,\l) \D(\bm+e_j,\l) + A_j^- (\bm,\l) \D(\bm-e_j,\l)\right]},&\textrm{ if } \en=2k, \textrm{ or}  \\
\displaystyle{0 = \sum_{j=1}^{k}\left[B_j^+ (\bm) \D(\bm+e_j) + B_j^- (\bm) \D(\bm-e_j)\right] + C(\bm)\D(\bm)},&\textrm{ if } \en=2k+1,
\end{cases}
\end{align}
as is easily seen from~\eqref{xactioneven} and~\eqref{xactionodd}.  

An $X$-invariant distribution is completely determined by its values on the floors $\floor{M_{\l}}$ of $\{M_\l\}_{\l\in\L}$.  In fact, if $\bm \in \floor{M_\l}$, then $\D^{\bm,\l}$ (see Definitions~\ref{defeven}(\emph{iv}) and~\ref{defodd}(\emph{iii})) determines the unique \emph{formal} invariant distribution taking the value $1$ at $(\bm,\l)$, and $0$ at all other ``floor'' points $({\bk},\k)$ with ${\bk}\in \floor{M_\k}$.  In order to show that $\D^{\bm,\l}$ is a \emph{bona fide} distribution, we will find an upper bound on its Sobolev order---a real number $s >0$ such that $\D^{\bm,\l}(f)$ converges for every $f \in W^{s}(\HH_{{\bn},\nu})$.  First, we have the following lemma, which holds for all $\D^{\bm,\l}$, without regard to whether or not $\bm\in \floor{M_\l}$.

\begin{lemma}\label{dmlbk}
For $\bm \in M_\l$, we have $\Abs{\D^{\bm,\l}(\bk,\l)} \ll_{\bn, \nu} 1$, and in the case that $\l\equiv 0$, we have $\Abs{\D^{\bm,0}(\bk,0)} \leq 1$ for all $\bk \in M_{0}$.
\end{lemma}

\begin{remark*}
Recall that $\D^{\bm,\l}(\bk,\l)=0$ if $\bk\ngeq\bm$, so this lemma only gives useful information when $\bk\geq\bm$.
\end{remark*}

\begin{proof}
We prove this for $\en=2k$, the proof for odd $\en$ being very similar.  For all $\bk > \bm$, we have that $\D^{\bm,\l}$ satisfies the PdE
\begin{multline*}
0 =\D^{\bm,\l}({\bk}+e_k,\l) + \a_{2e_k}(\bk-e_k)\D^{\bm,\l}({\bk}-e_k,\l) \\
+ \sum_{j=1}^{k-1}\left[\a_{e_k + e_j}(\bk - e_j)\D^{\bm,\l}(\bk - e_j, \l) + \a_{e_k - e_j}(\bk + e_j) \D^{\bm,\l}(\bk + e_j,\l)\right].
\end{multline*}
From the above PdE, we see that 
\begin{multline*}
\Abs{\D^{\bm,\l}({\bk}+e_k)} \leq \Abs{\a_{2e_k}(\bk-e_k)}\Abs{\D^{\bm,\l}({\bk}-e_k)} \\
	+ \sum_{j=1}^{k-1}\left(\Abs{\a_{e_k + e_j}(\bk - e_j)}\Abs{\D^{\bm,\l}(\bk - e_j)} + \Abs{\a_{e_k - e_j} (\bk + e_j)}\Abs{\D^{\bm,\l}(\bk + e_j)}\right).
\end{multline*}
There is some $P(\ceil{\bm},\ceil{\bk})>0$ such that $\Abs{\D(\boldsymbol{\tau})}\leq P(\ceil{\bm},\ceil{\bk})$ for all $\boldsymbol{\tau}\in M_\l$ with $\ceil{\boldsymbol{\tau}} \leq \ceil{\bk}$, and we can take $P(\ceil{\bm},\ceil{\bk})$ to increase as $\ceil{\bk}$ does.  So applying Lemma~\ref{alphas}, we have
\begin{align*}
	\Abs{\D^{\bm,\l}(\bk+e_k)} &\leq P(\ceil{\bm}, \ceil{\bk}-1) \\
	&\indent+ 2(k-1)C_{\nu}\,\frac{\min\{\ceil{\bn},\ceil{\l}\}}{\ceil{\bk}+1}\frac{\ceil{\bn}\ceil{\l}}{\left(\ceil{\bk}+1\right)^2 - m_{\l}^2}\,P(\ceil{\bm},\ceil{\bk}) \\
	&\leq P(\ceil{\bm},\ceil{\bk})\,\left[1 + 2(k-1)C_{\nu}\,\frac{\min\{\ceil{\bn},\ceil{\l}\}}{\ceil{\bk}+1}\frac{\ceil{\bn}\ceil{\l}}{\left(\ceil{\bk}+1\right)^2 - m_{\l}^2}\right]
\end{align*}
holding up to height $\ceil{\bk}$.  Let us then put $P(\ceil{\bm},\ceil{\boldsymbol{\t}})=0$ for all $\boldsymbol{\t}<\bm$, $P(\ceil{\bm},\ceil{\bm})=1$,
\begin{equation}\label{directly}
	P(\ceil{\bm},\ceil{\bm}+1)= C_\nu\,\frac{\min\{\ceil{\bn},\ceil{\l}\}}{\ceil{\bm}+1}\frac{\ceil{\bn}\ceil{\l}}{\left(\ceil{\bm}+1\right)^2-m_\l^2},
\end{equation}
and for $\ceil{\bk}>\ceil{\bm}$,
\[
P(\ceil{\bm},\ceil{\bk}+1) = \prod_{x=\ceil{\bm}+1}^{\ceil{\bk}}\left[1 + C_{\nu}\,\frac{\min\{\ceil{\bn},\ceil{\l}\}}{x+1} \frac{\ceil{\bn}\ceil{\l}}{\left(x+1\right)^2 - m_{\l}^2}\right],
\]
where we have absorbed the $2(k-1)$-factor into the constant $C_{\nu}$.  The logarithm is
\[
	\log P(\ceil{\bm},\ceil{\bk}+1) = \sum_{x=\ceil{\bm}+1}^{\ceil{\bk}}\log\left[1 + C_{\nu}\,\frac{\min\{\ceil{\bn},\ceil{\l}\}}{x+1}\frac{\ceil{\bn}\ceil{\l}}{\left(x+1\right)^2 - m_{\l}^2}\right]
\]
and the estimate $\log(1+x)\leq x$ for $x>0$ allows us to bound by
\[
\log P(\ceil{\bm},\ceil{\bk}+1)\leq\sum_{x=\ceil{\bm}+1}^{\ceil{\bk}}\left[C_{\nu}\,\frac{\min\{\ceil{\bn},\ceil{\l}\}}{x+1}\frac{\ceil{\bn}\ceil{\l}}{\left(x+1\right)^2 - m_{\l}^2}\right].
\]
We compare to an integral to obtain
\begin{align*}
\log P(\ceil{\bm},\ceil{\bk}+1)&\leq \left[C_{\nu}\,\frac{\min\{\ceil{\bn},\ceil{\l}\}\ceil{\bn}\ceil{\l}}{2m_\l^2}\log\left(\frac{(x+1)^2-m_{\l}^2}{(x+1)^2}\right) \right]_{\ceil{\bm}}^{\ceil{\bk}} \\
&\leq  C_{\nu}\,\frac{\min\{\ceil{\bn}^2,\ceil{\l}^2\}}{2m_\l}\log\left(\frac{(m_\l + 1)^2}{(m_\l +1)^2-m_\l^2}\right)
\end{align*}
After exponentiating, we have shown that 
\[
\Abs{\D^{\bm,\l}(\bk,\l)} \ll_{\nu} e^{\min\{\ceil{\bn}^2,\ceil{\l}^2\}}
\]
holds whenever $\ceil{\bk} \geq \ceil{\bm}+2$.  One sees that the same holds if $\ceil{\bk} = \ceil{\bm}+1$ by checking~\eqref{directly} directly.  

The second assertion is much easier because $\a_{e_k \pm e_j}(\bk,0)=0$. We have $\Abs{\D^{\bm,0}(\bm,0)}=1$ and 
\[
	\Abs{\D^{\bm,0}(\bm+ 2\ell \, e_k,0)} = \Abs{\a_{2e_k}(\bm+ (2\ell-2)e_k)}\Abs{\D^{\bm,0}(\bm+(2\ell-2)e_k,0)}
\]
for all $\ell \in\NN$, and $\D^{\bm,0}$ takes the value $0$ at all other points.  Lemma~\ref{alphas} immediately implies that $\Abs{\D^{\bm,0}(\bm+2\ell \,e_k,0)}\leq 1$.
\end{proof}

Now we can prove the following proposition, which shows that $\D^{\bm,\l}$ for $\bm \in \floor{M_\l}$ defines an element of $W^{-s}(\HH)$ for all $s > 1/2$.  Theorem~\ref{invdistthm} follows immediately.

\begin{proposition}\label{soborders}
For $(\bm,\l)$ with $\bm\in \floor{M_\l}$, $\D^{\bm,\l}$ is an $X$-invariant distribution of Sobolev order at most $1/2$.
\end{proposition}

\begin{proof}
That $\D^{\bm,\l}$ is $X$-invariant follows from the fact that it satisfies the difference equation~\eqref{distde} for all $\bk\in M_\l$.  We need only find an upper bound on its Sobolev order.  Let $f \in W^{s}(\HH_{{\bn},\nu})$.  Then 
\begin{align}
	\abs{\D^{\bm,\l}(f)}^2 &= \Abs{\sum_{\k\in\L}\sum_{{\bk}\in M_\k}f({\bk},\k)\D^{\bm,\l}({\bk},\k)}^2 \nonumber \\
		&\leq \sum_{\k\in\L}\sum_{{\bk}\in M_\k}(1+Q({\bk}))^s \abs{f({\bk},\k)}^2 \,\sum_{\k\in\L}\sum_{{\bk}\in M_\k}\frac{\abs{\D^{\bm,\l}({\bk},\k)}^2}{(1+Q({\bk}))^s} \nonumber \\
		&= \norm{f}_{s}^2 \, \sum_{{\bk}\in M_\l}\frac{\abs{\D^{\bm,\l}({\bk},\l)}^2}{(1+Q({\bk}))^s}. \nonumber
\intertext{By Lemma~\ref{dmlbk},}
		&\ll_{\bn,\nu} \norm{f}_{s}^2 \,  \sum_{{\bk}\in M_\l} \frac{1}{(1+Q({\bk}))^s}. \label{conv}
\end{align}
The expression~\eqref{conv} converges as long as $s > \frac{1}{2}$.
\end{proof}

%===============================================

\section{Sobolev norm of formal solution and proof of Theorem~\ref{cobgeodflow}}\label{cobgeodflowproof}

We are now prepared to state a proof of Theorem~\ref{cobgeodflow}.  The proof consists of computing the Sobolev norm of $g$, the formal solution (provided by Proposition~\ref{propall}) to the coboundary equation $Xg=f$ in a non-trivial irreducible unitary representation $\pi: \SO^\circ(\en,1)\to\U(\HH_{\bn,\nu})$.  

\begin{proof}[Proof of Theorem~\ref{cobgeodflow}]
Proposition~\ref{propall} supplies a formal solution $g$ to the coboundary equation $Xg=f$.  We compute the Sobolev norm of $g$.
\begin{align}
	\norm{g}_{t}^2 &= \sum_{\l\in\L}\sum_{\bm\in M_\l} (1+Q(\bm))^t\Abs{\sum_{{\bk}\in M_\l}\frac{f({\bk,\l})}{A_k^-(\bm+e_k,\l)} \D^{\bm+e_k,\l}({\bk},\l)}^2, \nonumber
\intertext{which, by the Cauchy--Schwartz inequality,}
		&\leq \sum_{\l\in\L}\left[\sum_{{\bk}\in M_\l}(1+Q({\bk}))^s\abs{f({\bk,\l})}^2 \sum_{\bm\in M_\l}\sum_{{\bk}\in M_\l}\frac{(1+Q(\bm))^t}{(1+Q({\bk}))^s }\,\Abs{\frac{\D^{\bm+e_k,\l}({\bk},\l)}{A_k^-(\bm+e_k,\l)}}^2\right]. \label{returning}
\end{align}
From Lemmas~\ref{dmlbk} and~\ref{coeffseven}, we get 
\begin{multline*}
	\sum_{\bm<{\bk}\in M_\l}\frac{(1+Q(\bm))^t}{(1+Q({\bk}))^s }\,\Abs{\frac{\D^{\bm+e_k,\l}({\bk},\l)}{A_k^+(\bm,\l)}}^2 \\
	\ll_{\bn,\nu} \sum_{\bm<{\bk}\in M_\l}\frac{(1+Q(\bm))^t}{(1+Q({\bk}))^s }\,\frac{1}{\left(\ceil{\bm} -\ceil{\bn}+ 1\right)\left(\ceil{\bm} -\ceil{\l}+ 1\right)}.
\end{multline*}
Applying Lemma~\ref{polynomials}, this is bounded by
\[
	\ll_{\bn, \nu} \sum_{\bm<{\bk}\in M_\l}\frac{(1+\ceil{\bm}^2)^t}{(1+\ceil{\bk}^2)^s }\,\frac{1}{\left(\ceil{\bm} -\ceil{\bn}+ 1\right)\left(\ceil{\bm} -\ceil{\l}+ 1\right)},
\]
which in turn is controlled by
\begin{align*}
	&\ll_{\bn, \nu} \sum_{\bm<{\bk}\in M_\l}\frac{(1+\ceil{\bm}^2)^t}{(1+ \ceil{\bk}^2)^s }\,\frac{1}{\left(\ceil{\bm} + 1\right)\left(\ceil{\bm} -m_\l +1\right)} \\
	&\ll_{\bn, \nu,s,t} \sum_{\bm \in M_\l}\left(1 + \ceil{\bm}\right)^{2t-1}\,\sum_{\bk > \bm}\frac{1}{(1+\ceil{\bk})^{2s }}.
\end{align*}
Now, for any height $\ceil{\bk}$ there are at most $\Abs{\floor{M_\l}}$ elements of $M_\l$ counted in the last sum, so we have
\begin{align*}
	\ll_{\bn, \nu,s,t} \Abs{\floor{M_\l}} \sum_{\bm \in M_\l}\left(1 + \ceil{\bm}\right)^{2t-1}\,\sum_{i = \ceil{\bm}+1}\frac{1}{(1+i)^{2s}},
\end{align*}
which converges as long as $s > \frac{1}{2}$; we bound it by the integral $\int_{\ceil{\bm}}^{\infty}\frac{dx}{(1+x)^{2s}}$ to arrive at
\begin{align*}
	\ll_{\bn,\nu,s,t} \Abs{\floor{M_\l}} \sum_{\bm \in M_\l}\frac{\left(1 + \ceil{\bm}\right)^{2t-1}}{(1+\ceil{\bm})^{2s -1}}.
\end{align*}
Arguing in the same way for this sum over $\ceil{\bm}$'s leads to
\begin{align*}
	\ll_{\bn, \nu,s,t} \Abs{\floor{M_\l}}^2 \sum_{i = m_\l}^{\infty} \frac{1}{\left(1+i\right)^{2s- 2t}}
\end{align*}
and again we bound by an integral,
\begin{align*}
	\ll_{\bn, \nu, s,t} \Abs{\floor{M_\l}}^2 \left[\frac{1}{\left(1+m_\l\right)^{2s - 2t}}+\int_{m_\l}^{\infty} \frac{dx}{\left(1+x\right)^{2s - 2t}}\right].
\end{align*}
This converges as long as $t < s - \frac{1}{2}$, giving
\begin{align*}
	\ll_{\bn, \nu,s,t} \Abs{\floor{M_\l}}^2 \left[\frac{1}{\left(1+m_\l\right)^{2s - 2t}}+ \frac{1}{\left(1+m_\l\right)^{2s- 2t  - 1 }}\right].
\end{align*}
The volume of the floor $\floor{M_\l}$ is easily bounded by $(m_\l + 1)^{k-1}$, so we are left with
\begin{align*}
	\ll_{\bn, \nu,s,t}\, \left[\frac{ (m_\l+1)^{2(k-1)}}{\left(m_\l+1\right)^{2s - 2t  - 1}}\right]
\end{align*}
and as long as $s-t \geq k - \frac{1}{2}$, the above expression is $\ll_{\bn,\nu,s,t} 1$.  Therefore, we have shown that
\begin{align*}
\norm{g}_{t}^2 &\ll_{\bn,\nu,s,t} \sum_{\l\in\L} \sum_{{\bk}\in M_\l}(1+Q({\bk}))^s\abs{f({\bk,\l})}^2 \\
	&\ll_{\bn,\nu,s,t}\, \norm{f}_s^2
\end{align*}
and the theorem is proved.
\end{proof}

\subsection{A brief discussion of the dependence of Theorem~\ref{cobgeodflow} on the irreducible unitary representation}\label{dependencesection}

The dependence of Theorem~\ref{cobgeodflow} on $\bn$ and $\nu$ stands in the way of our using it to prove a result analogous to Theorem~\ref{two} for the flow of $X$ on $\SO^\circ (\en,1)/\G$.  From our point of view, the main obstacle to relaxing this dependence is in Lemma~\ref{dmlbk}, where the constant $C_{\bn,\nu}>0$ (associated to the $\ll_{\bn,\nu}$ therein) grows like $e^{\min\{\ceil{\bn}^2,\ceil{\l}^2\}}$, meaning that if we wanted to use Theorem~\ref{cobgeodflow} to treat the aforementioned homogeneous flow, we would need not only a spectral gap for the Laplacian $\Delta$ on $\mathbb{K}(\en)\backslash \mathbb{G}(\en)/\G$ corresponding to some gap parameter $\nu_0\in\left[0,\frac{\en-1}{2}\right)$, which we have, but also a ``spectral \emph{cap}'' for the Laplacian on $\mathbb{M}(\en)\backslash \mathbb{G}(\en)/\G$ corresponding to a cap on $\ceil{\bn}$, which is an unrealistic assumption.  However, this exponential growth in our constant $C_{\bn,\nu}$ from Lemma~\ref{dmlbk} seems more a side-effect of our method of proof than an intrinsic facet of the problem. For example, notice that one of the first steps in the argument for that lemma is an application of the triangle inequality which practically conceals the fact that the distributions $\D^{\bm,\l}$ are defined by \emph{alternating} sums. It is conceivable that a deeper examination of these sums (taking into account their alternation) may yield more manageable growth in $C_{\bn,\nu}$. We mention this only as a speculation; it is of course unknown whether a statement like Theorem~\ref{two} should even hold for the flow of $X$ on $\SO^\circ(\en,1)/\G$. (After all, that flow is not Anosov, so we do not have the advantage of hindsight from the Liv\v{s}ic Theorem as we have for geodesic flows.)

In the next section we will see that these worries disappear once we are dealing with the flow of $X$ on $\mathbb{M}(\en)\backslash\SO^\circ(\en,1)/\G$, the unit tangent bundle of a finite-volume hyperbolic manifold.  This is because we will now be dealing only with $\mathbb{M}(\en)$-invariant elements of representations of $\SO^\circ(\en,1)$ and, it turns out, the only irreducible representations containing such elements are those where $\bn=(0,\dots,0,\ceil{\bn})$ and, furthermore, the coefficients of such elements are only non-zero on basis vectors $\{u(\bm,0)\}$. Hence, we may use the second part of Lemma~\ref{dmlbk}, proved for exactly this situation.

%===============================================

\section{Geodesic flows of finite-volume hyperbolic manifolds and proof of Theorem~\ref{two}}

We now turn our attention to geodesic flows of finite-volume hyperbolic manifolds. It is well-known that, under the identifications we have made, a hyperbolic $\en$-manifold is identified with $\mathbb{K}(\en)\backslash \mathbb{G}(\en)/\G$, where $\G\subset \mathbb{G}(\en):=\SO^\circ(\en,1)$ is some lattice. The geodesic flow of this manifold is the flow on its unit tangent bundle $\mathbb{M}(\en)\backslash \mathbb{G}(\en)/\G$ along the vector field $X$. 

In this section we set some important facts about $\mathbb{M}(\en)$-invariant elements of representations of $\SO^\circ(\en,1)$, and we prove Theorem~\ref{two}.

\subsection{$M$-invariant elements of representations}\label{firstminvelementssection}

Theorem~\ref{two} is about elements of $\Lii(\mathbb{M}(\en)\backslash \SO^\circ(\en,1)/\G)$. Since there is no representation of $\SO^\circ (\en,1)$ on this space, we cannot just apply what we have developed directly to this scenario.  Instead, we consider the subspace $\Lii(\SO^\circ(\en,1)/\G)^M$ of $M$-invariant elements of $\Lii(\mathbb{G}(\en)/\G)$, and make the identification 
\[
	\Lii(\mathbb{M}(\en)\backslash \SO^\circ(\en,1)/\G) \longleftrightarrow \Lii(\SO^\circ(\en,1)/\G)^M
\]
by
\[
	f(Mg\G) \longleftrightarrow \bar f (g\G).
\]
The operator $\Delta$ commutes with $M$, and we use it to define Sobolev norms on $\Lii(M\backslash G/\G)$.

\begin{lemma}\label{kernelmap}
Under the identification $\Lii(M\backslash G/\G)\longleftrightarrow\Lii(G/\G)^M$, 
\[
\ker\II_X^s (M\backslash G/\G)\hookrightarrow \ker\II_X^s (G/\G).
\]
\end{lemma}

\begin{proof}
Any distribution $\bar\D \in W^{-s}(G/\G)$ descends to a distribution $\D \in W^{-s}(M\backslash G/\G)$ defined by $\D(f) = \bar\D(\bar f)$, where $\bar f \leftrightarrow f$ in the identification above. Since $X$ commutes with $M$, $X$-invariance is preserved in this descent. So, if $f \in \ker\II_X^s(M \backslash G/\G)$, then for any $\bar\D\in\II_X^s(G/\G)$, we have $\bar\D(\bar f) = \D(f) = 0$.
\end{proof}

The next proposition tells us for which values of $\bn$ and $\nu$ the irreducible representation on $\HH_{\bn,\nu}$ admits $\mathbb{M}(\en)$-invariant elements, and furthermore it tells us which are the $\mathbb{M}(\en)$-invariant elements of those representations.

\begin{proposition}\label{minvelements}
The element $f = \sum f(\bm,\l)$ of the representation $\SO^\circ(\en,1)\to\HH_{\bn,\nu}$ is $\mathbb{M}(\en)$-invariant if and only if $f(\bm,\l)=0$ whenever $\l \not\equiv 0$.  In particular, the only irreducible unitary representations of $\SO^\circ(\en,1)$ which admit $\mathbb{M}(\en)$-invariant elements are those where $\bn=(0,\dots,0,\ceil{\bn})$.
\end{proposition}

\begin{proof}
This is easily read from the action of the Lie algebra of $M$ on the basis $\{u(\bm,\l)\}$ presented in~\cite{Hir1}. Another way to see it is to note that the top rows of the arrays $\l$ appearing under $\bm$ give the highest weights for the irreducible representations of $\mathbb{M}(\en)$ appearing in the restricted representation $\mathbb{M}(\en)\to\HH_{\bm}$, and $\mathbb{M}(\en)$-invariant elements lie in the irreducible representation of $\mathbb{M}(\en)$ with highest weight $0$, meaning that for a non-zero element $f \in \HH_{\bm}$ to be $\mathbb{M}(\en)$-invariant, it must have $f(\l)=0$ whenever $\l\not\equiv 0$. For it to be possible that $\l\equiv0$ appears under $\bm$, we must have $\bm = (0,\dots,0,\ceil{\bm})$, according to the inequalities indicated in Figure~\ref{figarray}, and this forces $\bn = (0,\dots,0,\ceil{\bn})$.
\end{proof}

\subsection{Proof of Theorem~\ref{two}}\label{twoproof}

\begin{proof}[Proof of Theorem~\ref{two}]
Let $f \in \ker\II_X^s (M\backslash G/\G)$, as required by the theorem statement.  The space $\Lii(M\backslash G/\G)$ is identified with the space $\Lii(G/\G)^M$ of $M$-invariant $\Lii$-functions on $G/\G$, so let us take $\bar f \in \Lii(G/\G)^M$ to be the image of $f$ in this identification.  By Lemma~\ref{kernelmap}, $\bar f$ lies in $\ker\II_X^s(G/\G)$, so its image in the direct integral decomposition
\begin{equation}\label{decomposition}
\Lii(G/\G) = \int_{\RR}^{\oplus}\HH_\s \,ds(\s)
\end{equation}
is of the form
\[
\bar f = \int_{\RR}^{\oplus}\bar f_\s \, ds(\s)
\]
where $\bar f_{\s} \in \ker\II_X^s (\HH_\s)$ for $ds$-almost every $\s$.  Furthermore, by Proposition~\ref{minvelements}, $\bar f_\s$ is only non-zero when $\HH_\s$ is of the form $\HH_{\bn,\nu}$ with $\bn=(0,\dots,0,\ceil{\bn})$, and for these $\s$, we know that $\bar f_{\s}(\bm,\l)=0$ whenever $\l\not\equiv 0$.  Therefore, the formal solution
\begin{align*}
\bar g_{\s} (\bm-e_k,\l) =
\begin{cases}
	\displaystyle{\sum_{{\bk}\in M_\l}\frac{\bar f_{\s}({\bk,\l})}{A_k^-(\bm,\l)} \D^{\bm,\l}({\bk},\l)}, &\textrm{if}\quad \en=2k, \textrm{ and} \\
	\displaystyle{\sum_{{\bk}\in M_\l}\frac{\bar f_\s ({\bk,\l})}{B_k^-(\bm,\l)} \D^{\bm,\l}({\bk},\l)}, &\textrm{if}\quad \en=2k+1
\end{cases}
\end{align*}
to the coboundary equation that we get from Proposition~\ref{propall} can be significantly simplified, because most of the terms are $0$.  In particular, notice that $\bar g_\s (\bm,\l)=0$ whenever $\l\not\equiv 0$, meaning that $\bar g_\s$ will be $M$-invariant once we show that the expressions defining it converge.  

Following the calculations in the proof of Theorem~\ref{cobgeodflow}, we find the Sobolev norms of $\bar g_\s \in \HH_\s$.  By definition,
\[
	\norm{\bar g_\s}_t^2 = \sum_{\l\in\L}\sum_{\bm\in M_\l}\left(1+Q_\s(\bm)\right)^t \Abs{\sum_{{\bk}\in M_\l}\frac{\bar f_{\s}({\bk,\l})}{A_k^-(\bm+e_k,\l)} \D^{\bm+e_k,\l}({\bk},\l)}^2.
\]
But the above observations allow us to disregard all terms except those with $\l\equiv 0$, so
\[
	\norm{\bar g_\s}_t^2 = \sum_{\bm\in M_0}\left(1+Q_\s(\bm)\right)^t \Abs{\sum_{{\bk}\in M_0}\frac{\bar f_{\s}({\bk,0})}{A_k^+(\bm,0)} \D^{\bm+e_k,0}({\bk},0)}^2.
\]
By the Cauchy--Schwartz inequality, this expression is bounded by
\[
\leq \sum_{{\bk}\in M_0}(1+Q_\s({\bk}))^s\abs{\bar f_\s ({\bk,0})}^2 \sum_{\bm\in M_0}\sum_{{\bk}\in M_0}\frac{(1+Q_\s (\bm))^t}{(1+Q_\s ({\bk}))^s }\,\Abs{\frac{\D^{\bm+e_k,\l}({\bk},0)}{A_k^+(\bm,0)}}^2
\]
and we see that the first sum is exactly the Sobolev norm of $\bar f_\s$ so that we can write
\[
\leq \norm{\bar f_\s}_s^2 \sum_{\bm < \bk \in M_0}\frac{(1+Q_\s (\bm))^t}{(1+Q_\s ({\bk}))^s }\,\Abs{\frac{\D^{\bm+e_k,\l}({\bk},0)}{A_k^+(\bm,0)}}^2
\]
and concern ourselves with bounding the second sum.  For this, we use Lemmas~\ref{coeffseven},~\ref{coeffsodd},~\ref{polynomials} and~\ref{dmlbk}, leaving
\[
\ll_{\nu_{\s}} \norm{\bar f_\s}_s^2 \sum_{\bm < \bk \in M_0}\frac{(1+\tilde\nu +2\ceil{\bm}^2 - \ceil{\bn}^2)^t}{(1+ \tilde\nu +2\ceil{\bk}^2- \ceil{\bn}^2)^s }\,\frac{1}{(\ceil{\bm}+1)^2-\ceil{\bn}^2},
\]
where $\nu_\s$ is a spectral gap parameter for the representation $\HH_\s$.  We re-write this as
\[
\ll_{\nu_{\s}} \norm{\bar f_\s}_s^2 \sum_{\bm \in M_0}\frac{(1+\tilde\nu +2\ceil{\bm}^2 - \ceil{\bn}^2)^t}{(\ceil{\bm}+1)((\ceil{\bm}+1)^2-\ceil{\bn}^2)}\,\sum_{\ceil{\bk}=\ceil{\bm}+1}^{\infty}\frac{\ceil{\bk}}{(1+ \tilde\nu +2\ceil{\bk}^2- \ceil{\bn}^2)^s},
\]
where we have introduced a factor of $\frac{\ceil{\bk}}{\ceil{\bm}+1}\geq 1$ because it is convenient for the next step.\footnote{One may be able to prove the theorem for $s >1/2$ and $t < s-\frac{1}{2}$ by being more delicate with this estimate.} The second sum is now bounded by the integral
\[
	\sum_{\ceil{\bk}=\ceil{\bm}+1}^{\infty}\frac{\ceil{\bk}}{(1+ \tilde\nu +2\ceil{\bk}^2- \ceil{\bn}^2)^s} \leq \int_{\ceil{\bm}}^{\infty}\frac{x\,dx}{(1+ \tilde\nu +2x^2- \ceil{\bn}^2)^s}
\]
so we can conclude that
\[
\ll_{\nu_{\s},s} \norm{\bar f_\s}_s^2 \sum_{\bm \in M_0}\frac{(1+\tilde\nu +2\ceil{\bm}^2 - \ceil{\bn}^2)^{t-s+1}}{(\ceil{\bm}+1)((\ceil{\bm}+1)^2-\ceil{\bn}^2)}.
\]
We can remove $\tilde\nu$ because it is always positive and because $t\leq s-1$.  This leaves us with 
\[
\ll_{\nu_{\s},s} \norm{\bar f_\s}_s^2 \sum_{\bm \in M_0}\frac{(1+ 2\ceil{\bm}^2 - \ceil{\bn}^2)^{t-s+1}}{(\ceil{\bm}+1)((\ceil{\bm}+1)^2-\ceil{\bn}^2)}.
\]
This sum converges, and is in fact bounded over possible values of $\ceil{\bn} \geq 0$ by some constant that only depends on the choice of $t$ and $s$ satisfying $t \leq s-1$, proving that 
\begin{equation}\label{irreducibletwo}
\norm{\bar g_\s}_t \ll_{\nu_\s, s, t} \norm{\bar f_\s}_s.
\end{equation}
We have thus proved an ``irreducible representations''-version of the theorem.

The next step is simply to notice that the spectral gap parameter $\nu_\s$ can be chosen uniformly across all representations appearing in the decomposition~\eqref{decomposition}:~Recall from Section~\ref{spectralgapsection} that, thanks to~\cite[Lemma~$3$]{Bek98}, there exists $\nu_0\in\left[0,\frac{\en-1}{2}\right)$ that is closer to $\frac{\en-1}{2}$ than any $\nu$ appearing in the direct integral decomposition of the left-regular representation of $G$ on $\Lii_0 (G/\G)$. We therefore have a solution
\[
	\bar g = \int_{\RR}^{\oplus}\bar g_\s \, ds(\s)
\]
to the coboundary equation $X\bar g = \bar f$ for the regular representation of $G$ on $\Lii(G/\G)$, satisfying
\[
	\Norm{\bar g}_t \ll_{\nu_0, s,t} \Norm{\bar f}_s
\]
for any $t\leq s-1$.  Furthermore, $\bar g$ is by construction $M$-invariant, so is a member of $\Lii(G/\G)^M$, and factors to a solution $g \in \Lii(M\backslash G/\G)$ to the coboundary equation $Xg=f$ on $M\backslash G/\G \cong S\mathcal{M}$ of Sobolev order at least $s-1$, and satisfying the same Sobolev estimate for every $t\leq s-1$.
\end{proof}

\subsection{A brief discussion of the relationship between Theorem~\ref{two} and the Liv\v{s}ic Theorem}

It is tempting to see this as a representation-theoretic approach to the Liv\v{s}ic Theorem.  Invariant measures supported on closed orbits of $\phi^t$ are themselves invariant distributions, and it is known that if $\mathcal{M}$ is compact then one can approximate any invariant distribution arbitrarily well in the weak topology by taking linear combinations of closed orbit measures.  Therefore, Theorem~\ref{two} implies the smooth version of the Liv\v{s}ic Theorem for geodesic flows of compact hyperbolic manifolds.  Of course, this is cheating, because the reason that we know that combinations of closed orbit measures approximate invariant distributions is that the Liv\v{s}ic Theorem is true.  Still, the fact remains that if one can prove that for \emph{every} lattice $\G\subset \SO^\circ(\en,1)$ the set of linear combinations of closed orbit measures of the geodesic flow on $\mathbb{M}(\en)\backslash \mathbb{G}(\en)/\G$ is weakly dense in the space of invariant distributions $\II_X(\Lii(\mathbb{M}(\en)\backslash \mathbb{G}(\en)/\G))$, then Theorem~\ref{two} implies a smooth version of the Liv\v{s}ic Theorem for finite-volume (not necessarily compact) hyperbolic manifolds, and functions in $\Cinf(\Lii(\mathbb{M}(\en)\backslash \mathbb{G}(\en)/\G))$.\footnote{We will come back to a similar discussion for $\RR^d$-actions in Part~\ref{partii}, Section~\ref{whynokk}, where the relationship between closed orbit measures and invariant distributions is not even known for compact manifolds.}

Already, if one restricts attention to bounded Lipschitz functions, there is a version of the Liv\v{s}ic Theorem for the geodesic flow of a hyperbolic manifold with cusps, due to T.~Foth and S.~Katok~\cite[Section~$3.1$]{FK01}, which is proved with dynamical techniques.  Though similar techniques are likely to lead to a smooth version, we point out that our Theorem~\ref{two} automatically implies that the solutions in the Foth--Katok statement of the Liv\v{s}ic Theorem are smooth, proving Theorem~\ref{livsicthm}.

%===============================================

\part{Higher cohomology of higher-rank Anosov actions}\label{partii}

%========================================================
\section{Preliminaries for Part~\ref{partii}}\label{preliminariesii}

Recall our definitions in Section~\ref{preliminaries} for higher-degree cohomology. We introduce here the same definitions in the context of unitary representations
\[
	\SO^\circ (\en_1,1)\times\dots\times\SO^\circ (\en_d,1)\longrightarrow\U(\HH),
\]
and for Sobolev spaces of such representations. To wit, $\O_{\RR^d}^n (W^s(\HH))$ denotes the set of $n$-forms over our $\RR^d$-action on the Hilbert space $\HH$ of Sobolev order $s$.  By an \emph{$n$-form of Sobolev order $s$ over the $\RR^d$-action on $\HH$}, we mean a map
\[
\o:(\mathrm{Lie}(\RR^{d}))^n \rightarrow W^{s}(\HH)
\] 
that is linear and anti-symmetric.  The exterior derivative $\di$, taking a form to another form of one higher degree, is defined by 
\begin{align*}
	\di\o(V_{1},\dots,V_{n+1}) &:= \sum_{j=1}^{n+1} (-1)^{j+1}\,V_{j}\,\o(V_1,\dots,\widehat{V_{j}},\dots, V_{n+1}),
\end{align*}  
where ``$\quad \widehat{}\quad$'' denotes omission.  The form $\di\o$ takes values in a lower Sobolev space $W^{s-1}(\HH)$.

It is natural to see $\o$ as an element $\o \in W^{s}(\HH)^{\binom{d}{n}}$, indexed by ordered $n$-tuples from the set $\left\{X_1, \dots, X_d\right\}$.  The form $\o$ is said to be \emph{closed}, and is called a \emph{cocycle}, if $\di\o=0$, or
\begin{align*}
	\di\o(X_{I}) &:= \sum_{j=1}^{n+1} (-1)^{j+1}\,X_{i_{j}}\,\o(X_{I_{j}}) = 0,
\end{align*}  
where $I:=(i_{1},\dots,i_{n+1})$ with $i_j \in \left\{1,\dots,d \right\}$ is the multi-index, and $I_{j}:=(i_1,\dots, \widehat{i_{j}}, \dots, i_n)$.  It is \emph{exact}, and is called a \emph{coboundary}, if there is an $(n-1)$-form $\eta$ satisfying $\di\eta=\o$.  Two forms that differ by a coboundary are \emph{cohomologous}.

Notice that if $n=d$, then $\o \in \O_{\RR^d}^n(W^s (\HH))$ is given by just one element 
\[
	\o(X_1,\dots,X_d) = f \in W^{s}(\HH); 
\]
it is automatically closed, and exactness is characterized by the existence of a $(d-1)$-form $\eta$ satisfying $\di\eta=\o$.  Or, setting $\eta(X_{I_j}) = (-1)^{j+1}g_{j}$,
\begin{align*}
	\di\eta(X_1,\dots,X_d) &= \sum_{j=1}^{d} (-1)^{j+1}\,X_{j}\,\eta(X_{I_j}) \\
			&= \sum_{j=1}^{d} X_{j}\,g_j = f.
\end{align*}
This is exactly the top-degree coboundary equation in Theorem~\ref{one}. 

We introduce the following norm on $\O_{\RR^d}^{n}(W^s(\HH))$ that will allow for more concise theorem statements.  For $\o \in \O_{\RR^d}^{n}(W^s(\HH))$, 
\[
	\norm{\o}_s^2 := \sum_{1\leq i_1 <\dots<i_n\leq d} \norm{\o(X_{i_1},\dots,X_{i_n})}_s^2.
\]
We trust that no confusion will arise from using $\norm{\cdot}_s$ to denote this norm. It should be clear from context whether we mean the norm on $\O_{\RR^d}^{n}(W^s(\HH))$ or on $W^s(\HH)$.  

\section{Results for higher-rank Anosov actions}\label{higherdegreeresults}

The aim of Part~\ref{partii} is to prove Theorem~\ref{one}.  We break this task into different ``pieces.''  Theorem~\ref{generaltopdegree} below corresponds to the ``top degree'' part of Theorem~\ref{one}, while Theorem~\ref{generallowerdegree} corresponds to the ``lower degrees'' part.  Both of these theorems are stated in a style similar to Theorem~\ref{two}.  That is, they hold for elements of Sobolev spaces.  Together, they imply Theorem~\ref{one} in the case where all $\en_i \geq 3$.

\begin{theorem}[Sobolev spaces version of Theorem~\ref{one} in top degree]\label{generaltopdegree}
Let $\mathcal{M}$ be a finite-volume irreducible quotient of the product $\Hyp^{\en_1}\times\dots\times\Hyp^{\en_d}$ by a discrete group of isometries, where $\en_i \geq 3$, and let $S\mathcal{M}$ be its unit tangent bundle.  Consider the standard Anosov action $\RR^d \curvearrowright S\mathcal{M}$. Given any $s>1$ and $t\leq s-1$, there are a constant $C_{s,t}>0$ and a number $\s_d(s):=\s_d$ (depending on $s$ and $d$) such that if $f \in W^{\s_d}(S\mathcal{M})$ is in the kernel of every $\RR^d$-invariant distribution of Sobolev order $\s_d$, then there exist functions $g_1,\dots,g_d \in W^t (S\mathcal{M})$ satisfying the degree-$d$ coboundary equation
	\[
		X_1\,g_1 + \dots + X_d\,g_d = f,
	\]
and the Sobolev estimates $\norm{g_i}_t \leq C_{s,t}\,\norm{f}_{\s_d}$.
\end{theorem}

\begin{remark*}
The number $\s_d$ is the result of the recursion $\s_n = 2(\s_{n-1}+s)-1$ with initial data $\s_1 =s$.  It is an artifact of our method of proof; we do not claim that the losses of Sobolev order here are the tightest possible.  In fact, these are already an ``improvement'' over the loss of Sobolev order found in the main theorems of~\cite{Ramhc}, in that the recursion does not involve a `` $+d $ ''-term.

One can check that $\s_d=(2^d + 2^{d-1}-2)s - (2^{d-1}-1)$.  From this we see that the theorem holds for functions of Sobolev order greater than $2^d -1$.
\end{remark*}

For lower-degree cohomology, we have the following.

\begin{theorem}[Sobolev spaces version of Theorem~\ref{one} in lower degrees]\label{generallowerdegree}
Consider the action $\RR^d \curvearrowright S\mathcal{M}$ from Theorem~\ref{generaltopdegree}. For any $s >1$ and $t \leq s-1$, there are a constant $C_{s,t}>0$ and a number $\vars_d(s):=\vars_d$ such that for any $n$-cocycle $\o \in \O_{\RR^d}^{n} (W^{\vars_d} (S\mathcal{M}))$ with $1\leq n\leq d-1$, there exists $\eta \in \O_{\RR^d}^{n-1} (W^{t}(S\mathcal{M}))$ with $\di\eta = \o$ and $\norm{\eta}_t \leq C_{\tilde\nu_0,s,t}\,\norm{\o}_{\vars_d}$.
\end{theorem}

\begin{remark*}
We have a viable definition for the number $\vars_d$ in the proof of Theorem~\ref{generallowerdegree}, but make no claim as to its optimality.
\end{remark*}

Theorems~\ref{generaltopdegree} and~\ref{generallowerdegree} are both proved through inductive arguments similar to those that were used in~\cite{Ramhc} to prove analogous theorems for products of $\SL(2,\RR)$.  However, some re-working is required, especially for Theorem~\ref{generaltopdegree}.  This is essentially because the space of invariant distributions for the geodesic flow on $\SO^\circ(\en,1)$ with $\en\geq 3$ is so different from that of $\SO^\circ (2,1)$: In any irreducible unitary representation of the latter, the space of invariant distributions is at most two-dimensional, and every basis element is annihilated by one (or, more importantly, all but one) of those dimensions.  These two properties are essential to the methods presented in~\cite{Ramhc}, and they are both absent when $\en \geq 3$.  (In fact, the second property does not hold for the horocycle flows in representations of $\SL(2,\RR)$, which is why only geodesic flows were considered in the previous paper.)  Indeed, Theorem~\ref{invdistthm} shows that we now have an infinite-dimensional space of invariant distributions in any irreducible representation of $\SO^\circ(\en,1)$, and the annihilator of a basis element $u(\bm,\l)$ can have co-dimension as large as $\Abs{\floor{M_{\l}}}$. 

Part~\ref{partii} of this article is therefore concerned with modifying the arguments used for products of $\SL(2,\RR)$ so that they work more generally for $\SO^\circ(\en_1,1)\times\dots\times\SO^\circ(\en_d,1)$.  Though Theorems~\ref{generaltopdegree} and~\ref{generallowerdegree} are stated for $\RR^d$-actions built from geodesic flows, the methods of proof are now flexible enough to accommodate other $\RR^d$-actions built from flows that have a theorem analogous to Theorem~\ref{cobgeodflow}.  The presentation here is tailored to actions built from geodesic flows, as in the above theorems, but in a forthcoming article, we will indicate how to adjust the argument for other actions that are not necessarily Anosov, for example the $\RR^d$-action on quotients of $\PSL(2,\RR)\times\dots\times\PSL(2,\RR)$ defined by horocycle flows.  There, the proof in top degree will use results from~\cite{FF} as a base case, and the proof for lower degrees will rely on a result in~\cite{M1} as base case.  

\begin{comment}
Combining the arguments for Theorems~\ref{generaltopdegree} and~\ref{generallowerdegree} with those for Theorem~\ref{horothm}  and the main theorems in~\cite{Ramhc} allows one to make the same statement for actions that are a mixture of these situations, and in particular, yields Theorem~\ref{one}.
\end{comment}
%=====================================================================

\section{Notation for Part~\ref{partii}}

Let us alert the reader to some significant changes in the notation.

Bold symbols in Part~\ref{parti}, such as $\bm, \bn, \bk$, were used to index the elements of the basis $\{u(\bm,\l)\}$ in an irreducible unitary representation $\HH_{\bn,\nu}$ of $\SO^\circ(\en,1)$.  Now that we have a product of such groups, irreducible unitary representations are tensor products \mbox{$\HH_{\bn_1,\nu_1}\otimes\dots\otimes\HH_{\bn_d,\nu_d}$}, and we have a basis $\{u(\bm_1,\l_1)\otimes\dots\otimes u(\bm_d, \l_d)\}$ by taking tensor products of the previous basis elements (see Section~\ref{productreps}).  We therefore find it convenient to use bold symbols in a different way.

The indexing set $\{(\bm,\l)\mid \bm\in M_{\bn}, \l\in\L_{\bm}\}$ will be denoted $\Z_{\bn,\nu}$ or just $\Z$ when there is no chance of confusion.  The set of indices corresponding to ``floor'' points ($\bm,\l$ where $\bm \in \floor{M_\l}$) will be denoted $\floor{\Z}$.  Products of such indexing sets will be denoted with bold symbols, for example $\bZ = \Z_{\bn_1,\nu_1}\times\dots\times\Z_{\bn_d,\nu_d}$.  Accordingly, bold symbols will now be used to denote elements of ``bold'' sets so, for example, $\bz \in \bZ$ is a $d$-duple $\bz=(z_1,\dots,z_d)$ where $z_i \in \Z_{\bn_i,\nu_i}$.  Occasionally, we may need to refer to the array $(\bm,\l)$ corresponding to an element $z\in\Z$, in which case we will use a subscript.  That is, $z=(\bm_z, \l_z) \in \Z$.

Sometimes we indicate products, tensor products, and sums with a subscript corresponding to the operation, rather than (or in conjunction with) a bold symbol, especially when we want to distinguish the last factor of a long product.  For example, 
\begin{itemize}
	\item $\HH_{\bn_1,\nu_1}\otimes\dots\otimes\HH_{\bn_d,\nu_d}=\HH_1\otimes\dots\otimes\HH_d = \HH_\otimes \otimes\HH_d$.  
	\item $\Z_{\bn_1,\nu_1}\times\dots\times\Z_{\bn_d,\nu_d} = \Z_{1}\times\dots\times\Z_{d} = \bZ_\times \times \Z_d$
	\item For $\bz_\times = (z_1,\dots,z_{d-1}) \in \bZ_\times$ and $z_d\in\Z_d$, $$Q_{\bn_1,\nu_1}(\bm_{z_1})+\dots+Q_{\bn_d,\nu_d}(\bm_{z_d}) = Q_{1}(\bm_{z_1})+\dots+Q_{d}(\bm_{z_d}) = Q_+(\bz_\times) + Q_d(z_d).$$
\end{itemize}

%=====================================================================

\section{Unitary representations of products}\label{productreps}

It is well-known that all irreducible unitary representations of a product $G_1\times\dots\times G_d$ of semisimple Lie groups are tensor products $\HH_1\otimes\dots\otimes\HH_d$ of irreducible unitary representations of the factors.  Therefore, we have a good description of the irreducible unitary representations of $\SO^\circ(\en_1,1)\times\dots\times\SO^\circ(\en_d,1)$ just by understanding the unitary dual of $\SO^\circ(\en,1)$.  

There is an orthonormal basis
\[
	\left\{u(\bz):=u_1 (z_1)\otimes\dots\otimes u_d (z_d)\right\}_{z_i\in\Z_i}
\] 
for $\HH:=\HH_{\bn_1,\nu}\otimes\dots\otimes\HH_{\bn_d,\nu_d}$ consisting of tensor products of the basis elements defined in Sections~\ref{basiseven} and~\ref{basisodd}.  The Lie algebra element $X_i$ acts on the $i^{\textrm{th}}$ component, according to~\eqref{xactioneven} if $\en_i$ is even, and according to~\eqref{xactionodd} if $\en_i$ is odd.  This allows us to see the degree-$d$ coboundary equation
\begin{equation}\label{dcobeq}
X_1 \, g_1 + \dots + X_d\, g_d = f
\end{equation}
for some element $f = \sum_{\bz\in\bZ} f(\bz)u(\bz) \in \HH$ as a partial difference equation on the subset $\bZ:=\Z_{\bn_1,\nu_1}\times\dots\times\Z_{\bn_d,\nu_d}$ of $\ZZ^q$ (for a large enough $q\in\NN$), similarly to how we proceeded in Part~\ref{parti}.  Fortunately, it will not be necessary to scrutinize this PdE in the same amount of detail.

%=====================================================================

\subsection{Sobolev spaces}

As in Section~\ref{sobnorms}, we have the operators $\Delta, \Box, \Box_K$.  Each is defined by sums of the corresponding operators $\Delta_i, \Box_i, \Box_{K_i}$ coming from the factors $G_i = \SO^\circ(\en_i,1)$.  So for an element $f = \sum_{\bz\in\bZ} f(\bz)u(\bz)$ of an irreducible unitary representation $\HH:=\HH_{\bn_1,\nu}\otimes\dots\otimes\HH_{\bn_d,\nu_d}$, Sobolev norms are given by
\[
	\norm{f}_s^2 = \sum_{\bz\in\bZ} \left(1 + Q_+ (\bz)\right)^s \,\abs{f(\bz)}^2\,\norm{u(\bz)}^2,
\]
where $Q_+ (\bz):= Q_{\bn_1,\nu_1}(\bm_{z_1})+\dots+Q_{\bn_d,\nu_d}(\bm_{z_d})$.

We will find it useful to work with projected versions of $f$.  Let $\ell \in \{1,\dots,d\}$.  Fixing $z_\ell \in \Z_\ell$ we define
\begin{align*}
	(f\mid_{z_{\ell}}) &= \sum_{i\neq\ell}\sum_{z_i\in\Z_i} f(\bz)\,\norm{u_{z_\ell}}\,u_1(z_1)\otimes\dots\otimes\widehat{u_\ell (z_{\ell})}\otimes\dots\otimes u_d(z_d) \\
	&\in \HH_{\bn_1,\nu_1}\otimes\dots\otimes\widehat{\HH_{\bn_\ell,\nu_\ell}}\otimes\dots\otimes\HH_{\bn_d,\nu_d}
\end{align*}
and the restricted Sobolev norm
\[
	\Norm{(f\mid_{z_{\ell}})}_s^2 = \sum_{i\neq\ell}\sum_{z_i\in\Z_i}\left(1 + Q_1 (\bm_{z_1}) +\dots+\widehat{Q_\ell (\bm_{z_\ell})}+\dots+Q_d (\bm_{z_d})\right)^s \,\abs{f(\bz)}^2\,\norm{u(\bz)}^2.
\]
We immediately see that $\Norm{(f\mid_{z_{\ell}})}_s \leq \Norm{f}_s$.  We will also need the following simple lemma.

\begin{lemma}\label{adding}
For any $s, s' \in \NN$, 
\[
	\sum_{z_\ell \in \Z_\ell} \left(1+ Q_\ell (\bm_{z_\ell})\right)^{s}\,\Norm{(f\mid_{z_\ell})}_{s'}^2 \leq \Norm{f}_{s + s'}^2.
\]
\end{lemma}

\begin{proof}
We just calculate.  By definition, the left-hand side is 
\begin{align*}
&=\sum_{\bz\in\bZ}\left(1+ Q_\ell (\bm_{z_\ell})\right)^{s} \left(1 + Q_1 (\bm_{z_1}) +\dots+\widehat{Q_\ell (\bm_{z_\ell})}+\dots+Q_d (\bm_{z_d})\right)^{s'}\,\abs{f(\bz)}^2\,\norm{u(\bz)}^2.
\intertext{Since all the $Q$'s are positive, this is}
&\leq \sum_{\bz\in\bZ}\left(1 + Q_1 (\bm_{z_1}) +\dots+Q_\ell (\bm_{z_\ell})+\dots+Q_d (\bm_{z_d})\right)^{s+s'}\,\abs{f(\bz)}^2\,\norm{u(\bz)}^2,
\end{align*}
which is exactly the right-hand side, $\norm{f}_{s+s'}^2$.
\end{proof}

%=====================================================================

\subsection{Invariant distributions}
We are interested in distributions that are invariant under the $\RR^d$-action:
\[
	\II_{X_1,\dots,X_d}(\HH) = \left\{ \D \in \E^{\prime}(\HH) \mid \LLL_{X_i}\D=0 \quad\textrm{for all}\quad i=1,\dots,d \right\}
\]
and
\[
	\II_{X_1,\dots,X_d}^{s}(\HH) = \left\{ \D \in W^{-s}(\HH) \mid \LLL_{X_i}\D=0 \quad\textrm{for all}\quad i=1,\dots,d \right\},
\]
where $\HH$ is a unitary representation of $G_1\times\dots\times G_d$.  In an irreducible unitary representation $\HH_1\otimes\dots\otimes\HH_d$, these sets are easy to describe in terms of the distributions defined in Section~\ref{formalsection}.  Theorem~\ref{invdistthm} tells us that in any irreducible unitary representation of $\SO^\circ(\en,1)$ the $X$-invariant distributions are spanned by the set $\left\{\D^{w}\mid w \in \floor{\Z}\right\}$.  Since an element of $\II_{X_1,\dots,X_d}(\HH_1\otimes\dots\otimes\HH_d)$ only needs to satisfy its corresponding PdE~\eqref{distde} in its corresponding component of $\HH_1\otimes\dots\otimes\HH_d$, it is easy to see that $\left\{\D^{\bw}\mid\bw\in\floor{\bZ}\right\}$ is a spanning set, where $\floor{\bZ} :=\floor{\Z_1}\times\dots\times\floor{\Z_d}$, and 
\[
\D^{\bw}(u(\bz)) = \D^{w_1}(u_1(z_1))\cdots\D^{w_d}(u_d(z_d))
\]
for any $\bz=(z_1,\dots,z_d)\in\bZ$.  The following proposition gives an upper bound on the Sobolev order of these distributions.

\begin{proposition}\label{hrsoborders}
For $\bw \in \floor{\bZ}$, $\D^{\bw}$ is an $X_1,\dots, X_{d}$-invariant distribution of Sobolev order at most $\frac{1}{2}d$.
\end{proposition}

\begin{proof}
Let $s >\frac{7}{2}d$ and $f \in W^s (\HH_1\otimes\dots\otimes\HH_d)$.  We use the Cauchy--Schwartz inequality to bound
\[
	\Abs{\D^{\bw}(f)}^2 = \Abs{\sum_{\bz \in \bZ }f(\bz)\,\D^{\bw}(u (\bz))}^2 \leq \Norm{f}_{s}^2 \,\sum_{\bz \in \bZ}\frac{\Abs{\D^{\bw}(u (\bz))}^2}{(1+Q (\bm_{\bz}))^{s}}.
\]
The right-most sum is bounded by
\[
\prod_{i=1}^{d}\sum_{z_i \in \Z_i}\frac{\Abs{\D^{w_i}(u (z_i))}^2}{(1+Q (\bm_{z_i}))^{s/d}},
\]
which converges because $\frac{s}{d} > \frac{1}{2}$, and because of Lemma~\ref{dmlbk}.
\end{proof}

%=====================================================================

\subsection{$M$-invariant elements}\label{secondminvelementssection}

Since we only consider representations of $\SO^\circ (\en_1,1)\times\dots\times\SO^\circ (\en_d,1)$ that admit $M$-invariant elements, it is worth pointing out some specifics regarding \emph{which} representations fit this description, and furthermore which elements of these representations are $M$-invariant. This is really rather easy, in view of the discussion in Section~\ref{firstminvelementssection}.

We deduce from Proposition~\ref{minvelements} that an irreducible unitary representation of $\SO^\circ (\en_1,1)\times\dots\times\SO^\circ (\en_d,1)$ on $\HH_{\bn_1, \nu_1}\otimes\dots\otimes\HH_{\bn_d, \nu_d}$ admits $M$-invariant elements only if $\bn_i = (0,\dots,0,\ceil{\bn_i})$ for every $i=1,\dots,d$, and that $f \in \HH_{\bn_1, \nu_1}\otimes\dots\otimes\HH_{\bn_d, \nu_d}$ is $M$-invariant if and only if $f(z_1,\dots,z_d) = 0$ whenever $\l_{z_i} \not\equiv 0$ for some $i \in \{1,\dots,d\}$.

The following lemma shows that whenever we have a solution to the degree-$d$ coboundary equation for an $M$-invariant function, then we can also find a solution that is itself $M$-invariant, and has Sobolev norms bounded by the Sobolev norms of the original solution. We will use it in the proofs of Theorems~\ref{generaltopdegree} and~\ref{generallowerdegree}.

\begin{lemma}\label{wlog}
If $f \in \Lii(G/\G)^M$, then for any solution $g_1,\dots,g_d \in \Lii(G/\G)$ to 
\[
	X_1\,g_1+\dots+X_d\,g_d = f
\]
there is another solution $\tilde g_1,\dots, \tilde g_d \in \Lii(G/\G)^M$ such that $\norm{\tilde g_i}_t \leq \norm{g_i}_t$.
\end{lemma}

\begin{proof}
We just put
\[
	\tilde g_i (mx) := \int_M g_i (mx)\,dm
\]
where $dm$ denotes normalized Haar measure on $M$. It is clear that $\tilde g_i$ is $M$-invariant.  To see that it satisfies the degree-$d$ coboundary equation, we calculate
\begin{align*}
	\left(X_1\,\tilde g_1+\dots +X_d\,\tilde g_d\right)(x) &= \int_M X_1\,g_1 (mx)+\dots +X_d\,g_d(mx)\,dm \\
		&= \int_M f(mx)\,dm = f(x),
\end{align*}
since $f$ is $M$-invariant.  For the Sobolev norms,
\begin{align*}
	\norm{\tilde g_i}_t^2 &= \int_{G/\G} \Abs{(\bone+\Delta)^{t/2} \,\tilde g_i (x)}^2\,d\mu_{G/\G} \\
		&= \int_{G/\G} \Abs{(\bone+\Delta)^{t/2} \,\int_M g_i (mx)\,dm}^2\,d\mu_{G/\G}
\intertext{and since $m(M)=1$,}
		&\leq \int_{G/\G} \int_M \Abs{(\bone+\Delta)^{t/2}\,g_i (mx)}^2\,dm\,d\mu_{G/\G} = \norm{g_i}_t^2,
\end{align*}
which proves the lemma.
\end{proof}

%=====================================================================

\section{Top-degree cohomology and proof of Theorem~\ref{generaltopdegree}}\label{topdegreesection}

%=====================================================================

%=====================================================================

We first prove a version of Theorem~\ref{generaltopdegree} for irreducible unitary representations of $G=G_1\times\dots\times G_d$ on $\HH_{\bn_1,\nu_1}\otimes\dots\otimes\HH_{\bn_d,\nu_d}$.  Let us take an element
\[
	f\in W^{s}(\HH_{\bn_1,\nu_1}\otimes\dots\otimes\HH_{\bn_d,\nu_d}) := W^{s}(\HH_{\otimes}\otimes\HH_d)
\]
where $s$ is ``large enough,'' $\HH_d := \HH_{\bn_d,\nu_d}$, and $\HH_{\otimes} := \HH_{\bn_1,\nu_1}\otimes\dots\otimes\HH_{\bn_{d-1},\nu_{d-1}}$.  We would like to solve the coboundary equation
\[
	f = X_1\, g_1 +\dots+ X_d\,g_d,
\]
provided $f \in \ker\II_{X_1,\dots, X_d}$.  The strategy is to split $f$ as $f = f_{\otimes} + f_d$, where $f_{\otimes}$ is in the kernel of all $X_1,\dots,X_{d-1}$-invariant distributions, and $f_d$ is in the kernel of all $X_d$-invariant distributions.  To this end, define
\begin{align}\label{fotimesadhoc}
	f_{\otimes} (\bz,w) = \begin{cases} \sum_{z\in\Z_d} f(\bz, z)\,\D^{w}(u_d(z)) &\textrm{if } w \in \floor{\Z_d} \\
												0 &\textrm{if not.}
							\end{cases}
\end{align}
We then define $f_d = f - f_{\otimes}$.

First we show that the expressions for $f_{\otimes}$ and $f_d$ converge and determine elements of the Hilbert space $\HH$.  In fact, they retain some of the Sobolev regularity of $f$.

\begin{lemma}\label{regularity}
Suppose $f \in W^s (\HH_\otimes \otimes \HH_d)^M$ for some $s>1$.  Then  $f_\otimes$ and $f_d$ have Sobolev order at least $s-1$.  Furthermore, 
\[
\norm{f_\otimes}_{\frac{s-1}{2}} \ll_{s} \norm{f}_{s} \quad\textrm{and}\quad \norm{f_d}_{\frac{s-1}{2}} \ll_{s} \norm{f}_{s}.
\]
\end{lemma}

\begin{proof}
We prove the lemma for $f_\otimes$, which implies the $f_d$ part.  Let $\t >0$.  Then
\begin{align*}
\norm{f_{\otimes}}_{\t}^2 &= \sum_{(\bz,w)\in\bZ_\times \times\Z_d}(1+Q_{+}(\bz)+Q_d (w))^{\t}\,\abs{f_{\otimes}(\bz,w)}^2\,\norm{u_\otimes(\bz) \otimes u_d (w)}^2 \\
	&= \sum_{(\bz,w)\in\bZ_\times \times\floor{\Z_d}}(1+Q_{+}(\bz)+Q_d (w))^{\t} \Abs{\sum_{z\in\Z_d} f(\bz,z)\,\D^{w}(u_d(z))}^2
\end{align*}
Then, applying the Cauchy--Schwartz inequality to the second sum,
\begin{multline*}
	\leq \sum_{(\bz,w)\in\bZ_\times \times\floor{\Z_d}}(1+Q_{+}(\bz)+Q_d (w))^{\t} \\ \sum_{z\in\Z_d}(1+Q_d (z))^{\t'}\, \abs{f(\bz, z)}^2 \sum_{z\in M_{\l_w}\times\{\l_w\}} (1+Q_d (z))^{-\t'}\,\abs{\D^{w}(u_d(z))}^2
\end{multline*}
and by Lemma~\ref{dmlbk},
\begin{align*}
	\ll \sum_{(\bz,w)\in\bZ_\times \times\floor{\Z_d}}(1+Q_{+}(\bz)+Q_d (w))^{\t}\,\norm{(f\mid_{\bz})}_{\t'}^2 \sum_{z\in M_{\l_w}\times\{\l_w\}} \frac{1}{(1+Q_d (z))^{\t'}}.
\end{align*}
The inequality $(1+A+B)\leq(1+A)(1+B)$ for $A,B \geq 0$ gives
\begin{align*}
	\ll \sum_{\bz\in\bZ_\times}(1+Q_{+}(\bz))^{\t}\,\norm{(f\mid_{\bz})}_{\t'}^2 \sum_{w \in \floor{\Z_d}}(1 + Q_d (w))^{\t} \sum_{z \in M_{\l_w}\times\{\l_w\}}\frac{1}{(1+Q_d (z))^{\t'}}
\end{align*}
and Lemma~\ref{adding} allows us to re-write the first sum,
\begin{align*}
	\ll \norm{f}_{\t+\t'}^{2} \sum_{w \in \floor{\Z_d}}(1 + Q_d (w))^{\t} \sum_{\bm \in M_{\l_w}}\frac{1}{(1+Q_d (\bm))^{\t'}}.
\end{align*}
Now, our assumption that $f$ is $M$-fixed implies that $\floor{\Z_d}$ is the singleton
\[
	\floor{\Z_d} = \left\{\left( (0,\dots,0,\ceil{\bn^{(d)}}), 0 \right)\right\}
\]
by Proposition~\ref{minvelements}, and so the first sum only has one summand.  Lemma~\ref{polynomials} allows us to write
\begin{align*}
	\ll \norm{f}_{\t+\t'}^{2}\, (1+\tilde\nu+ \ceil{\bn^{(d)}}^2)^{\t} \,\sum_{\bm \in M_{0}}\frac{1}{(1+\tilde\nu+ 2\ceil{\bm}^2 - \ceil{\bn^{(d)}}^2)^{\t'}}.
\end{align*}
The sum is bounded\footnote{This bound, like the one in Section~\ref{twoproof}, is looser than it needs to be, but makes the calculations easier. A tighter integral bound may lead to a theorem statement with smaller loss of Sobolev regularity in solutions.} by
\[
	\sum_{\bm \in M_{0}}\frac{1}{(1+\tilde\nu+ 2\ceil{\bm}^2 - \ceil{\bn^{(d)}}^2)^{\t'}} \leq \int_{\ceil{\bn^{(d)}}}^{\infty} \frac{(x+1)\,dx}{(1+\tilde\nu+ 2x^2 - \ceil{\bn^{(d)}}^2)^{\t'}}
\]
so that we have
\[
	\norm{f_{\otimes}}_{\t}^2 \ll \norm{f}_{\t+\t'}^{2}\, \frac{1}{\t'-1}\,(1+\tilde\nu+ \ceil{\bn^{(d)}}^2)^{\t - \t' +1}.
\]
First, this implies that as long as $\t'>1$ and $\norm{f}_{\t+\t'}$ is finite, then $\norm{f_\otimes}_{\t}$ is finite.  This proves the first claim, that $f_\otimes$ (and, hence, $f_d$ also) loses at most $1$ from $f$'s Sobolev order.  Now, by choosing $\t' = \t +1$, we have that $\norm{f_\otimes}_{\t}^2 \ll (1/\t)\norm{f}_{2\t+1}^2$ and, in particular, $\norm{f_\otimes}_{\t} \ll_{\t} \norm{f}_{2\t+1}$.  The lemma is proved by putting $\t = (s-1)/2$.
\end{proof}

The next lemma follows immediately from the first part of Lemma~\ref{regularity}, and the simple observation that $f_\otimes$ and $f_d$ are $M$-invariant if $f$ is.

\begin{lemma}\label{nowobvious}
If $f \in W^s (\HH_{\otimes}\otimes\HH_d)^M$ then 
\[
(f_{\otimes}\mid_{w})\in W^{s-1}(\HH_{\otimes})^M  \quad\textrm{for all}\quad w \in \Z_d
\]
and 
\[
(f_d \mid_{\bz})\in W^{s-1}(\HH_d)^M \quad\textrm{for all}\quad \bz \in \bZ_\times.
\]
\end{lemma}

The next lemma shows that $f_{\otimes}$ and $f_d$ are in the desired kernels of invariant distributions.

\begin{lemma}\label{kernels}
Suppose $f \in \ker\II_{X_1,\dots,X_d}^s (\HH_{\otimes}\otimes\HH_d)$, where $s>d$.  Then for every $w\in\Z_d$, we have $(f_{\otimes}\mid_w)\in\ker\II_{X_1,\dots,X_{d-1}}^{s-1}(\HH_{\otimes})$.  Similarly, for every $\bz \in \bZ_\times$, we have $(f_d \mid_{\bz})\in\ker\II_{X_d}^{s-1}(\HH_{d})$.
\end{lemma}

\begin{proof}
Let $\bw \in \floor{\bZ_\times}$ so that $\D^{\bw} \in \II_{X_1,\dots,X_{d-1}}^{s-1}(\HH_{\otimes})$ (by Proposition~\ref{hrsoborders}).  If $w \notin \floor{\Z_d}$, then $(f_\otimes \mid_w)=0$ (see~\eqref{fotimesadhoc}), so it is enough to check the result for $w \in\floor{\Z_d}$.  By Lemma~\ref{nowobvious}, $(f_\otimes \mid_w)$ lies in the domain of $\D^{\bw}$ and we may compute
\begin{align*}
	\D^{\bw} (f_{\otimes}\mid_w) &= \sum_{\bz \in \bZ_\times} f_{\otimes}(\bz,w)\D^{\bw}(u_\otimes (\bz)) \\
		&= \sum_{\bz \in \bZ_\times} \left[\sum_{z\in\Z_d} f(\bz, z)\,\D^{w}(u_d(z))\right]\D^{\bw}(u_\otimes (\bz)) \\
		&=  \sum_{\bz \in \bZ_\times}\sum_{z\in\Z_d} f(\bz, z)\,\D^{\bw,w}(u_\otimes(\bz)\otimes u_d(z)) \\
		&= \D^{(\bw,w)}(f) = 0,
\end{align*}
since $\D^{(\bw,w)} \in \II_{X_1,\dots,X_d}^s(\HH_\otimes \otimes\HH_d)$ by Proposition~\ref{hrsoborders}, and because $f$ is assumed to be in $\ker\II_{X_1,\dots,X_d}^s (\HH_\otimes \otimes\HH_d)$.  

Now let $x \in \floor{\Z_d}$ so that $D^x \in \II_{X_d}^{1}(\HH_d) \subset \II_{X_d}^{s-1}(\HH_d)$, by Propositition~\ref{soborders}. Again, we see by Lemma~\ref{nowobvious} that $f_d$ is in its domain, and we can compute
\begin{align*}
	\D^{x} (f_d \mid_{\bz}) &= \sum_{w \in \Z_d} f_{d}(\bz,w)\D^{x}(u_x(w)) \\
		&= \sum_{w \in \Z_d} \left[ f(\bz,w) - f_\otimes (\bz,w)\right]\D^{x}(u_d(w)) \\
		&= \D^{x}(f\mid_{\bz}) - \sum_{w\in\floor{\Z_d}}\D^{x}(u_d (w))\,\sum_{z\in\Z_d} f(\bz, z)\D^{w}(u_d (z)) \\
		&= 0,
\end{align*} 
since $\D^{x}(u_d(w))=0$ for all $w \in \floor{\Z_d}$ except $w=x$, for which it equals $1$.
\end{proof}

We are now prepared to state the proof of a version of Theorem~\ref{generaltopdegree} for irreducible unitary representations.

\begin{theorem}[Irreducible version of Theorem~\ref{generaltopdegree}]\label{irreducibletopdegree}
Let $\HH=\HH_1\otimes\dots\otimes\HH_d$ be the Hilbert space of an irreducible unitary representation of $G = G_1 \times\dots\times G_d$ that admits $M$-invariant elements, and let $\tilde\nu_0>0$ be smaller than any $K_i$-invariant eigenvalue of the operator $\Delta_i$ on the restricted representation $G_i\to\U(\HH)$.  For any $s>1$ and $t\leq s-1$, there is a constant $C_{\tilde\nu_0,s,t}$ such that, for every $f\in\ker\II_{X_1,\dots,X_d}^{\s_d}(\HH)^M$, where $\s_d = (2^d + 2^{d-1}-2)s - (2^{d-1}-1)$, there exist $g_1,\dots,g_d \in W^{t}(\HH)$ satisfying the degree-$d$ coboundary equation for $f$, and satisfying the Sobolev estimates $\norm{g_i}_t \leq C_{\tilde\nu_0,s,t}\,\norm{f}_{\s_d}$ for $i=1,\dots,d$.
\end{theorem}

\begin{remark*}
In terms of the $d$-cocycle $\o \in \O_{\RR^d}^d(W^{\s_d}(\HH))$ defined by $\o(X_1,\dots,X_d) = f$, Theorem~\ref{irreducibletopdegree}'s conclusion is that there is a $(d-1)$-form $\eta$ satisfying $\di\eta = \o$ and $\norm{\eta}_t \ll_{\tilde\nu_0, s,t} \norm{\o}_{\s_d}$, where the norms for forms are those expressed in Section~\ref{preliminariesii}.
\end{remark*}

\begin{proof}
The proof is by induction, and Theorem~\ref{two} is the base case, particularly the ``irreducible'' version that comprises the first half of its proof, culminating in the bound~\eqref{irreducibletwo}.  

Let $s > 1$ and $f \in \ker\II_{X_1,\dots,X_d}^{\s_d} (\HH_\otimes \otimes \HH_d)^M$, and let $t \leq s-1$.  By Lemmas~\ref{nowobvious} and~\ref{kernels}, we have $(f_\otimes\mid_w) \in \ker\II_{X_1,\dots, X_{d-1}}^{\s_{d}-1}(\HH_\otimes)^{M_{\times}}$ for all $w \in \Z_d$, and $(f_d \mid_{\bz}) \in \ker\II_{X_d}^{\s_{d}-1}(\HH_d)^{M_d}$ for all $\bz \in \bZ_\times$. Since $\s_d - 1 > \s_{d-1} \geq \s_1=s$, we have the inclusions
\[
	\ker\II_{X_1,\dots, X_{d-1}}^{\s_{d}-1}(\HH_\otimes) \subset \ker\II_{X_1,\dots, X_{d-1}}^{\s_{d-1}}(\HH_\otimes)\quad\mathrm{and}\quad \ker\II_{X_d}^{\s_{d}-1}(\HH_d)\subset\ker\II_{X_d}^{s}(\HH_d),
\]
therefore, the inductive assumption and the base case provide us with $g_{1,w},\dots,g_{d-1,w}\in W^{t}(\HH_\otimes)$ and $g_{d,\bz} \in W^{t}(\HH_d)$ satisfying
\[
	\sum_{i=1}^{d-1} X_{i} \,g_{i,w} = (f_\otimes\mid_w)\quad\textrm{and}\quad X_d\, g_{d,\bz} = (f_d\mid_{\bz}),
\]
and the bounds
\[
	\norm{g_{i,w}}_t \ll_{\tilde\nu_0,s,t} \norm{(f_\otimes\mid_w)}_{\s_{d-1}} \quad\textrm{and}\quad \norm{g_{d,\bz}}_t \ll_{\tilde\nu_0,s,t} \norm{(f_d\mid_{\bz})}_{s}.
\]
We put $g_i (\bz,w):= g_{i,w}(\bz)$ for $i=1,\dots,d-1$, and $g_d (\bz,w):= g_{d,\bz}(w)$, and claim that the $g_i$'s formally define a solution to the coboundary equation.  To see this,
\begin{align*}
	\sum_{i=1}^d X_i \, g_i &= \sum_{i=1}^d \sum_{(\bz,w)\in \bZ_\times \times\Z_d} g_i (\bz,w) X_i (u_\otimes(\bz)\otimes u_d(w)) \\
		&= \sum_{w\in \Z_d}\sum_{i=1}^{d-1} (X_i\, g_{i,w})\otimes u_d(w) + \sum_{\bz\in \bZ_\times} u_\otimes(\bz)\otimes (X_d\, g_{d,\bz}) \\
		&= \sum_{w\in \Z_d} (f_\otimes\mid_w)\otimes u_d(w) + \sum_{\bz\in \bZ_\times} u_\otimes(\bz)\otimes (f_d\mid_{\bz}) \\
		&= f_\otimes + f_d = f.
\end{align*}
Finally, we check the Sobolev norms.  For $i=1,\dots,d-1$,
\begin{align*}
	\norm{g_i}_t^2 &= \sum_{\bz,w}(1+Q_+ (\bz) + Q_d (w))^t \abs{g_i (\bz,w)}^2 \\
		&\leq \sum_{(\bz,w)\in\bZ_\times\times\Z_d}(1+Q_+ (\bz))^t(1 + Q_d (w))^t \abs{g_i (\bz,w)}^2 \\
		&= \sum_{w\in\Z_d}(1 + Q_d (w))^t \norm{g_{i,w}}_{t}^2 \\
		&\ll_{\tilde\nu_0,s,t} \sum_{w\in\Z_d}(1 + Q_d (w))^t \norm{(f_\otimes\mid_w)}_{\s_{d-1}}^2 \\
		&\ll_{\tilde\nu_0,s,t}  \norm{f_\otimes}_{\s_{d-1}+t}^2 \ll_{\tilde\nu_0,s,t}  \norm{f}_{2(\s_{d-1}+s)-1}^2 = \norm{f}_{\s_d}^2,
\end{align*}
where in the last line we have used the bounds in Lemma~\ref{regularity}, and
\begin{align*}
	\norm{g_d}_t^2 &= \sum_{\bz,w}(1+Q_+ (\bz) + Q_d (w))^t \abs{g_d (\bz,w)}^2 \\
		&\leq \sum_{(\bz,w)\in\bZ_\times\times\Z_d}(1+Q_+ (\bz))^t(1 + Q_d (w))^t \abs{g_d (\bz,w)}^2 \\
		&= \sum_{\bz\in\bZ_\times}(1 + Q_+ (\bz))^t \norm{g_{d,\bz}}_{t}^2 \\
		&\ll_{\tilde\nu_0,s,t} \sum_{\bz\in\bZ_\times}(1 + Q_+ (\bz))^t \norm{(f_d\mid_{\bz})}_{s}^2 \\
		&\ll_{\tilde\nu_0,s,t}  \norm{f_d}_{s+t}^2 \ll_{\tilde\nu_0,s,t}  \norm{f}_{4s-1}^2 \leq \norm{f}_{\s_{d}}^2,
\end{align*}
which proves the theorem.
\end{proof}

%========================================================

The proof of Theorem~\ref{generaltopdegree} proceeds as in the proof of Theorem~\ref{two}.  We decompose a unitary representation into irreducibles, and apply Theorem~\ref{irreducibletopdegree} in each.

\begin{proof}[Proof of Theorem~\ref{generaltopdegree}]
Let $\RR^d\curvearrowright S\mathcal{M}\cong M\backslash G/\G$ be as in the theorem statement, and let $s>1$ and $t\leq s-1$. Suppose $f \in W^{\s_d}(M\backslash G/\G)$ is in the kernel of every $\RR^d$-invariant distribution of Sobolev order $\s_d:=\s_d(s)$.  It is identified naturally with $\bar f \in \ker\II_{X_1,\dots,X_d}^{\s_d}(G/\G)^M$. 

We have the direct integral decomposition
\begin{equation}\label{thedecomposition}
	\Lii(G/\G) = \int_{\RR}^\oplus \HH_\s\,ds(\s)
\end{equation}
where $ds$-almost every $\HH_\s$ is an irreducible unitary representation of $G=G_1\times\dots\times G_d$, hence is of the form $\HH_1\otimes\dots\otimes\HH_d$.  The function $\bar f$ decomposes as
\[
	\bar f = \int_{\RR}^\oplus \bar f_\s\,ds(\s)
\]
where $\bar f_\s \in \II_{X_1,\dots,X_d}^{\s_d}(\HH_\s)^M$ for $ds$-almost every $\s$.  For these, Theorem~\ref{irreducibletopdegree} produces elements $\bar g_{1,\s},\dots, \bar g_{d,\s} \in W^t (\HH_\s)$ satisfying the degree-$d$ coboundary equation for $\bar f_\s$, and the estimates $\norm{\bar g_{i,\s}}_t \ll_{\tilde\nu_\s,s,t} \norm{\bar f_\s}_{\s_d}$ on Sobolev norms, where $\tilde\nu_\s$ is a spectral gap for $\Delta_i$ on $K_i$-invariant elements of the restricted representation $G_i \to \U(\HH_\s)$, for all $i=1,\dots, d$.  By Theorem~\ref{spectralgapthm} we can choose $\tilde\nu_\s = \tilde\nu_0$ uniformly across (almost) all representations appearing in the decomposition~\eqref{thedecomposition}.  Therefore, the expressions
\[
	\bar g_i := \int_{\RR}^\oplus \bar g_{i,\s}\,ds(\s)
\]
define a solution $\bar g_1,\dots, \bar g_d \in W^t(G/\G)$ to the degree-$d$ coboundary equation for $\bar f$, with the bounds $\norm{\bar g_i}_t \ll_{\tilde\nu_0,s,t} \norm{\bar f}_{\s_d}$. By Lemma~\ref{wlog}, we may assume without loss of generality that the $\bar g_i$'s are $M$-invariant.  Therefore, they factor to a solution $g_1,\dots, g_d \in W^t(M\backslash G/\G)$ to the coboundary equation
\[
	X_1\,g_1+\dots+X_d\,g_d = f 
\]	
satisfying $\norm{g_i}_t \ll_{\tilde\nu_0,s,t} \norm{f}_{\s_d}$, which proves the theorem.
\end{proof}

%===============================================
\subsection{A brief discussion of the gap between Theorem~\ref{generaltopdegree} and the Katok--Katok Conjecture}\label{whynokk}

Theorem~\ref{generaltopdegree} would constitute a verification of the Katok--Katok Conjecture in top degree for the actions we have considered if it were known that linear combinations of closed orbit measures form a dense subset of the space of invariant distributions.  However, proving this seems to be a formidable problem in itself. Even for $\ZZ^d$-actions on the torus, Katok and Katok prove a result analogous to ours, where obstructions to solving the top-degree coboundary equation come from invariant distributions (``\emph{pseudomeasures},'' in the terminology of~\cite{KK95}). Separately, they show that these are approximable by measures supported on closed orbits; it is one of the main results of their paper~\cite[Theorem~$3$]{KK95}, and is an extension of a result of W.~Veech for toral endomorphisms~\cite{Vee86}. 

It does not seem that the corresponding statement for our Weyl chamber flows can be deduced from the arguments we have presented. In fact, what we have used so far in Part~\ref{partii} does not even require the $\RR^d$-action to be Anosov. In an upcoming paper, we show how our strategy can be adapted to treat the case of \emph{unipotent} $\RR^d$-actions on $\PSL(2,\RR)^d/\G$, arriving at similar results:~obstructions in the top degree come from invariant distributions, and lower-degree cohomologies trivialize. However, since the action is unipotent, we know that there may be invariant distributions \emph{not} approximated by closed orbit measures. This suggests that our methods, though highly adaptable to other situations, may not be strong enough to prove the full Katok--Katok Conjecture on their own.

%===============================================

\section{Lower-degree cohomology and proof of Theorem~\ref{generallowerdegree}}\label{lowerdegreesection}

The aim of this section is to prove Theorem~\ref{generallowerdegree}.  The proof is very similar to the proof in~\cite{Ramhc} for $\SL(2,\RR)\times\dots\times\SL(2,\RR)$.  

%===============================================

Let us define restricted versions of forms.  For an $n$-form $\o \in \O_{\RR^d}^{n}(W^{s}(\HH_{1}\otimes\dots\otimes\HH_d))$, define $\o_\ell \in \O_{\RR^{d}}^{n}(W^{s}(\HH_{1}\otimes\dots\otimes\HH_d))$ to be indexed by $i_1 < \dots < i_n  \subset \{1,\dots,\widehat{\ell},\dots,d\}$,
\[
	\o_{\ell}(X_{i_1},\dots,X_{i_n}) = \o(X_{i_1},\dots,X_{i_n}).
\]
This is just the form $\o$, with the index $\ell$ ``missing.''  Fixing a basis element $u_\ell(z) \in \HH_\ell$, we define a restricted version $(\o_\ell \mid_{z}) \in\O_{\RR^{d-1}}^{n}(W^{s}(\HH_{1}\otimes\dots\otimes\widehat{\HH_\ell}\otimes\dots\otimes\HH_d))$ by
\[
	(\o_{\ell}\mid_{z})(X_{i_1},\dots,X_{i_n}) = (\o(X_{i_1},\dots,X_{i_n}))\mid_{z}.
\]
It is an $n$-form over the $\RR^{d-1}$-action by $X_1,\dots,\widehat{X_\ell},\dots,X_d$ on $\HH_{\otimes,\ell}:=\HH_1 \otimes\dots\otimes\widehat{\HH_\ell}\otimes\dots\otimes\HH_d$.  The following simple lemma shows that if $\o$ is a closed form, then so are $\o_\ell$ and $(\o_{\ell}\mid_{z})$.

\begin{lemma} \label{closed}
	Let $\o \in \O_{\RR^d}^{n}(W^{s}(\HH_{1}\otimes\dots\otimes\HH_{d}))$, with $\di\o=0$.  Then for any $\ell=1,\dots,d$,  $\di(\o_{\ell}\mid_{z})=0$ for all $z \in \Z_\ell$.
\end{lemma}

\begin{proof}
The proof is a simple calculation:~for any $i_1<\dots<i_{n+1} \subset\{1,\dots,\widehat\ell,\dots,d\}$,
\begin{align*}
	\di(\o_{\ell}\mid_{z})(X_{i_1},\dots,X_{i_{n+1}}) &= \sum_{j=1}^{n+1} (-1)^{j+1}X_{i_j}\,(\o_{\ell}\mid_{z})(X_{I_j}) \\
		&= \sum_{j=1}^{n+1} (-1)^{j+1}X_{i_j}\,(\o \mid_{z})(X_{I_j}) \\
		&= \di\o (X_{i_1},\dots,X_{i_{n+1}})\mid_z \\
		&= 0,
\end{align*}
as desired.
\end{proof}

The next proposition is a version of the so-called higher-rank trick.  It shows that for $\o \in \O_{\RR^d}^{d-1}(W^{s}(\HH_1 \otimes\dots\otimes\HH_d))$ with $\di\o=0$, the top-degree cocycle 
\[
(\o_{\ell}\mid_{z}) \in \O_{\RR^{d-1}}^{d-1}(W^{s}(\HH_{\otimes,\ell}))
\]
is in the kernel of all $X_1,\dots, \widehat{X_\ell}, \dots,X_d$-invariant distributions, and hence is exact, for every $z\in \Z_\ell$, by Theorem~\ref{irreducibletopdegree}.

\begin{proposition} \label{higherranktrick}
	Let $\o \in \O_{\RR^{d}}^{d-1}(W^{s}(\HH_{1}\otimes\dots\otimes\HH_d))$ be a closed $(d-1)$-form, and $\ell=1,\dots,d$.  Then for every $z \in \Z_\ell$, we have that 
	\[
	(\o_{\ell}\mid_{z})(X_1,\dots, \widehat{X_\ell},\dots,X_d) \in \ker\II_{X_1,\dots,\widehat{X_\ell},\dots,X_{d}}^{s}(\HH_{\otimes,\ell}).
	\]
\end{proposition}

\begin{proof}
The proof is in~\cite[Lemma~$5.2$]{Ramhc}. We omit it.
\end{proof}

The following is the base case for the inductive argument used in the proof of Theorem~\ref{irreduciblelowerdegree}, which is a version of Theorem~\ref{generallowerdegree} for irreducible unitary representations.

\begin{proposition}[Base case for Theorem~\ref{irreduciblelowerdegree}]\label{basecaselowerdegree}
Suppose 
\[
\SO^\circ(\en_1, 1)\times\SO^\circ(\en_2, 1) \to \U(\HH_1 \otimes\HH_2)
\]
is an irreducible unitary representation with both factors non-trivial, and admitting $M(\en_1)\times M(\en_2)$-invariant elements.  For any $s>1$ and $t\leq s-1$ there is a constant $C_{\nu_0,s,t}>0$ such that whenever $\o \in \O_{\RR^2}^1(W^{4s-1}(\HH)^M)$ is a $1$-cocycle, there is an element $\eta \in W^t (\HH)$, a ``0-form,'' such that $\di\eta=\o$ and $\norm{\eta}_t \leq C_{\nu_0,s,t}\,\norm{\o}_{4s-1}$.
\end{proposition}

\begin{proof}
The closedness condition $\di\o=0$ is equivalent to $f=\o(X_1)$ and $g = \o(X_2)$ satisfying $X_2\,f = X_1\,g$. Proposition~\ref{higherranktrick} implies, in particular, that $(f\mid_{z_2}) \in \ker\II_{X_1}^{2s}(\HH_1)$ for every $z_2 \in \Z_2$, and $(g\mid_{z_1}) \in \ker\II_{X_2}^{2s}(\HH_2)$ for every $z_1 \in \Z_1$. (After all, $2s < 4s -1$.)

Now, Theorem~\ref{two} (or, rather, the irreducible version making up the first half of its proof in Section~\ref{twoproof}) implies that there exists $\eta_{1,z_2} \in W^{2s-1}$ satisfying $X_1 \,\eta_{1,z_2} = (f \mid_{z_2})$ and the bounds 
\begin{equation}\label{sept}
\norm{\eta_{1,z_2}}_t \ll_{\nu_0,s,t} \norm{(f \mid_{z_2})}_s
\end{equation}
for any $t \leq s-1$ and
\begin{equation}\label{oct}
\norm{\eta_{1,z_2}}_{2s-1} \ll_{\nu_0,s,t} \norm{(f \mid_{z_2})}_{2s}.
\end{equation}
We define $\eta_1 \in \HH_1\otimes\HH_2$ by $(\eta_1\mid_{z_2})=\eta_{1,z_2}$. Notice that $X_1\,\eta_1 = f$, and using~\eqref{sept} we have that
\begin{align*}
	\norm{\eta_1}_t^2 &= \sum_{z_1,z_2}(1 +Q_1 (z_1) +Q_2(z_2))^t\,\Abs{\eta_1 (z_1,z_2)}^2 \\
		&\leq \sum_{z_1,z_2}(1 +Q_1 (z_1))^t\,(1 +Q_2(z_2))^t\,\Abs{\eta_1 (z_1,z_2)}^2 \\
		&= \sum_{z_2}(1 +Q_2(z_2))^t\,\Norm{(\eta_1\mid_{z_2})}_t^2 \\
		&\ll_{\nu_0,s,t} \sum_{z_2}(1 +Q_2(z_2))^t\,\Norm{(f\mid_{z_2})}_s^2 \leq \Norm{f}_{2s-1}^2.
\end{align*}
The same reasoning for $g$ will allow us to obtain $\eta_2 \in \HH_1\otimes\HH_2$ with $X_2\,\eta_2 = g$ and a similar Sobolev estimate. Similar calculations using~\eqref{oct} instead of~\eqref{sept} will show that $\eta_1, \eta_2 \in W^{2s-1}(\HH_1\otimes\HH_2)$.

We claim that $\eta_1$ and $\eta_2$ coincide. First, since $\eta_1-\eta_2 \in W^{2s-1}$, and $2s-1>1$, we know in particular that $X_2\,(\eta_1-\eta_2) \in \HH_1\otimes\HH_2$ has positive Sobolev order. Also, since $X_1$ and $X_2$ commute, we can show that $X_2\,(\eta_1 - \eta_2)$ is a solution to the equation $X_1\,u = 0$, because $X_2\,f - X_1\,g=0$. But $X_1\,u=0$ has no non-trivial solutions of positive Sobolev order in $\HH_1$, so we must have $X_2\,(\eta_1-\eta_2)\mid_{z_2} = 0$ for every $z_2 \in \Z_2$, which implies that $X_2\,(\eta_1 - \eta_2)=0$. We also know that $X_2\,u=0$ has no non-trivial solutions of positive order in $\HH_2$, so the same reasoning leads to $\eta_1 - \eta_2=0$. Therefore, $\eta=\eta_1=\eta_2 \in W^t (\HH_1\otimes\HH_2)$ is our desired ``$0$-form'' and the proposition is proved.
\end{proof}

We will also need the following easy lemma. 

\begin{lemma}\label{easylemma}
If $f \in W^{s+1} (\HH_1 \otimes\dots\otimes\HH_d)^M$ then $\Norm{X_1 f}_s \ll \Norm{f}_{s+1}$.
\end{lemma}

\begin{remark}
Though it is not a proof, it may be instructive to note that from the point of view of another standard definition of Sobolev norm, namely,
\[
\Norm{f}_s^2 = \sum_{\{V_{i_1}, \dots, V_{i_s}\}\subset\so(\en,1)}\Norm{V_{i_1}V_{i_2}\dots V_{i_s}f}^2,
\]
this lemma is obvious, and in fact $\Norm{X_1 f}_s \leq \Norm{f}_{s+1}$. This other norm is equivalent to the norm we are using, meaning that the two are asymptotic in the sense of ``$\asymp$''.  However, the asymptote may depend on the representation.  We therefore prove Lemma~\ref{easylemma} by a calculation.
\end{remark}

\begin{proof}
Let $f \in W^{s+1}(\HH_1\otimes\dots\otimes\HH_d)^M$. Then
\begin{align*}
	\Norm{X_1\,f}_s^2 &= \sum_{\bz} (1+Q(\bz))^s\,\Abs{X_1\,f (\bz)}^2 \\
		&\ll \sum_{\bz} (1+Q(\bz))^s\,\Abs{A_{k_1}^+(z_1)\,f (\bz)}^2,
	\intertext{which, according to Lemmas~\ref{Akbound} and~\ref{Bkbound},}
		&\ll \sum_{\bz} (1+Q(\bz))^{s+1}\,\Abs{f (\bz)}^2 = \norm{f}_{s+1}^2
\end{align*}
which proves the lemma.
\end{proof}
%==========================
%====================================================

Finally, we define the number $\vars_d (s)$ by the following rule. First, for any $s \geq 1$, we set the number $\vars_2 (s) = 4s-1$. For $d\geq 3$, we put
\[
	\vars_d (s) = \max\left\{\begin{matrix}\vars_{d-1} (\vars_{d-1}+s+1)+\vars_{d-1}(s)+s, \\ \s_{d-1} (\vars_{d-1}+s+1)+\vars_{d-1}(s)+s\end{matrix}\right\},
\]
where $\s_d(s)$ is as in Theorem~\ref{generaltopdegree}. We are now prepared to state the proof of Theorem~\ref{generallowerdegree} for irreducible unitary representations.

\begin{theorem}[Irreducible version of Theorem~\ref{generallowerdegree}]\label{irreduciblelowerdegree}
Let $\HH=\HH_1 \otimes\dots\otimes \HH_d$ be an irreducible representation of $G$ with no trivial factor, admitting $M$-invariant elements, and let $\tilde\nu_0>0$ be smaller than any $K_i$-invariant eigenvalue of the operator $\Delta_i$.  For any $s >1$ and $t \leq s-1$, there is a constant $C_{\tilde\nu_0,s,t}>0$ such that for any $n$-cocycle $\o \in \O_{\RR^d}^{n} (W^{\varsigma_d} (\HH)^M)$ with $1\leq n\leq d-1$, there exists $\eta \in \O_{\RR^d}^{n-1} (W^{t}(\HH))$ with $\di\eta = \o$ and $\norm{\eta}_t \leq C_{\tilde\nu_0,s,t}\,\norm{\o}_{\vars_d}$.
\end{theorem}

\begin{proof}
The proof is an adaptation of the inductive argument found in~\cite[p. 25]{KK95}, the base case being that of $1$-cocycles over $\RR^2$-actions, given by Proposition~\ref{basecaselowerdegree}.  Suppose the result holds for $\RR^p$-actions, whenever $2\leq p\leq d-1$.  Let $s>1$ and $t \leq s-1$.

For every $z \in \Z_1$, $(\o_1 \mid_z)$ is an $n$-form over the $\RR^{d-1}$-action on $\HH_{\otimes,1}$ generated by the elements $X_2,\dots,X_d$, and Lemma~\ref{closed} guarantees that it is closed.

If $n <d-1$, then by the induction hypothesis there is an $(n-1)$-form 
\[
\eta_{1,z} \in\O_{\RR^{d-1}}^n (W^{\vars_{d-1}+s} (\HH_{\otimes,1}))
\]
satisfying the coboundary equation $\di\eta_{1,z} =(\o_1 \mid_z)$ and
\begin{equation}\label{etabound1}
	\Norm{\eta_{1,z}}_{\vars_{d-1}+s} \ll_{\tilde\nu_0,s,t} \Norm{(\o_1 \mid_z)}_{\vars_{d} - \vars_{d-1}-s}
\end{equation}
and
\[
	\Norm{\eta_{1,z}}_t \ll_{\tilde\nu_0,s,t}\Norm{(\o_1 \mid_z)}_{\vars_{d-1}}.
\]
This is because $\vars_d(s) - \vars_{d-1}(s)-s\geq \vars_{d-1}(\vars_{d-1}(s) +s+1)$, so we are essentially applying the inductive hypothesis with $\vars_{d-1}(s)+s+1$ in the role that $s$ fills in the theorem statement.

On the other hand, if $n=d-1$, then $(\o_1 \mid_z)$ is a top-degree form over the $\RR^{d-1}$-action, and Proposition~\ref{higherranktrick} implies that
\[
	(\o_1 \mid_z)(X_2,\dots,X_d) \in \ker\II_{X_2, \dots,X_d}^{\vars_d} (\HH_{\otimes,1})^M,
\]
which in turn implies, by Theorem~\ref{irreducibletopdegree}, that there is an $(n-1)$-form $\eta_{1,z}$ with $\di\eta_{1,z}=(\o_1\mid_z)$ and satisfying
\begin{equation}\label{etabound2}
	\Norm{\eta_{1,z}}_{\vars_{d-1}+s} \ll_{\tilde\nu_0,s,t}\Norm{(\o_1\mid_z)}_{\vars_d - \vars_{d-1}-s}
\end{equation}
and
\[
	\Norm{\eta_{1,z}}_{t} \ll_{\tilde\nu_0,s,t}\Norm{(\o_1\mid_z)}_{\vars_{d-1}}.
\]
This is because $\vars_d(s)- \vars_{d-1}(s)-s \geq \s_{d-1}(\vars_{d-1}(s)+s+1)$.

We now define $\eta_1 \in \O_{\RR^d}^{n-1}(W^{\vars_{d-1}+s} (\HH_1\otimes\dots\otimes\HH_d))$ by $(\eta_1 \mid_z)=\eta_{1,z}$ for all $z\in\Z_1$.  Notice that
\begin{align}
	\Norm{\eta_1}_{\vars_{d-1}+s}^2 &= \sum_{1\leq i_1<\dots<i_{n-1}\leq d}\,\sum_{\bz\in\bZ}(1 + Q_+ (\bz))^{\vars_{d-1}+s} \Abs{\eta_1 (X_{i_1},\dots,X_{i_{n-1}})(\bz)}^2 \nonumber \\
		&\leq \sum_{1\leq i_1<\dots<i_{n-1}\leq d}\,\sum_{z_1 \in\Z_1}(1 + Q_{1} (z_1))^{\vars_{d-1}+s} \nonumber\\
		&\indent\times\sum_{\bz_\times \in\bZ_{\times,1}}(1 + Q_{+,1} (\bz_\times))^{\vars_{d-1}+s} \Abs{\eta_1 (X_{i_1},\dots,X_{i_{n-1}})(\bz_\times)}^2 \nonumber \\
		&= \sum_{z \in\Z_1}(1 + Q_{1} (z))^{\vars_{d-1}+s}\Norm{(\eta_1\mid_{z})}_{\vars_{d-1}+s}^2 \nonumber
\intertext{and, by~\eqref{etabound1} and~\eqref{etabound2},}
		&\ll_{\tilde\nu_0,s,t} \sum_{z\in\Z_1}(1 + Q_{1} (z))^{\vars_{d-1}+s} \Norm{(\o_1\mid_z)}_{\vars_d - \vars_{d-1}-s}^2 \nonumber \\
		&\ll_{\tilde\nu_0,s,t} \Norm{\o_1}_{\vars_d}^2 \leq \Norm{\o}_{\vars_d}^2, \label{operationfailed}
\end{align}
which takes care of the components of $\o$ that do not contain the index $1$.

It is now left to write a solution for the parts of $\o$ that \emph{do} contain the index 1.  Consider
\[
	\theta (X_{i_2},\dots,X_{i_n}) = \o(X_1,X_{i_2},\dots,X_{i_n}) - X_1\, \eta_1 (X_{i_2},\dots,X_{i_n})
\]
and note that $\theta$ is an element of $\O_{\RR^d}^{n-1}(W^{\vars_{d-1}+s-1}(\HH))$ and for any $\t \leq \vars_{d-1}+s-1$,
\begin{align*}
	\Norm{\theta (X_{i_2},\dots,X_{i_n})}_{\t} &\leq \Norm{\o (X_1, X_{i_2},\dots,X_{i_n})}_{\t} + \Norm{X_1\eta_1 (X_{i_2},\dots,X_{i_n})}_{\t}.
\end{align*}
By Lemma~\ref{easylemma} and~\eqref{operationfailed}, 
\[
	\Norm{\theta (X_{i_2},\dots,X_{i_n})}_{\t} \ll \Norm{\o (X_1, X_{i_2},\dots,X_{i_n})}_{\t} + \Norm{\eta_1 (X_{i_2},\dots,X_{i_n})}_{\t+1} \ll_{\nu_0,s,t} \Norm{\o}_{\vars_d},
\]
and this implies that
\begin{equation}\label{thetabound}
	\Norm{\theta}_{\t} \ll_{\tilde\nu_0,s,t} \Norm{\o}_{\vars_d}.
\end{equation}

We claim that $(\theta_1\mid_z)=(\theta\mid_z) \in \O_{\RR^{d-1}}^{n-1}(W^{\vars_{d-1}+s-1}(\HH_{\otimes,1}))$ is closed for any $z \in \Z_1$.  This is a calculation.  By recalling that $\di\o=0$ and that $\di(\eta_1\mid_z) =(\o_1\mid_z)$ for all $z \in \Z_1$, and lastly that $\o_1 (X_{i_1},\dots,X_{i_n}) = \o(X_{i_1},\dots,X_{i_n})$ whenever $i_1 >1$, we get
\[
	\di(\theta\mid_z) (X_{i_1},\dots,X_{i_n}) = X_1\,(\o\mid_z)(X_{i_1},\dots,X_{i_n}) - X_1\,\sum_{j=1}^{n}(-1)^{j+1} X_{i_j}\,(\eta_1\mid_z) (X_{I_j}) =0
\]
whenever $i_1 >1$. 

Therefore, the induction hypothesis guarantees the existence of $\k_z \in \O_{\RR^{d-1}}^{n-2}(W^t(\HH_{\otimes,1}))$ satisfying the coboundary equation $\di\k_z=(\theta\mid_z)$ and the estimate
\begin{equation}
	\Norm{\k_z}_t \ll_{\tilde\nu_0,s,t}\Norm{(\theta \mid_z)}_{\vars_{d-1}}. \label{kappabound}
\end{equation}
We finally define $\eta$ by 
\[
	(\eta\mid_z) (X_{i_1}, X_{i_2},\dots,X_{i_{n-1}}) = \begin{cases} \k_z (X_{i_2},\dots,X_{i_{n-1}}) &\textrm{if}\quad i_1 = 1\\
																			(\eta_1\mid_z) (X_{i_1},\dots,X_{i_{n-1}}) &\textrm{if}\quad i_1 \neq 1.
																\end{cases}
\]	
Then $\di\eta=\o$ and the desired Sobolev estimates are met. Namely,
\begin{align*}
	\Norm{\eta (X_1, X_{i_2},\dots,X_{i_{n-1}})}_t^2 &= \sum_{\bz\in\bZ}(1 + Q_+ (\bz))^t \Abs{\eta (X_1,X_{i_2},\dots,X_{i_{n-1}})(\bz)}^2 \\
		&\leq \sum_{z_1 \in\Z_1}\sum_{\bz_\times \in\bZ_{\times,1}}(1 + Q_{1} (z_1))^t \\
		&\indent\times(1 + Q_{+,1} (\bz_\times))^t \Abs{\eta (X_1, X_{i_2},\dots,X_{i_{n-1}})(\bz_\times)}^2 \\
		&= \sum_{z \in\Z_1}(1 + Q_{1} (z))^t\Norm{(\eta \mid_{z}) (X_1, X_{i_2},\dots,X_{i_{n-1}})}_t^2 \\
		&\leq \sum_{z \in\Z_1}(1 + Q_{1} (z))^t\Norm{(\eta \mid_{z})}_t^2
\intertext{and, by~\eqref{kappabound}, this is}
		&\ll_{\tilde\nu_0,s,t} \sum_{z\in\Z_1}(1 + Q_{1} (z))^t \Norm{(\theta\mid_z)}_{\vars_{d-1}}^2 \ll_{\tilde\nu_0,s,t} \Norm{\theta}_{\vars_{d-1}+s-1}^2
\end{align*}
which, by~\eqref{thetabound}, is controlled by $\ll_{\tilde\nu_0,s,t} \Norm{\o}_{\vars_d}^2$.  This implies that $\norm{\eta}_t \ll_{\tilde\nu_0,s,t} \norm{\o}_{\vars_d}$ and completes the proof of the theorem.
\end{proof}

%===============================================

The proof of Theorem~\ref{generallowerdegree} is very similar to the proofs of Theorems~\ref{two} and~\ref{generaltopdegree}.  We decompose a unitary representation into irreducibles, and apply the irreducible version of the theorem, in this case, Theorem~\ref{irreduciblelowerdegree}.

\begin{proof}[Proof of Theorem~\ref{generallowerdegree}]
Let $\RR^d\curvearrowright S\mathcal{M}\cong M\backslash G/\G$ be as in the theorem statement and let $s>1$ and $t \leq s-1$. Suppose $\o \in \O_{\RR^d}^n (W^{\vars_d}(M\backslash G/\G))$ is an $n$-cocycle (\emph{i.e.} $\di\o=0$), where $1\leq n\leq d-1$.  We naturally identify it with $\bar\o \in \O_{\RR^d}^n (W^{\vars_d}(G/\G)^M)$, an $n$-cocycle defined by $M$-invariant functions of Sobolev order $\vars_d:=\vars_d (s)$.

There is the direct integral decomposition~\eqref{thedecomposition}, and the functions defining $\bar\o$ decompose accordingly, defining, for each $\s$ where $\HH_\s$ is irreducible (which is $ds$-almost all of them), an $n$-cocycle $\bar\o_\s \in \O_{\RR^d}^n(W^{\vars_d}(\HH_\s))$. Here, Theorem~\ref{irreduciblelowerdegree} produces an $(n-1)$-form $\bar\eta_\s \in \O_{\RR^d}^{n-1}(W^{t}(\HH_\s))$ satisfying $\di\bar\eta_\s = \bar\o_\s$ and the bound $\norm{\bar\eta_\s}_t \ll_{\tilde\nu_\s,s,t}\norm{\bar\o_\s}_{\vars_d}$, where $\tilde\nu_\s$ is a spectral gap for all the $\Delta_i$ in the restrictions $G_i \to \U(\HH_\s)$.  Theorem~\ref{spectralgapthm} allows us to choose the same $\tilde\nu_\s = \tilde\nu_0$ for all irreducible representations appearing in the decomposition~\eqref{thedecomposition}, which in turn allows us to write a solution $\bar\eta \in \O_{\RR^d}^{n-1}(W^t(G/\G))$ to $\di\bar\eta = \bar\o$ by defining
\[
	\bar\eta (X_{i_1},\dots,X_{i_{n-1}}) := \int_{\RR}^{\oplus} \bar\eta_\s (X_{i_1},\dots,X_{i_{n-1}})\,ds(\s)
\]
for all $1\leq i_1 <\dots<i_{n-1}\leq d$. We now have $\norm{\bar\eta}_t \ll_{\tilde\nu_0,s,t} \norm{\bar\o}_{\vars_d}$.  By Lemma~\ref{wlog}, we may assume that $\bar\eta$ is defined by $M$-invariant functions, so that it factors to a solution $\eta \in \O_{\RR^d}^{n-1}(W^t(M\backslash G/\G))$ to the coboundary equation $\di\eta=\o$ satisfying the required estimate $\norm{\eta}_t \ll_{\tilde\nu_0,s,t} \norm{\o}_{\vars_d}$ on Sobolev norms.
\end{proof}
%===============================================

\section{Proof of Theorem~\ref{one}}\label{oneproof}
\begin{proof}[Proof of Theorem~\ref{one}]
We have essentially already proved Theorem~\ref{one}. It is automatically implied by Theorems~\ref{generaltopdegree} and~\ref{generallowerdegree}, and~\cite[Theorems~$6.2$ and~$6.3$]{Ramhc}, since the space $\Cinf(\HH)$ of smooth vectors in any representation coincides with the intersection of all Sobolev spaces $W^s(\HH)$ of positive order.
\end{proof}
%===============================================

\appendix

\section{Computations for $A$'s and $\a$'s}\label{appendixa}

This appendix contains calculations involving the coefficients $A_j^{\pm}$, $j=1,\dots,k$, defined by~\eqref{defA}, and the coefficients $\a$, from Definition~\ref{defeven}.

%===============================================
%=============================================================

\subsection{Bounds for $A_k^+$}\label{Akbound}

\begin{lemma}\label{coeffseven}
There is a positive constant $C_\nu$ depending only on $\nu$ and satisfying
\[
\Abs{A_k (\bm,\l)}^2 \geq C_{\nu}\,\left(\ceil{\bm}-\ceil{\bn}+1\right)\left(\ceil{\bm}-\ceil{\l}+1\right)
\]
for all $\bm\in M_\l$.  We also have that
\[
\Abs{A_k (\bm,0)}^2 \geq C_{\nu_0}\,\left((\ceil{\bm}+1)^2-\ceil{\bn}^2\right)
\]
holds whenever $\l\equiv 0$, where $C_{\nu_0}$ depends only on some choice of $\nu_0 \in \big[0,\frac{\en-1}{2}\big)$ that is closer to the right end-point of that interval than $\nu$ is. Finally, there is a $C>0$ such that 
\[
\Abs{A_k (\bm,0)}^2 \leq C\,\left(1 + \tilde\nu + 2\ceil{\bm}^2-\ceil{\bn}^2\right)
\]
for all $\ceil{\bm}\geq\ceil{\bn}$.
\end{lemma}

\begin{proof}
From~\eqref{defA} we have
\begin{equation}\label{becomes}
4\Abs{A_k^+ (\bm,\l)}^2 = \Abs{\nu^2-\left(y_k + \frac{1}{2}\right)^2}\,\prod_{r=1}^{k-1} \left[\frac{\left(x_r - \frac{1}{2}\right)^2-\left(y_k + \frac{1}{2}\right)^2}{y_r^2 - y_k^2}\right]\,\left[\frac{z_r^2-\left(y_k + \frac{1}{2}\right)^2}{y_r^2 - (y_k +1)^2}\right].
\end{equation}
Examining the inequalities in Figure~\ref{figarray}, we see that $y_r \geq x_{r-1}$ and $y_r \geq z_{r-1}+\frac{1}{2}$ for $r=2,\dots,k-1$, so
\begin{multline*}
4\Abs{A_k^+ (\bm,\l)}^2 \geq \Abs{\nu^2-\left(y_k + \frac{1}{2}\right)^2}\,\left[\frac{\left(x_{k-1} - \frac{1}{2}\right)^2-\left(y_k + \frac{1}{2}\right)^2}{y_1^2 - y_k^2}\right]\,\left[\frac{z_{k-1}^2-(y_k + \frac{1}{2})^2}{\left(y_1+\frac{1}{2}\right)^2 - (y_k +1)^2}\right]\\
	\times\prod_{r=1}^{k-2} \left[\frac{\left(x_r - \frac{1}{2}\right)^2-\left(y_k + \frac{1}{2}\right)^2}{x_r^2 - y_k^2}\right]\,\left[\frac{z_r^2-(y_k + \frac{1}{2})^2}{\left(z_r+\frac{1}{2}\right)^2 - (y_k +1)^2}\right]
\end{multline*}
We re-write,
\begin{multline*}
4\Abs{A_k^+ (\bm,\l)}^2 \geq \Abs{\nu^2-\left(y_k + \frac{1}{2}\right)^2}\,\left[\frac{\left(y_k + \frac{1}{2}\right)^2 - \left(x_{k-1} - \frac{1}{2}\right)^2}{y_k^2 - y_1^2}\right]\,\left[\frac{\left(y_k + \frac{1}{2}\right)^2 - z_{k-1}^2}{(y_k +1)^2 - \left(y_1+\frac{1}{2}\right)^2}\right]\\
	\times\prod_{r=1}^{k-2} \left[1 + \frac{1}{y_k - x_r}\right]\,\left[1- \frac{1}{y_k +z_r+\frac{3}{2}}\right].
\end{multline*}
By comparing $x_r \geq y_r+1$ and $z_r \geq y_r+\frac{1}{2}$, we see that the last line is bounded below by $1$, so we are left with
\[
4\Abs{A_k^+ (\bm,\l)}^2 \geq \Abs{\nu^2-\left(y_k + \frac{1}{2}\right)^2}\,\left[\frac{\left(y_k + \frac{1}{2}\right)^2 - \left(x_{k-1} - \frac{1}{2}\right)^2}{y_k^2 - y_1^2}\right]\,\left[\frac{\left(y_k + \frac{1}{2}\right)^2 - z_{k-1}^2}{(y_k +1)^2 - \left(y_1+\frac{1}{2}\right)^2}\right].
\]
Translating according to the definitions of $x,y,z$, we have 
\begin{multline}\label{similarsituation}
4\Abs{A_k^+ (\bm,\l)}^2 \geq \Abs{\nu^2-\left(\ceil{\bm} +k- \frac{1}{2}\right)^2} \\
	\times\left[\frac{\left(\ceil{\bm} +k- \frac{1}{2}\right)^2 - \left(\ceil{\l}+k - \frac{3}{2}\right)^2}{\left(\ceil{\bm}+k-1\right)^2 - m_1^2}\right]\,\left[\frac{\left(\ceil{\bm} +k- \frac{1}{2}\right)^2 - \left(\ceil{\bn}+k-\frac{3}{2}\right)^2}{(\left(\ceil{\bm}+k\right)^2 - \left(m_1+\frac{1}{2}\right)^2}\right].
\end{multline}
We easily see that there is some $C>0$ such that
\begin{align*}
	&\geq \Abs{\nu^2-\left(\ceil{\bm} + k - \frac{1}{2}\right)^2}\cdot C\,\left(\frac{\ceil{\bm} -\ceil{\bn} + 1}{\ceil{\bm} + 1}\right)\left(\frac{\ceil{\bm} -\ceil{\l} + 1}{\ceil{\bm} + 1}\right)  \\
	&\geq C_{\nu}\,\left(\ceil{\bm}-\ceil{\bn}+1\right)\left(\ceil{\bm}-\ceil{\l}+1\right) 
\end{align*}
for some $C_\nu >0$.  This is the lower bound in Lemma~\ref{coeffseven}.

For the second part, the expression~\eqref{becomes} becomes
\begin{multline}\label{obtainedsimilarly}
4\Abs{A_k^+ (\bm,0)}^2 = \Abs{\nu^2-\left(\ceil{\bm} + k - \frac{1}{2}\right)^2} \\
	\times \prod_{r=1}^{k-2} \frac{\left[(r - \frac{1}{2})^2-(\ceil{\bm} + k - \frac{1}{2})^2\right]\left[(r -\frac{1}{2})^2-(\ceil{\bm} + k - \frac{1}{2})^2\right]}{\left[(r - 1)^2 - (\ceil{\bm} + k - 1)^2\right][(r - 1)^2 - (\ceil{\bm} + k)^2]} \\
	\times \frac{\left[(k - \frac{3}{2})^2-(\ceil{\bm} + k - \frac{1}{2})^2\right]\left[(\ceil{\bn} + k -\frac{3}{2})^2-(\ceil{\bm} + k - \frac{1}{2})^2\right]}{\left[(k - 2)^2 - (\ceil{\bm} + k - 1)^2\right][(k - 2)^2 - (\ceil{\bm} + k)^2]}.
\end{multline}
The middle term is positive and approaches $1$ as $\ceil{\bm}\to\infty$, so is bounded below by some constant $C>0$ and above by some other constant $C'>0$.  Taking this into account, and re-writing the last term,
\begin{align*}
4\Abs{A_k^+ (\bm,0)}^2 &\geq C \Abs{\nu^2-\left(\ceil{\bm} + k - \frac{1}{2}\right)^2} \frac{(\ceil{\bm}+\ceil{\bn}+2k-2)(\ceil{\bm}-\ceil{\bn}+1)}{(\ceil{\bm}+2k-3)(\ceil{\bm}+2)} \\
	&\geq C \Abs{\nu_0^2-\left(\ceil{\bm} + k - \frac{1}{2}\right)^2} \frac{(\ceil{\bm}+\ceil{\bn}+2k-2)(\ceil{\bm}-\ceil{\bn}+1)}{(\ceil{\bm}+2k-3)(\ceil{\bm}+2)}
\end{align*}
where $\nu_0 \in \big[0,\frac{\en-1}{2}\big)$ as in the lemma statement.  (For example, if we are in a representation from the principal series, then $\nu \in i\RR$, so we can just take $\nu_0 = 0$.)  From here it is easy to see that there is some $C_{\nu_0}>0$ depending only on $\nu_0$  such that 
\[
	\Abs{A_k^+ (\bm,0)}^2 \geq C_{\nu_0}\,(\ceil{\bm}+\ceil{\bn}+1)(\ceil{\bm}-\ceil{\bn}+1),
\]
as required.

An upper bound is obtained similarly.  From~\eqref{obtainedsimilarly} and the succeeding discussion, we have
\begin{align*}
4\Abs{A_k^+ (\bm,0)}^2 &\leq C' \Abs{\nu^2-\left(\ceil{\bm} + k - \frac{1}{2}\right)^2} \frac{(\ceil{\bm}+\ceil{\bn}+2k-2)(\ceil{\bm}-\ceil{\bn}+1)}{(\ceil{\bm}+2k-3)(\ceil{\bm}+2)} \\
	&= C' \left(\ceil{\bm}^2 + (2k-1)\ceil{\bm} +\tilde\nu\right)\frac{(\ceil{\bm}+\ceil{\bn}+2k-2)(\ceil{\bm}-\ceil{\bn}+1)}{(\ceil{\bm}+2k-3)(\ceil{\bm}+2)}
\end{align*}
and from here we easily see that by modifying $C'$,
\[
4\Abs{A_k^+ (\bm,0)}^2 \leq C' \left(1 + \tilde\nu + 2\ceil{\bm}^2 -\ceil{\bn}^2\right)
\]
proving the upper bound.
\end{proof}

%===============================================
%===============================================

\subsection{Bounds for $\a$'s}\label{abound}

\begin{lemma}\label{alphas}
There is a positive constant $C_{\nu}$ depending on $k$ and $\nu$ such that
\[
	\Abs{\a_{e_k \pm e_j}(\bm)}^2 \leq C_{\nu}\,\frac{\min\{\ceil{\bn},\ceil{\l}\}^2}{\left(\ceil{\bm}+1\right)^2}\cdot\frac{\ceil{\l}^2}{\left(\ceil{\bm} + 1\right)^2 - \ceil{\l}^2}\cdot\frac{\ceil{\bn}^2}{\left(\ceil{\bm} + 1\right)^2 -  \ceil{\bn}^2}.
\]
Also, we have $\abs{\a_{2e_k}(\bm)}\leq 1$ for all $\bm$.
\end{lemma}

\begin{comment}
\begin{lemma}\label{alphas}
There is a positive constant $C_{\nu}$ depending on $k$ and $\nu$ such that
\[
	\Abs{\a_{e_k \pm e_j}(\bm)}^2 \leq C_{\nu}\,\frac{\min\{\ceil{\bn},\ceil{\l}\}^2}{\left(\ceil{\bm}+1\right)^2}\cdot\frac{\ceil{\l}^2}{\left(\ceil{\bm} + 1\right)^2 - \ceil{\l}^2}\cdot\frac{\ceil{\bn}^2}{\left(\ceil{\bm} + 1\right)^2 -  \ceil{\bn}^2}.
\]
If $\l \not\equiv 0$,
\[
	\Abs{\a_{2e_k}(\bm)}^2 \leq \left[1 + \frac{2}{\ceil{\bm} - \ceil{\l} + 2}\right]\left[1 + \frac{2}{\ceil{\bm} -\ceil{\bn} +2}\right]\left[1 + \frac{2}{\ceil{\bm}+1}\right].
\]
If $\l\equiv 0$,
\[
	\Abs{\a_{2e_k}(\bm)}^2 \leq \left[1 + \frac{2}{\ceil{\bm} -\ceil{\bn} +2}\right].
\]
\end{lemma}
\end{comment}

\begin{proof}
For the first inequality, we have from Definition~\ref{defeven} that $\a_{2e_k}(\bm) = A_j^+ (\bm)/A_k^+ (\bm+e_j)$.  We divide the definitions~\eqref{defA} of $A_j$ and $A_k$ to get
\begin{multline*}
	\Abs{\a_{e_k + e_j}(\bm)}^2 = \frac{\prod_{r=1}^{k-1}\left[\left(x_r - \frac{1}{2}\right)^2-\left(y_j + \frac{1}{2}\right)^2\right]\left[z_r^2-\left(y_j + \frac{1}{2}\right)^2\right]\cdot\left[\nu^2-\left(y_j + \frac{1}{2}\right)^2\right]}{\prod_{r\neq j}(y_r^2 - y_j^2)[y_r^2 - (y_j + 1)^2]} \\
		\times \frac{\prod_{r\neq j,k}(y_r^2 - y_k^2)[y_r^2 - (y_k + 1)^2]}{\prod_{r\neq j,k}\left[\left(x_r - \frac{1}{2}\right)^2-\left(y_k + \frac{1}{2}\right)^2\right]\left[z_r^2-\left(y_k + \frac{1}{2}\right)^2\right]\cdot\left[\nu^2-\left(y_k + \frac{1}{2}\right)^2\right]} \\
		\times \frac{\left[(y_j +1)^2 - y_k^2\right][(y_j +1)^2 - (y_k + 1)^2]}{\left[\left(x_j - \frac{1}{2}\right)^2-\left(y_k + \frac{1}{2}\right)^2\right]\left[z_j^2-\left(y_k + \frac{1}{2}\right)^2\right]}
\end{multline*}
We re-arrange the terms in the following way:
\begin{multline*}
	\Abs{\a_{e_k + e_j}(\bm)}^2 = \Abs{\frac{\left(y_j +\frac{1}{2}\right)^2 - \nu^2}{\left(y_k +\frac{1}{2}\right)^2 - \nu^2}}\,\left[\frac{y_k^2 - (y_j+1)^2}{y_k^2 - y_j^2}\right]\,\left[\frac{(y_k+1)^2 - (y_j+1)^2}{y_k^2 - (y_j+1)^2}\right] \\
		\times \left[\frac{\left(y_j+\frac{1}{2}\right)^2 - \left(x_j-\frac{1}{2}\right)^2}{\left(y_k+\frac{1}{2}\right)^2 - \left(x_j-\frac{1}{2}\right)^2}\right]\,\left[\frac{\left(y_j+\frac{1}{2}\right)^2 - z_j^2}{\left(y_k+\frac{1}{2}\right)^2 - z_j^2}\right]\\
		\times\prod_{r\neq j,k} \left[\frac{\left(y_j+\frac{1}{2}\right)^2 - \left(x_r-\frac{1}{2}\right)^2}{\left(y_k+\frac{1}{2}\right)^2 - \left(x_r-\frac{1}{2}\right)^2}\right]\,\left[\frac{\left(y_j+\frac{1}{2}\right)^2 - z_r^2}{\left(y_k+\frac{1}{2}\right)^2 - z_r^2}\right]\,\left[\frac{y_k^2-y_r^2}{y_j^2-y_r^2}\right]\,\left[\frac{(y_k+1)^2-y_r^2}{(y_j+1)^2-y_r^2}\right].
\end{multline*}
Now, since for $a>b$, the expression $\left(\frac{a-t^2}{b-t^2}\right)$ is an increasing function of $t\geq 0$, and because (from the inequalities in Figure~\ref{figarray}) we know that $y_r \leq \min\{x_r -1, z_r -1/2\}$, we can replace

\[
	\frac{y_k^2-y_r^2}{y_j^2-y_r^2}\textrm{ with }\frac{y_k^2-\left(z_r-\frac{1}{2}\right)^2}{y_j^2-\left(z_r-\frac{1}{2}\right)^2}\quad\textrm{and}\quad\frac{(y_k+1)^2-y_r^2}{(y_j+1)^2-y_r^2}\textrm{ with }\frac{(y_k+1)^2-\left(x_r-1\right)^2}{(y_j+1)^2-(x_r-1)^2},
\]
in the last two factors above, leaving
\begin{multline*}
	\Abs{\a_{e_k + e_j}(\bm)}^2 = \Abs{\frac{\left(y_j +\frac{1}{2}\right)^2 - \nu^2}{\left(y_k +\frac{1}{2}\right)^2 - \nu^2}}\,\left[\frac{y_k^2 - (y_j+1)^2}{y_k^2 - y_j^2}\right]\,\left[\frac{(y_k+1)^2 - (y_j+1)^2}{y_k^2 - (y_j+1)^2}\right] \\
		\times \left[\frac{\left(y_j+\frac{1}{2}\right)^2 - \left(x_j-\frac{1}{2}\right)^2}{\left(y_k+\frac{1}{2}\right)^2 - \left(x_j-\frac{1}{2}\right)^2}\right]\,\left[\frac{\left(y_j+\frac{1}{2}\right)^2 - z_j^2}{\left(y_k+\frac{1}{2}\right)^2 - z_j^2}\right]\\
		\times\prod_{r\neq j,k} \left[\frac{\left(y_j+\frac{1}{2}\right)^2 - \left(x_r-\frac{1}{2}\right)^2}{\left(y_k+\frac{1}{2}\right)^2 - \left(x_r-\frac{1}{2}\right)^2}\right]\,\left[\frac{(y_k+1)^2-\left(x_r-1\right)^2}{(y_j+1)^2-(x_r-1)^2}\right]\\
		\times\prod_{r\neq j,k}\left[\frac{\left(y_j+\frac{1}{2}\right)^2 - z_r^2}{\left(y_k+\frac{1}{2}\right)^2 - z_r^2}\right]\,\left[\frac{y_k^2-\left(z_r-\frac{1}{2}\right)^2}{y_j^2-\left(z_r-\frac{1}{2}\right)^2}\right].
\end{multline*}
Another re-arrangement of the factors in the last line gives
\begin{multline*}
	\Abs{\a_{e_k + e_j}(\bm)}^2 = \Abs{\frac{\left(y_j +\frac{1}{2}\right)^2 - \nu^2}{\left(y_k +\frac{1}{2}\right)^2 - \nu^2}}\,\left[\frac{y_k^2 - (y_j+1)^2}{y_k^2 - y_j^2}\right]\,\left[\frac{(y_k+1)^2 - (y_j+1)^2}{y_k^2 - (y_j+1)^2}\right] \\
		\times \left[\frac{\left(y_j+\frac{1}{2}\right)^2 - \left(x_j-\frac{1}{2}\right)^2}{\left(y_k+\frac{1}{2}\right)^2 - \left(x_j-\frac{1}{2}\right)^2}\right]\,\left[\frac{\left(y_j+\frac{1}{2}\right)^2 - z_j^2}{\left(y_k+\frac{1}{2}\right)^2 - z_j^2}\right]\\
		\times\prod_{r\neq j,k} \left[\frac{\left(y_j+\frac{1}{2}\right)^2 - \left(x_r-\frac{1}{2}\right)^2}{(y_j+1)^2-(x_r-1)^2}\right]\,\left[\frac{(y_k+1)^2-\left(x_r-1\right)^2}{\left(y_k+\frac{1}{2}\right)^2 - \left(x_r-\frac{1}{2}\right)^2}\right]\\
		\times\prod_{r\neq j,k}\left[\frac{\left(y_j+\frac{1}{2}\right)^2 - z_r^2}{y_j^2-\left(z_r-\frac{1}{2}\right)^2}\right]\,\left[\frac{y_k^2-\left(z_r-\frac{1}{2}\right)^2}{\left(y_k+\frac{1}{2}\right)^2 - z_r^2}\right],
\end{multline*}
and this simplifies to
\begin{multline*}
	\Abs{\a_{e_k + e_j}(\bm)}^2 = \Abs{\frac{\left(y_j +\frac{1}{2}\right)^2 - \nu^2}{\left(y_k +\frac{1}{2}\right)^2 - \nu^2}}\,\left[1+\frac{2}{y_k+ y_j}\right] \,\left[\frac{\left(y_j+\frac{1}{2}\right)^2 - \left(x_j-\frac{1}{2}\right)^2}{\left(y_k+\frac{1}{2}\right)^2 - \left(x_j-\frac{1}{2}\right)^2}\right]\,\left[\frac{\left(y_j+\frac{1}{2}\right)^2 - z_j^2}{\left(y_k+\frac{1}{2}\right)^2 - z_j^2}\right]\\
		\times\prod_{r\neq j,k} \left[1- \frac{1}{y_j-x_r+2}\right]\,\left[1+ \frac{1}{y_k-x_r+1}\right]\,\left[1+\frac{1}{y_j+z_r-\frac{1}{2}}\right]\,\left[1-\frac{1}{y_k+z_r+\frac{1}{2}}\right]
\end{multline*}
The factor
\[
\left[1+\frac{2}{y_k+ y_j}\right]\,\prod_{r\neq j,k} \left[1- \frac{1}{y_j-x_r+2}\right]\,\left[1+ \frac{1}{y_k-x_r+1}\right]\,\left[1+\frac{1}{y_j+z_r-\frac{1}{2}}\right]\,\left[1-\frac{1}{y_k+z_r+\frac{1}{2}}\right]
\]
is bounded by some uniform constant $C>0$, independent of $\bn$ and $\nu$, so we are left with
\[
	\Abs{\a_{e_k + e_j}(\bm)}^2 \leq C\, \Abs{\frac{\left(y_j +\frac{1}{2}\right)^2 - \nu^2}{\left(y_k +\frac{1}{2}\right)^2 - \nu^2}}\,\left[\frac{\left(y_j+\frac{1}{2}\right)^2 - \left(x_j-\frac{1}{2}\right)^2}{\left(y_k+\frac{1}{2}\right)^2 - \left(x_j-\frac{1}{2}\right)^2}\right]\,\left[\frac{\left(y_j+\frac{1}{2}\right)^2 - z_j^2}{\left(y_k+\frac{1}{2}\right)^2 - z_j^2}\right]
\]
and since $\left(y_j +\frac{1}{2}\right) \leq \left(x_j -\frac{1}{2}\right)$ and $\left(y_j +\frac{1}{2}\right) \leq z_j$, 
\[
	\Abs{\a_{e_k + e_j}(\bm)}^2 \leq C\,\Abs{\frac{\left(y_j +\frac{1}{2}\right)^2 - \nu^2}{\left(y_k +\frac{1}{2}\right)^2 - \nu^2}}\,\left[\frac{\left(x_j-\frac{1}{2}\right)^2}{\left(y_k+\frac{1}{2}\right)^2 - \left(x_j-\frac{1}{2}\right)^2}\right]\,\left[\frac{z_j^2}{\left(y_k+\frac{1}{2}\right)^2 - z_j^2}\right].
\]
Translating according to the definitions of $x,y,z$, this expression becomes
\begin{multline*}
	\Abs{\frac{\left(m_j +j-\frac{1}{2}\right)^2 - \nu^2}{\left(\ceil{\bm} +k-\frac{1}{2}\right)^2 - \nu^2}}\\
	\times\left[\frac{\left(\l_j+j-\frac{1}{2}\right)^2}{\left(\ceil{\bm}+k-\frac{1}{2}\right)^2 - \left(\l_j+j-\frac{1}{2}\right)^2}\right]\,\left[\frac{\left(n_j+j-\frac{1}{2}\right)^2}{\left(\ceil{\bm}+k-\frac{1}{2}\right)^2 - \left(n_j+j-\frac{1}{2}\right)^2}\right],
\end{multline*}
which is itself bounded by
\begin{multline}\label{inat}
	\Abs{\frac{\left(\min\{\ceil{\bn},\ceil{\l}\} +k-\frac{3}{2}\right)^2 - \nu^2}{\left(\ceil{\bm} +k-\frac{1}{2}\right)^2 - \nu^2}}\\
	\left[\frac{\left(\ceil{\l}+k-\frac{3}{2}\right)^2}{\left(\ceil{\bm}+k-\frac{1}{2}\right)^2 - \left(\ceil{\l}+k-\frac{3}{2}\right)^2}\right]\,\left[\frac{\left(\ceil{\bn}+k-\frac{3}{2}\right)^2}{\left(\ceil{\bm}+k-\frac{1}{2}\right)^2 - \left(\ceil{\bn}+k-\frac{3}{2}\right)^2}\right],
\end{multline}
and it is now clear that 
\[
	\Abs{\a_{e_k + e_j}(\bm)}^2\leq C_{\nu}\,\frac{\min\{\ceil{\bn},\ceil{\l}\}^2}{\left(\ceil{\bm} +1\right)^2}\cdot \frac{\ceil{\l}^2}{\left(\ceil{\bm}+1\right)^2 - \ceil{\l}^2}\cdot\frac{\ceil{\bn}^2}{\left(\ceil{\bm}+1\right)^2 - \ceil{\bn}^2},
\]
where $C_\nu >0$ is a constant only depending on $\nu$. This is what we were meant to show.

For the second inequality, we see from~\eqref{defA} that
\begin{multline*}
	\Abs{\a_{2e_k}(\bm)}^2 = \prod_{r=1}^{k-1} \left[1 - \frac{2(y_k +1)}{(y_k +\frac{3}{2})^2 - (x_r - \frac{1}{2})^2}\right] \left[1 - \frac{2(y_k+1)}{(y_k +\frac{3}{2})^2 - z_r^2}\right] \left[1 + \frac{4(y_k+1)}{y_k^2 - y_r^2}\right] \\
	\times \Abs{1 - \frac{2(y_k+1)}{(y_k +\frac{3}{2})^2 - \nu^2}}.
\end{multline*}
From the inequalities in Figure~\ref{figarray} we have that $y_r^2 \leq \min\{(x_r-1)^2,(z_r-\frac{1}{2})^2\}$.  Calling this minimum $w_r^2$ for now, we have
\begin{equation}\label{wind}
	\Abs{\a_{2e_k}(\bm)}^2 \leq \prod_{r=1}^{k-1} \left[1 - \frac{2(y_k +1)}{(y_k +\frac{3}{2})^2 - (w_r + \frac{1}{2})^2}\right]^2 \left[1 + \frac{4(y_k+1)}{y_k^2 - w_r^2}\right] \times\Abs{1 - \frac{2(y_k+1)}{(y_k +\frac{3}{2})^2 - \nu^2}}.
\end{equation}
Furthermore, since $(y_k +\frac{3}{2})^2 - (w_r + \frac{1}{2})^2 \leq y_k^2 - w_r^2$ for all $r =1,\dots,k-1$,
\begin{align*}
	\Abs{\a_{2e_k}(\bm)}^2 &\leq \prod_{r=1}^{k-1} \left[1 - \frac{2(y_k +1)}{y_k^2 - w_r^2}\right]^2 \left[1 + \frac{4(y_k+1)}{y_k^2 - w_r^2}\right] \times\Abs{1 - \frac{2(y_k+1)}{(y_k +\frac{3}{2})^2 - \nu^2}} \\
		&\leq \prod_{r=1}^{k-1} \left[1 - \frac{12(y_k +1)^2}{(y_k^2 - w_r^2)^2} + \frac{16(y_k +1)^3}{(y_k^2 - w_r^2)^3}\right] \times\Abs{1 - \frac{2(y_k+1)}{(y_k +\frac{3}{2})^2 - \nu^2}}.
\end{align*}
Note that $y_k \geq \max\{x_{k-1}, (z_{k-1}+\frac{1}{2})\} \geq w_r +1$ for all $r$.  Also, it is easy to check that
\[
\left[1 - \frac{12(y_k +1)^2}{(y_k^2 - w_r^2)^2} + \frac{16(y_k +1)^3}{(y_k^2 - w_r^2)^3}\right] \leq 1
\]
whenever $y_k \geq w_r +2$, so that~\eqref{wind} becomes
\[
	\Abs{\a_{2e_k}(\bm)}^2 \leq \left[1 - \frac{2(y_k +1)}{(y_k +\frac{3}{2})^2 - (w_{k-1} + \frac{1}{2})^2}\right]^2 \left[1 + \frac{4(y_k+1)}{y_k^2 - w_{k-1}^2}\right]\cdot\Abs{1 - \frac{2(y_k+1)}{(y_k +\frac{3}{2})^2 - \nu^2}},
\]
because we have bounded all other terms (where $y_k \geq w_r +2$) by $1$. We only need to check the case where $y_k = w_{k-1}+1$. Making this substitution in the expression
\[
\left[1 - \frac{2(y_k +1)}{(y_k +\frac{3}{2})^2 - (w_{k-1} + \frac{1}{2})^2}\right]^2 \left[1 + \frac{4(y_k+1)}{y_k^2 - w_{k-1}^2}\right]
\]
yields
\begin{align*}
\left[1 - \frac{w_{k-1}+2}{2w_{k-1}+3}\right]^2 \left[1 + \frac{4(w_{k-1}+2)}{2w_{k-1}+1}\right]
\end{align*}
which is bounded by $1$ as long as $w_{k-1}\geq 0$, which is always the case.  Therefore, we have
\[
	\Abs{\a_{2e_k}(\bm)}^2 \leq \Abs{1 - \frac{2(y_k+1)}{(y_k +\frac{3}{2})^2 - \nu^2}}.
\]
From here, by examining the possible values for $\nu$ in the principal, complementary, end-point, and discrete series representations, it is easy to see that $\Abs{\a_{2e_k}(\bm)}\leq 1$, as desired.
\end{proof}

\section{Computations for $B$'s, $C$'s, and $\b$'s}\label{appendixb}

This appendix contains calculations involving the coefficients $B_j^{\pm}$, $j=1,\dots,k$ and $C$, defined by~\eqref{defB}, and the coefficients $\b$, from Definition~\ref{defodd}.

%===============================================

\subsection{Bounds for $B_k^+$}\label{Bkbound}

\begin{lemma}\label{coeffsodd}
There is a positive constant $C_\nu$ depending only on $\nu$ and satisfying
\[
\Abs{B_k (\bm,\l)}^2 \geq C_{\nu}\,\left(\ceil{\bm}-\ceil{\bn}+1\right)\left(\ceil{\bm}-\ceil{\l}+1\right)
\]
for all $\bm\in M_\l$.  We also have that
\[
\Abs{B_k (\bm,0)}^2 \geq C_{\nu_0}\,\left((\ceil{\bm}+1)^2-\ceil{\bn}^2\right)
\]
holds whenever $\l\equiv 0$, where $C_{\nu_0}$ depends only on some choice of $\nu_0 \in \big[0,\frac{\en-1}{2}\big)$ that is closer to the right end-point of that interval than $\nu$ is. Finally, there is a $C>0$ such that 
\[
\Abs{B_k (\bm,0)}^2 \leq C\,\left(1 + \tilde\nu + 2\ceil{\bm}^2-\ceil{\bn}^2\right)
\]
for all $\ceil{\bm}\geq\ceil{\bn}$.
\end{lemma}

\begin{proof}
From~\eqref{defB} we have 
\begin{equation}\label{forsecondpart}
	\Abs{B_k^+ (\bm,\l)}^2 = \Abs{\nu^2 - y_k^2}\frac{\left(x_k^2 - y_k^2\right)\left(z_k^2 - y_k^2\right)}{y_k^2\left(4y_k^2-1\right)}\,\prod_{r\neq k}\left[\frac{x_r^2 - y_k^2}{y_r^2-y_k^2}\right]\,\left[\frac{z_r^2 - y_k^2}{(y_r-1)^2-y_k^2}\right].
\end{equation}
From the inequalities in Figure~\ref{figarray} we have that $y_r \geq x_r +1$ and $y_r \geq z_r +1$, so
\[
	\Abs{B_k^+ (\bm,\l)}^2 \geq \Abs{\nu^2 - y_k^2}\frac{\left(x_k^2 - y_k^2\right)\left(z_k^2 - y_k^2\right)}{y_k^2\left(4y_k^2-1\right)}\,\prod_{r\neq k}\left[\frac{x_r^2 - y_k^2}{\left(x_r+1\right)^2-y_k^2}\right]\,\left[\frac{z_r^2 - y_k^2}{z_r^2-y_k^2}\right].
\]
The last product term is greater than $1$, so we can immediately remove it, leaving
\[
	\Abs{B_k^+ (\bm,\l)}^2 \geq \Abs{\nu^2 - y_k^2}\frac{\left(y_k^2 - x_k^2\right)\left(y_k^2 - z_k^2 \right)}{y_k^2\left(4y_k^2-1\right)}.
\]
Translating $x,y,z$, we have arrived at 
\[
	\Abs{B_k^+ (\bm,\l)}^2 \geq \Abs{\nu^2 - \left(\ceil{\bm}+k\right)^2}\,\frac{\left(\ceil{\bm}+k\right)^2 - \left(\ceil{\l}+k-1\right)^2}{\left(\ceil{\bm}+k\right)^2}\frac{\left(\ceil{\bm}+k\right)^2 - \left(\ceil{\bn}+k-1\right)^2}{4\left(\ceil{\bm}+k\right)^2-1}.
\]
We now find ourselves in essentially the same situation as~\eqref{similarsituation}.  Proceeding as before, we find the desired result.

For the second part, notice that expression~\eqref{forsecondpart} becomes
\begin{multline*}
	\Abs{B_k^+ (\bm,0)}^2 = \Abs{\nu^2 - \left(\ceil{\bm}+k\right)^2} \\
		\times\frac{\left[(\ceil{\bm}+k)^2 - (k-1)^2\right]\left[(\ceil{\bm}+k)^2 - (\ceil{\bn} + k-1)^2\right]}{(\ceil{\bm}+k)^2\left[4(\ceil{\bm}+k)^2-1\right]} \\
		\times \prod_{r\neq k}\frac{\left[(r-1)^2 - (\ceil{\bm}+k)^2\right]\left[(r-1)^2 - (\ceil{\bm}+k)^2\right]}{\left[r^2-(\ceil{\bm}+k)^2\right][(r-1)^2-(\ceil{\bm}+k)^2]}.
\end{multline*}
As in the proof of Lemma~\ref{coeffseven} we see that the last term is bounded below by some constant $C>0$ (in fact, $1$) and above by some other constant $C' >0$.  We therefore write
\[
	\Abs{B_k^+ (\bm,0)}^2 \geq C\, \Abs{\nu^2 - \left(\ceil{\bm}+k\right)^2}\,\frac{\left[(\ceil{\bm}+k)^2 - (k-1)^2\right]\left[(\ceil{\bm}+k)^2 - (\ceil{\bn} + k-1)^2\right]}{(\ceil{\bm}+k)^2\left[4(\ceil{\bm}+k)^2-1\right]}.
\]
Again, we may proceed from this point as in the proof of Lemma~\ref{coeffseven} to conclude that for $\nu_0 \in [0, \frac{\en-1}{2})$, there is some $C_{\nu_0}>0$ depending only on $\nu_0$ and, of course, $k$, such that 
\[
	\Abs{B_k^+ (\bm,0)}^2 \geq C_{\nu_0}\,(\ceil{\bm}+\ceil{\bn}+1)(\ceil{\bm}-\ceil{\bn}+1)
\]
and furthermore that
\[
	\Abs{B_k^+ (\bm,0)}^2 \ll 1 + \tilde\nu +2\ceil{\bm}^2 - \ceil{\bn}^2
\]
which proves the lemma.
\end{proof}

\subsection{Bounds for $\b$'s}\label{bbound}

\begin{lemma}\label{betas}
There is a positive constant $C_{\nu}$ depending on $k$ and $\nu$ such that
\[
	\Abs{\b_{e_k \pm e_j}(\bm)}^2 \leq C_{\nu}\,\frac{\min\{\ceil{\bn},\ceil{\l}\}^2}{\left(\ceil{\bm}+1\right)^2}\cdot\frac{\ceil{\l}^2}{\left(\ceil{\bm} + 1\right)^2 - \ceil{\l}^2}\cdot\frac{\ceil{\bn}^2}{\left(\ceil{\bm} + 1\right)^2 -  \ceil{\bn}^2},
\]
and
\[
	\Abs{\b_{e_k}(\bm)}^2 \leq \frac{C_\nu}{\left(\ceil{\bm}+1\right)^2}\cdot\frac{\ceil{\l}^2}{\left(\ceil{\bm}+1\right)^2 -\ceil{\l}^2}\cdot \frac{\ceil{\bn}^2}{\left(\ceil{\bm}+1\right)^2 - \ceil{\bn}^2}.
\]
Also, we have $\abs{\b_{2e_k}(\bm)}\leq 1$ for all $\bm$.
\end{lemma}

\begin{proof}
From $\b_{e_k +e_j}(\bm) = B_j^+(\bm)/B_k^+(\bm+e_j)$ and~\eqref{defB} we find
\begin{multline*}
	\Abs{\b_{e_k + e_j}(\bm)}^2 = \Abs{\frac{y_j^2 - \nu^2}{y_k^2 - \nu^2}}\,\left[\frac{y_k^2 \left(4y_k^2-1\right)}{y_j^2 \left(4y_j^2-1\right)}\right] \left[\frac{\left(y_k-1\right)^2 - y_j^2}{y_k^2 - \left(y_j+1\right)^2}\right]\\
	\times\prod_{r=1}^{k}\left[\frac{x_r^2 - y_j^2}{x_r^2 - y_k^2}\right]\left[\frac{z_r^2 - y_j^2}{z_r^2 - y_k^2}\right]\,\prod_{r\neq j,k}\left[\frac{y_r^2 - y_k^2}{y_r^2 - y_j^2}\right]\left[\frac{\left(y_r+1\right)^2 - y_k^2}{\left(y_r+1\right)^2 - y_j^2}\right].
\end{multline*}
Since for $a>b$ the expression $\left(\frac{t^2 - a}{t^2 -b}\right)$ is increasing on $t>0$, and because $y_r\leq x_{r+1}$ and $y_r \leq z_{r+1}$, we can bound by
\begin{multline*}
	\Abs{\b_{e_k + e_j}(\bm)}^2 \leq \Abs{\frac{y_j^2 - \nu^2}{y_k^2 - \nu^2}}\,\left[\frac{y_k^2 \left(4y_k^2-1\right)}{y_j^2 \left(4y_j^2-1\right)}\right] \left[\frac{\left(y_k-1\right)^2 - y_j^2}{y_k^2 - \left(y_j+1\right)^2}\right]\\
	\times\prod_{r=1}^{k}\left[\frac{x_r^2 - y_j^2}{x_r^2 - y_k^2}\right]\left[\frac{z_r^2 - y_j^2}{z_r^2 - y_k^2}\right]\,\prod_{r\neq 1,j+1}\left[\frac{x_r^2 - y_k^2}{x_r^2 - y_j^2}\right]\left[\frac{\left(z_r+1\right)^2 - y_k^2}{\left(z_r+1\right)^2 - y_j^2}\right].
\end{multline*}
Re-arranging terms and making the necessary cancellations,
\begin{multline*}
	\Abs{\b_{e_k + e_j}(\bm)}^2 \leq \Abs{\frac{y_j^2 - \nu^2}{y_k^2 - \nu^2}}\,\left[\frac{y_k^2 \left(4y_k^2-1\right)}{y_j^2 \left(4y_j^2-1\right)}\right] \left[\frac{\left(y_k-1\right)^2 - y_j^2}{y_k^2 - \left(y_j+1\right)^2}\right] \\
	\left[\frac{x_1^2 - y_j^2}{x_1^2 - y_k^2}\right]\left[\frac{z_1^2 - y_j^2}{z_1^2 - y_k^2}\right] \left[\frac{x_{j+1}^2 - y_j^2}{x_{j+1}^2 - y_k^2}\right]\left[\frac{z_{j+1}^2 - y_j^2}{z_{j+1}^2 - y_k^2}\right]\\
	\times\prod_{r\neq 1,j+1}\left[\frac{z_r^2 - y_j^2}{\left(z_r+1\right)^2 - y_j^2}\right]\,\left[\frac{\left(z_r+1\right)^2 - y_k^2}{z_r^2 - y_k^2}\right].
\end{multline*}
The last line is bounded by some universal constant $C>0$, and this leaves
\begin{multline*}
	\Abs{\b_{e_k + e_j}(\bm)}^2 \leq C\,\Abs{\frac{y_j^2 - \nu^2}{y_k^2 - \nu^2}}\,\left[\frac{y_k^2 \left(4y_k^2-1\right)}{y_j^2 \left(4y_j^2-1\right)}\right] \left[\frac{\left(y_k-1\right)^2 - y_j^2}{y_k^2 - \left(y_j+1\right)^2}\right] \\
	\left[\frac{x_1^2 - y_j^2}{x_1^2 - y_k^2}\right]\left[\frac{z_1^2 - y_j^2}{z_1^2 - y_k^2}\right] \left[\frac{x_{j+1}^2 - y_j^2}{x_{j+1}^2 - y_k^2}\right]\left[\frac{z_{j+1}^2 - y_j^2}{z_{j+1}^2 - y_k^2}\right]
\end{multline*}
and since 
\[
	\left[\frac{x_1^2 - y_j^2}{x_1^2 - y_k^2}\right]\left[\frac{z_1^2 - y_j^2}{z_1^2 - y_k^2}\right]\leq \frac{y_j^4}{y_k^4}
\]
we can write
\[
	\Abs{\b_{e_k + e_j}(\bm)}^2 \leq C\,\Abs{\frac{y_j^2 - \nu^2}{y_k^2 - \nu^2}}\,\left[\frac{y_j^2 \left(4y_k^2-1\right)}{y_k^2 \left(4y_j^2-1\right)}\right] \left[\frac{\left(y_k-1\right)^2 - y_j^2}{y_k^2 - \left(y_j+1\right)^2}\right] \,\left[\frac{x_{j+1}^2 - y_j^2}{x_{j+1}^2 - y_k^2}\right]\left[\frac{z_{j+1}^2 - y_j^2}{z_{j+1}^2 - y_k^2}\right]
\]
The second and third factors now have some universal bound, so we can absorb them into $C$, leaving us, after translating $x,y,z$, with
\begin{multline*}
	\Abs{\b_{e_k + e_j}(\bm)}^2 \leq C\,\Abs{\frac{\left(m_j+j\right)^2 - \nu^2}{\left(\ceil{\bm}+k\right)^2 - \nu^2}}\\
	\times\left[\frac{\left(\l_{j+1}+j\right)^2 - \left(m_j+j\right)^2}{\left(\l_{j+1}+j\right)^2 - \left(\ceil{\bm}+k\right)^2}\right]\left[\frac{\left(n_{j+1}+j\right)^2 - \left(m_j+j\right)^2}{\left(n_{j+1}+j\right)^2 - \left(\ceil{\bm}+k\right)^2}\right]
\end{multline*}
which is bounded by
\begin{multline*}
	\Abs{\b_{e_k + e_j}(\bm)}^2 \leq C\,\Abs{\frac{\left(\min\{\ceil{\bn},\ceil{\l}\}+k-1\right)^2 - \nu^2}{\left(\ceil{\bm}+k\right)^2 - \nu^2}}\\
		\times\left[\frac{\left(\ceil{\l}+k-1\right)^2}{\left(\ceil{\l}+k-1\right)^2 - \left(\ceil{\bm}+k\right)^2}\right]\left[\frac{\left(\ceil{\bn}+k-1\right)^2}{\left(\ceil{\bn}+k-1\right)^2 - \left(\ceil{\bm}+k\right)^2}\right].
\end{multline*}
This is the situation we were in at~\eqref{inat} in the proof of Lemma~\ref{alphas}. The result follows in the same way.

Now, for $\b_{e_k}(\bm) = C(\bm)/B_k^+(\bm)$, we have
\[
	\Abs{\b_{e_k}(\bm)}^2 = \Abs{\frac{\nu^2}{\nu^2 - y_k^2}}\,\frac{y_k^2\left(4y_k^2 -1\right)}{y_k^2\left(y_k -1\right)^2}\,\prod_{r=1}^{k-1}\left[\frac{y_r^2 - y_k^2}{y_r^2}\right]\left[\frac{\left(y_r-1\right)^2 - y_k^2}{\left(y_r-1\right)^2}\right]\,\prod_{r=1}^{k}\left[\frac{x_r^2}{x_r^2 - y_k^2}\right]\left[\frac{z_r^2}{z_r^2 - y_k^2}\right].
\]
Since $y_r \geq x_r +1$ and $y_r \geq z_r +1$, we bound by
\[
	\Abs{\b_{e_k}(\bm)}^2 \leq \Abs{\frac{\nu^2}{\nu^2 - y_k^2}}\,\frac{y_k^2\left(4y_k^2 -1\right)}{y_k^2\left(y_k -1\right)^2}\,\prod_{r=1}^{k-1}\left[\frac{y_k^2 - \left(x_r + 1\right)^2}{\left(x_r + 1\right)^2}\right]\left[\frac{y_k^2 - z_r^2}{z_r^2}\right]\,\prod_{r=1}^{k}\left[\frac{x_r^2}{y_k^2 -x_r^2}\right]\left[\frac{z_r^2}{y_k^2 - z_r^2}\right],
\]
and making cancellations and re-writing,
\[
	\Abs{\b_{e_k}(\bm)}^2 \leq \Abs{\frac{\nu^2}{\nu^2 - y_k^2}}\,\frac{\left(4y_k^2 -1\right)}{\left(y_k -1\right)^2}\left[\frac{x_k^2}{y_k^2 -x_k^2}\right]\left[\frac{z_k^2}{y_k^2 - z_k^2}\right]\,\prod_{r=1}^{k-1}\left[\frac{y_k^2 - \left(x_r + 1\right)^2}{y_k^2 -x_r^2}\right] \left[\frac{x_r^2}{\left(x_r + 1\right)^2}\right].
\]
Re-arranging terms again,
\begin{multline*}
	\Abs{\b_{e_k}(\bm)}^2 \leq \Abs{\frac{\nu^2}{\nu^2 - y_k^2}}\,\left[\frac{x_k^2}{y_k^2 -x_k^2}\right]\left[\frac{z_k^2}{y_k^2 - z_k^2}\right] \\
	\times \frac{\left(4y_k^2 -1\right)}{\left(y_k -1\right)^2}\,\prod_{r=1}^{k-1}\left[1 - \frac{2x_r + 1}{y_k^2 -x_r^2}\right] \left[1-\frac{2x_r+1}{\left(x_r + 1\right)^2}\right],
\end{multline*}
we see that the last three factors are bounded by some $C>0$ independent of $\bn$ and $\nu$, so
\[
	\Abs{\b_{e_k}(\bm)}^2 \leq C\,\Abs{\frac{\nu^2}{\nu^2 - y_k^2}}\,\left[\frac{x_k^2}{y_k^2 -x_k^2}\right]\left[\frac{z_k^2}{y_k^2 - z_k^2}\right].
\]
Translating $x,y,z$, and taking into account that $C$'s only appear in the principal series, and therefore $\nu \in i\RR$, 
\begin{multline*}
	\Abs{\b_{e_k}(\bm)}^2 \leq \frac{C\,\nu^2}{\nu^2 - \left(\ceil{\bm}+k\right)^2}\\
	\times\left[\frac{\left(\ceil{\l}+k-1\right)^2}{\left(\ceil{\bm}+k\right)^2 -\left(\ceil{\l}+k-1\right)^2}\right]\left[\frac{\left(\ceil{\bn}+k-1\right)^2}{\left(\ceil{\bm}+k\right)^2 - \left(\ceil{\bn}+k-1\right)^2}\right].
\end{multline*}
From here we easily see that there is some constant $C_\nu$ such that
\[
	\Abs{\b_{e_k}(\bm)}^2 \leq \frac{C_\nu}{\left(\ceil{\bm}+1\right)^2}\cdot\frac{\ceil{\l}^2}{\left(\ceil{\bm}+1\right)^2 -\ceil{\l}^2}\cdot \frac{\ceil{\bn}^2}{\left(\ceil{\bm}+1\right)^2 - \ceil{\bn}^2},
\]
which establishes the second part of the lemma.

Finally, for $\b_{2e_k}(\bm) = B_k^+(\bm)/B_k^+(\bm+e_k)$, we have
\begin{multline*}
	\Abs{\b_{2e_k}(\bm)}^2 = \Abs{\frac{\nu^2 - y_k^2}{\nu^2 - \left(y_k+1\right)^2}}\,\frac{\left(y_k+1\right)^2\left[4\left(y_k+1\right)^2-1\right]}{y_k^2\left[4y_k^2-1\right]}\\
	\times\prod_{r=1}^{k}\left[\frac{x_r^2 -y_k^2}{x_r^2 -\left(y_k+1\right)^2}\right]\left[\frac{z_r^2 -y_k^2}{z_r^2 -\left(y_k+1\right)^2}\right]\,\prod_{r=1}^{k-1} \left[\frac{y_r^2 - \left(y_k +1\right)^2}{y_r^2 - y_k^2}\right]\left[\frac{\left(y_r-1\right)^2 - \left(y_k +1\right)^2}{\left(y_r-1\right)^2 - y_k^2}\right].
\end{multline*}
The inequalities in Figure~\ref{figarray} imply that $y_r \leq x_{r+1}$ and $y_r \leq z_{r+1}$, and this can be used to bound the last product above, leaving
\begin{multline*}
	\Abs{\b_{2e_k}(\bm)}^2 \leq \Abs{\frac{\nu^2 - y_k^2}{\nu^2 - \left(y_k+1\right)^2}}\,\frac{\left(y_k+1\right)^2\left[4\left(y_k+1\right)^2-1\right]}{y_k^2\left[4y_k^2-1\right]}\\
	\times\prod_{r=1}^{k}\left[\frac{x_r^2 -y_k^2}{x_r^2 -\left(y_k+1\right)^2}\right]\left[\frac{z_r^2 -y_k^2}{z_r^2 -\left(y_k+1\right)^2}\right]\,\prod_{r=2}^{k} \left[\frac{x_r^2 - \left(y_k +1\right)^2}{x_r^2 - y_k^2}\right]\left[\frac{\left(z_r-1\right)^2 - \left(y_k +1\right)^2}{\left(z_r-1\right)^2 - y_k^2}\right].
\end{multline*}
After making cancellations and re-arranging terms, we have
\begin{multline*}
	\Abs{\b_{2e_k}(\bm)}^2 \leq \Abs{\frac{\nu^2 - y_k^2}{\nu^2 - \left(y_k+1\right)^2}}\,\frac{\left(y_k+1\right)^2\left[4\left(y_k+1\right)^2-1\right]}{y_k^2\left[4y_k^2-1\right]}\,\left[\frac{x_1^2 -y_k^2}{x_1^2 -\left(y_k+1\right)^2}\right]\\
	\times\left[\frac{z_1^2 -y_k^2}{z_1^2 -\left(y_k+1\right)^2}\right]\,\prod_{r=2}^{k}\left[\frac{z_r^2 -y_k^2}{z_r^2 -\left(y_k+1\right)^2}\right]\,\left[\frac{\left(z_r-1\right)^2 - \left(y_k +1\right)^2}{\left(z_r-1\right)^2 - y_k^2}\right].
\end{multline*}
One can check that each factor in the product term is bounded by $1$. We will therefore remove all but the first one, bounding
\begin{align*}
	\Abs{\b_{2e_k}(\bm)}^2 &\leq \Abs{\frac{\nu^2 - y_k^2}{\nu^2 - \left(y_k+1\right)^2}}\,\left[\frac{\left(y_k+1\right)^2}{y_k^2}\right]\left[\frac{4\left(y_k+1\right)^2-1}{4y_k^2-1}\right]\left[\frac{y_k^2}{\left(y_k+1\right)^2}\right]^2 \\
		&\indent\times\left[\frac{z_2^2 -y_k^2}{z_2^2 -\left(y_k+1\right)^2}\right]\,\left[\frac{\left(z_2-1\right)^2 - \left(y_k +1\right)^2}{\left(z_2-1\right)^2 - y_k^2}\right] \\
		&= \Abs{\frac{\nu^2 - y_k^2}{\nu^2 - \left(y_k+1\right)^2}}\\
		&\indent\times\left[\frac{4\left(y_k+1\right)^2-1}{4y_k^2-1}\right]\,\left[\frac{y_k^2}{\left(y_k+1\right)^2}\right]\left[\frac{z_2^2 -y_k^2}{z_2^2 -\left(y_k+1\right)^2}\right]\,\left[\frac{\left(z_2-1\right)^2 - \left(y_k +1\right)^2}{\left(z_2-1\right)^2 - y_k^2}\right].
\end{align*}
The last line is bounded by $1$ on our domain $\{1\leq z_2 \leq y_k-1\}$, so we are finally left with 
\[
	\Abs{\b_{2e_k}(\bm)}^2 \leq \Abs{\frac{\nu^2 - y_k^2}{\nu^2 - \left(y_k+1\right)^2}} \leq 1
\]
as we had set out to prove.
\end{proof}

%===============================================

\subsection*{Acknowledgments}

It is a pleasure to thank Alexander Gorodnik for fruitful discussions and sound advice, Svetlana Katok for an informative discussion, and the referee for taking the time to carefully read the manuscript. The author is supported by ERC.

%===============================================

\bibliographystyle{amsalpha}
\bibliography{../bibliography}
\end{document}